\documentclass[review,onefignum,onetabnum]{siamart171218}
\nolinenumbers
\usepackage{geometry}
\usepackage{mathfont}
\usepackage{mathrsfs}
\usepackage{amssymb}
\usepackage[numbers,sort]{natbib}

\usepackage{lipsum}
\usepackage{amsfonts}
\usepackage{graphicx}
\usepackage{epstopdf}
\usepackage{algorithmic}
\usepackage{subcaption}
\usepackage{mathtools}
\ifpdf
  \DeclareGraphicsExtensions{.eps,.pdf,.png,.jpg}
\else
  \DeclareGraphicsExtensions{.eps}
\fi

\newsiamremark{remark}{Remark}
\newsiamremark{hypothesis}{Hypothesis}
\crefname{hypothesis}{Hypothesis}{Hypotheses}
\newsiamthm{claim}{Claim}

\headers{Parameterized Wasserstein Gradient Flow}{Yijie Jin, Shu Liu, Hao Wu, Xiaojing Ye, Haomin Zhou}

\title{Parameterized Wasserstein Gradient Flow%\thanks{Submitted to the editors DATE.
}

\author{
    Yijie Jin \thanks{School of Mathematics, Georgia Institute of Technology, Atlanta, GA, USA
    (\email{yijiejin@gatech.edu}).}
    \and 
    Shu Liu\thanks{Department of Mathematics, University of California, Los Angles, CA, USA(\email{shuliu@math.ucla.edu}).}    
    \and
    Hao Wu\thanks{School of Mathematics, Georgia Institute of Technology, Atlanta, GA, USA
    (\email{hwu406@gmail.com}).}
    \and 
    Xiaojing Ye\thanks{Department of Mathematics and Statistics, Georgia State University, Atlanta, GA, USA(\email{xye@gsu.edu}).}
    \and
    Haomin Zhou\thanks{School of Mathematics, Georgia Institute of Technology, Atlanta, GA, USA(\email{hmzhou@gatech.edu}).}
     }

\usepackage{amsopn}
\DeclareMathOperator{\diag}{diag}

\makeatletter
\newcommand*{\addFileDependency}[1]{% argument=file name and extension
  \typeout{(#1)}% latexmk will find this if $recorder=0 (however, in that case, it will ignore #1 if it is a .aux or .pdf file etc and it exists! if it doesn't exist, it will appear in the list of dependents regardless)
  \@addtofilelist{#1}% if you want it to appear in \listfiles, not really necessary and latexmk doesn't use this
  \IfFileExists{#1}{}{\typeout{No file #1.}}% latexmk will find this message if #1 doesn't exist (yet)
}
\makeatother

\newcommand*{\myexternaldocument}[1]{%
    \externaldocument{#1}%
    \addFileDependency{#1.tex}%
    \addFileDependency{#1.aux}%
}
\ifpdf
\hypersetup{
  pdftitle={Parameterized Wasserstein Hamiltonian Flow},
  pdfauthor={Yijie Jin, Shu Liu, Hao Wu, Xiaojing Ye, Haomin Zhou}
}
\fi

\myexternaldocument{ex_supplement}

\definecolor{glaucous}{rgb}{0.38, 0.51, 0.71}

\newtheorem{example}[theorem]{Example}
\newtheorem{assumption}{Assumption}

\begin{document}

\maketitle

% REQUIRED
\begin{abstract}
We develop a fast and scalable numerical approach to solve Wasserstein gradient flows (WGFs), which is particularly suitable for high-dimensional cases. 
%WGFs are gradient flows of smooth energy functionals defined on the Wasserstein space of probability densities over an $d$-dimensional manifold without boundary. They are a general class of geometric flows on Wasserstein spaces and include many important PDEs, such as Fokker-Planck equations and porous media equations, as special cases. However, solving WGFs for high dimensional problems, e.g., $d>5$, is considered very challenging in practice. 
Our approach is to use general reduced-order models, like deep neural networks, to parameterize the push-forward maps such that they can push a simple reference density to the one solving the given WGF. 
%This approach essentially reduces these flows defined on infinite-dimensional Wasserstein mainfold to finite-dimensional dynamical systems of the parameters of reduced-order model. 
The new dynamical system is called parameterized WGF (PWGF), and it is defined on the finite-dimensional parameter space equipped with a pullback Wasserstein metric. Our numerical scheme can approximate the solutions of WGFs for general energy functionals effectively, without 
requiring spatial discretization or nonconvex optimization procedures, thus avoiding some  limitations of classical numerical methods and more recent deep-learning based approaches. A comprehensive analysis of the approximation errors measured by Wasserstein distance is also provided in this work. Numerical experiments show promising computational efficiency and verified accuracy on a variety of WGF examples using our approach. 
\end{abstract}

% REQUIRED

% REQUIRED
%\begin{AMS}
%  68Q25, 68R10, 68U05
%\end{AMS}
\section{Introduction}
% The outline is not required, but we show an example here.
Wasserstein gradient flow (WGF) is a powerful tool for understanding and analyzing density evolution processes. In the seminal work \cite{JKO} by Jordan,
Kinderlehrer and Otto, they showed that Fokker-Planck equation (FPE) is essentially the gradient flow of the relative entropy functional under the Wasserstein metric. Since then, WGFs have shown extensive applications in optimal transport theory, optimization problems, Fokker-Planck equation, porous medium equation, and more \cite{otto-PME, chewi2020svgd, liu2022neural}.
However, numerical computation of general WGFs remains a challenging problem, especially when the state space is of high dimension. 
Furthermore, it is often desirable to find a sampler that generates samples following the solution of WGF rather than the actual density function solving WGF in many real-world statistics and machine learning applications.

In this paper, we focus on the numerical computation of WGF to address the aforementioned issues. Let $\mathcal{M}$ be a smooth manifold without boundary. For simplicity, we assume $\mathcal{M}=\Rbb^{d}$ throughout the present work, while generalization to general manifolds is straightforward. We also omit subscript $\mathcal{M}$ for all integrals unless otherwise noted. 
Denote the density function space defined on $\mathcal{M}$ as
\begin{equation}
\label{eq:PM}
    \mathcal{P}(\mathcal{M})=\Big\{\rho: \mathcal{M} \to \mathbb{R} \colon\int \rho(x)dx=1,~\rho(x)\geqslant 0,~\int |x|^2\rho(x)~ dx<\infty\Big\}.
\end{equation}
Suppose $\Pcal (\mathcal{M})$ is equipped with the Wasserstein-2 distance \cite{Lafferty,otto-PME}:
\begin{equation}
\label{eq:W2}
    W(\rho_1, \rho_2) = \Big(\inf_{\pi\in \Pi (\rho_1, \rho_2)}\iint |x-y|^2d\pi(x, y) \Big)^{1/2}
\end{equation}
for any $\rho_1,\rho_2 \in \Pcal(\mathcal{M})$, where $\Pi (\rho_1, \rho_2)$ is the joint density space with $\rho_1$ and $\rho_2$ as marginals.
Then $\mathcal{P}(\mathcal{M})$ becomes an infinite-dimensional Riemannian manifold with $W$ inducing its Riemannian metric, and a WGF can be written in the following general form:
\begin{align}
\label{equ: WGF}
    \frac{\partial \rho}{\partial t}= - \textrm{grad}_{W}\mathcal{F}(\rho), \quad \rho(0, x)=\rho_0(x),
\end{align}
where $x\in \mathcal{M}$, $\rho_0$ is a given initial probability density, $\mathcal{F}:\Pcal(\mathcal{M}) \rightarrow \mathbb{R}$ is some energy functional defined on $\Pcal(\mathcal{M})$, and $\textrm{grad}_{W}$ stands for the gradient of functionals on $\mathcal{P}(\mathcal{M})$ with respect to the Wasserstein metric. 
%
% With specific choices of $\mathcal{F}$, WGF \eqref{equ: WGF} reduces to Fokker-Planck equation, porous medium equation, and some other types of well-known systems. 
%
Many well-known probability evolution equations, such as the Fokker-Planck equation and porous medium equation, are essentially WGF \eqref{equ: WGF} with specific choices of $\mathcal{F}$.

In recent years, there has been a surge of interest to compute WGF \eqref{equ: WGF} numerically. In \cite{JKO}, a proximal point algorithm is used to solve the Fokker-Planck equation as a special case of WGF. Later, various numerical methods have been developed to solve the WGF; see \cite{carrillo2022primal, peyre2015entropic, mokrov2021large, fan2022variational, nurbekyan2023efficient}. Among them, some use classical methods such as finite difference \cite{leveque2007finite} and finite element method \cite{brenner2008mathematical}. They are applicable to solving WGF with low dimension $d$. This limitation is due to their spatial discretization and hence they suffer the issue known as ``curse-of-dimensionality'' as the number of unknowns increases exponentially fast in terms of problem dimension $d$. Sampling-based approaches \cite{mokrov2021large, fan2022variational} can be used to solve WGF in high dimension, however, the computation for the minimization and energy functional evaluation used in those algorithms may become time-consuming. 

The goal of this work is to develop a fast computational scheme of \eqref{equ: WGF} in high-dimensional settings. 
%
%We follow the strategy used in \cite{liu2022neural} for solving FPE. 
%
Similar to \cite{liu2022neural, mokrov2021large, fan2022variational}, we seek for a time-dependent push-forward map to generate samples whose density function solves \eqref{equ: WGF}. 
However, compared to these existing methods, our approach features several notable differences: (i) We use a \textit{new pullback Wasserstein metric} that can be computed efficiently while still maintaining the accuracy up to the same order as the results obtained in \cite{liu2022neural}; (ii) Our algorithm can be used to compute general WGF with a large class of (possibly nonlinear) $\Fcal$; (iii) We parameterize the push-forward map using neural ordinary differential equation \cite{chen2018neural} and show that it is computationally convenient to handle functional of probability densities in our case. As a result, we obtain an approximate solution to the WGF as well as a sample generator of the solution distribution; and (iv) Our numerical scheme does not need any spatial discretization and thus is scalable for problems defined on high-dimensional spaces (e.g., $d \geqslant 15$). Moreover, our method does not need to solve any nonconvex optimization problem as in typical deep neural network training. Instead, it only requires fast matrix-vector products where the vector size is determined by the number of parameters in the selected push-forward map. 

\section{Related work}
The Fokker-Planck equation is first interpreted as a special case of WGF in \cite{JKO}. It is shown that the Fokker-Planck equation can be viewed as the gradient flow on the Wasserstein manifold of a special energy functional, which consists of a linear potential energy and a relative entropy. 
%The authors of \cite{} give a general description of Wasserstein Hamiltonian flow on density manifold,  establish several equivalent formulations and reveal its connections with many well-known PDEs. 
For general (possibly nonlinear) energy functionals, one can apply similar derivations to obtain their corresponding WGFs.
As WGFs are essentially time-evolution partial differential equations (PDEs), classical numerical methods such as finite difference and finite element methods can be applied with proper modifications \cite{peyre2015entropic, benamou2016discretization, carlier2017convergence, li2020fisher, carrillo2022primal}. However, the application of these classical methods is limited to low-dimensional cases due to the curse of dimensionality.

In recent years, sampling-based approaches have been proposed as a promising alternative solution method \cite{fan2022variational, liu2022neural,mokrov2021large, hu2022energetic, lee2023deep}. Many of these methods start from the numerical scheme used in \cite{JKO} known as the Jordan-Kinderlehrer-Otto (JKO) scheme, which is a proximal algorithm in the optimization context. By utilizing the Benamou-Brenier formula, we know that for any time step, the optimal push-forward map in the JKO scheme can be expressed as the gradient of some convex function. Motivated by this fact, several recent works are devoted to solving the WGF by approximating the optimal push-forward map by neural networks $T_{\theta}$. However, there are two main challenges in this approach: one is the evaluation of potential energy $\mathcal{F}$, especially when $\mathcal{F}$ involves the density function $\rho$ explicitly. As mentioned in \cite{Li_2019}, this problem can be addressed by introducing an extra optimization procedure. 
% \liu{In \cite{liu2022neural}, we do not introduce extra optimization to evaluate the relative entropy. Do you want to mean \cite{Li_2019}? See section 4, we use a variational formulation to compute the entropy.} 
However, such a computation is generally expensive; the other is the computational challenge for solving the JKO scheme, which is an optimization problem on the density manifold.

In contrast to those existing methods, our method is mostly motivated by \cite{liu2022neural}, in which the authors develop the parameterization to solve FPE. This approach exploits WGFs in the parameter space, in which the parametric functions, such as deep neural networks, are used to parameterize the push-forward maps. We design a numerical scheme for the parameter dynamics. The main challenge in our formulation is the computation of the Wasserstein gradient on the parameter manifold. We show that it can be effectively approximated once introducing the new pullback Wasserstein metric as discussed in \eqref{relaxed metric tensor}. Furthermore, we employ several powerful neural networks, such as the normalizing flow \cite{kobyzev2020normalizing} and continuous normalizing flow \cite{chen2018neural}, which are both convenient and efficient in evaluating the energy functionals involving push-forward densities and their gradients.

In addition to the aforementioned literatures, it is worth mentioning that the study on the Wasserstein information matrix as the metric tensor defined on the parameter manifold has been introduced in \cite{li2023wasserstein}; Some numerical analysis results on $1$-dimensional neural projected WGFs have been reported in \cite{zuo2024numerical}; Studies on WGFs on the manifold of Gaussian distributions have been conducted in \cite{chewi2020svgd, yi2023bridging} and the references therein. 

Furthermore, a series of deep-learning methods \cite{anderson2022evolution, bruna2024neural, du2021evolutional, gaby2024neural} have been composed to compute general time-evolution PDEs. Compared with these references which directly approximate the PDE solutions via neural networks, our algorithm is developed on the space of push-forward maps and utilizes the geometric principle of the WGFs. Our treatment respects the essential properties of WGFs such as positivity, conservation of mass, and energy dissipation: the first two properties are naturally preserved on push-forward models; the energy dissipation property is justified via our numerical experiments demonstrated in Section \ref{sec: numerical method}.

\section{Parameterization of Wasserstein gradient flow}
%In this section, we first introduce the background of Wasserstein metric briefly, then derive the parameterization of Wasserstein Hamiltonian flow and suggest the parameterization of Wasserstein Hamiltonian flow with relaxed metric.

In this section, we briefly review the derivation of WGF and several examples of WGF including Fokker-Planck and porous media equations, as well as the pullback Wasserstein metric in parameter space. Then we present our approach to solve WGF numerically and provide a comprehensive analysis of the proposed scheme. 

\subsection{Background on Wasserstein gradient flow}
\label{Background WGF}
%Firstly we introduce the Wasserstein metric and Wasserstein Hamiltonian flow. Please see \cite{chow2020wasserstein}  for more details.%\label{sec:main}
We denote the tangent space of $\Pcal(\mathcal{M})$ at $\rho$ by
\begin{align*}
    \mathcal{T}_\rho\mathcal{P}(\mathcal{M})=\Big\{\sigma\in C^{\infty}(\mathcal{M}) \colon \int \sigma(x)\, dx=0\Big\}.
\end{align*}
For any specific $\rho\in\mathcal{P}(\mathcal{M})$, the Wasserstein metric tensor $g^W$ is a positive definite bilinear form on the tangent bundle $\mathcal{T}\mathcal{P}(\mathcal{M}) = \{(\rho,\sigma)\colon \rho\in \mathcal{P}(\mathcal{M}),~\sigma\in \mathcal{T}_\rho\mathcal{P}(\mathcal{M})\}$ defined by:
\begin{equation}
g^W(\rho)(\sigma_1,\sigma_2)=\int  \nabla\Phi_1(x)\cdot\nabla\Phi_2(x)\rho(x) ~dx,\quad \forall \sigma_i\in \mathcal{T}_\rho\mathcal{P}(M), \ i=1,2\nonumber\label{def_metric}
\end{equation}
where $\sigma_i=-\nabla\cdot(\rho\nabla\Phi_i)$ for $i=1,2$.
Suppose $\mathcal{F}\colon \mathcal{P}(\mathcal{M})\rightarrow \mathbb{R}$ is a smooth functional defined on $\Pcal(\mathcal{M})$, its Riemannian gradient over $(\mathcal{P}(\mathcal{M}),g^W)$ is given as follows:
\begin{equation}
\textrm{grad}_W\mathcal{F}(\rho)={g^{W}(\rho)}^{-1}\left(\frac{\delta\mathcal{F}}{\delta\rho}\right)(x)\\
  =-\nabla\cdot \Big(\rho(x) \nabla \frac{\delta \Fcal}{\delta\rho}(x)\Big),
   \label{gradflow}
\end{equation}
where $\frac{\delta \Fcal}{\delta\rho}$ is the first variation of $\mathcal{F}$ in the $L^2$ sense. The Wasserstein gradient flow (WGF) of $\mathcal{F}$ is given by
\begin{align}
\label{def: WGF}
    \frac{\partial \rho}{\partial t} = -\textrm{grad}_W\mathcal{F}(\rho).
\end{align}

%\ye{It looks like better to order examples below as Fokker-Planck (eventually linear), Porous medium equation (nonlinear), then Aggregation (nonlinear and interactive). }

% {\bf Particle-level dynamics of WGF} 
%

Many well-known equations can be formulated as WGF with selected energy functionals. Here are three examples.

\begin{example}[Fokker-Planck equation]
\label{ex:FPE}
Let $V\in \mathcal{C}^2(M)$ be a given potential function and $\rho_*(x)=\frac{1}{Z_D}e^{-V(x)/D}$ be the corresponding Gibbs distribution, where $D>0$ is a fixed constant and $Z_D = \int e^{-\frac{V(x)}{D}} \,dx$ is the normalization constant.
Suppose $\mathcal{F}(\rho)$ is the relative entropy with respect to $\rho_{*}$ scaled by $D$, i.e.,
\begin{align}
\label{def: entropy H}
    \mathcal{F}(\rho) = D \, \mathcal{H}_{KL}(\rho\|\rho_*) \quad \mbox{where} \quad \mathcal{H}_{KL}(\rho\|\rho_*) := \left(\int \frac{1}{D} V(x)\rho(x) + \rho(x)\log \rho(x) \,dx \right) + \log Z_D .
\end{align}
Then the WGF \eqref{def: WGF} with $\Fcal$ defined in \eqref{def: entropy H} becomes the Fokker-Planck equation:
\begin{align}
\label{def: FPE}
\frac{\partial \rho}{\partial t} =-\mathrm{grad}_W\mathcal{F}(\rho)=\nabla \cdot (\rho \nabla V)+D\nabla\cdot (\rho\nabla \log \rho) = \nabla \cdot (\rho \nabla V)+D\Delta \rho.
\end{align}
\end{example}
The relative entropy $\mathcal{H}_{KL}(\rho\|\rho_*)$ is closely related to the Fisher information defined as
    \begin{align}
        \mathcal{I}(\rho|\rho_*)=\int \left\|\nabla \log \left(\frac{\rho(x)}{\rho_*(x)}\right)\right\|^2\rho(x) \,dx.
    \end{align}
It provides an upper bound to $\mathcal{H}_{KL}$ and further guarantees the uniform convergence of the dynamics \eqref{def: FPE} as shown in the following theorem, which will be used in our error estimation later. 
\begin{theorem}[Holley--Stroock perturbation \cite{Holley1987LogarithmicSI}]
\label{theorem: fi bound}
    Suppose the potential $V$ can be decomposed as $V=U+\phi$ where $\nabla ^2U\succeq KI$ for some $K>0$ and $\phi \in L^{\infty}$. Denote $osc(\phi):=\sup \phi -\inf \phi<\infty$. 
    Then the following \emph{logarithmic Sobolev inequality} holds
    \begin{align}
        \mathcal{H}_{KL}(\rho\|\rho_*)\leqslant \frac{e^{osc(\phi)}}{K} \mathcal{I}(\rho|\rho_*)
    \end{align}
    for any $\rho \in \mathcal{P}(M)$. Assume $\rho$ solves equation \eqref{def: FPE} with initial value $\rho(0, \cdot)=\rho_0(x)$, then 
    \begin{align}
        \mathcal{H}_{KL}(\rho_t\|\rho_*)\leqslant \mathcal{H}_{KL}(\rho_0\|\rho_*)e^{-\frac{DK}{osc(\phi)}t}, \quad \forall\, t>0.
    \end{align}
\end{theorem}

\begin{example}[Porous medium equation]
Consider the energy functional $\Fcal$ defined by
\begin{align}
    \mathcal{F}(\rho)=\frac{1}{m-1}\int \rho^m(x) \, dx.
\end{align}
for an integer $m>1$. 
The corresponding WGF becomes the porous medium equation (PME), which is a nonlinear heat equation given by
\begin{equation}
\label{PME}
\frac{\partial \rho}{\partial t} = \Delta (\rho^m).
\end{equation}
\end{example}
There are many applications of PME in physical problems, for example, heat transfer or diffusion, gas flow \cite{vazquez2007porous}, and flow of reactive fluids in porous media \cite{ladd2021reactive}.

\begin{example}[Aggregation model]
Consider the interaction energy functional 
\begin{equation}
\label{eq:interact-F}
\mathcal{F}(\rho) = \frac{1}{2}\iint J(|x-y|)\rho(x)\rho(y)\,dx\,dy
\end{equation}
which is used to model the interactive behavior of swarm of particles \cite{fetecau2011swarm}. Here $J$ is an interaction kernel consisting of repulsive and attractive parts. A typical choice for $J$ is
\begin{align}
    J(x)=\frac{|x|^a}{a}-\frac{|x|^b}{b}, \quad x\in \mathcal{M} .
\end{align}
where $a>b>0$ are positive constants. Its corresponding WGF is
\begin{equation}
\frac{\partial \rho}{\partial t} = \nabla \cdot ( \rho \nabla (J * \rho)).
\end{equation}    
\end{example}
We note that Keller-Siegel system is a special case of aggregation model \cite{blanchet2013parabolic}. 
We remark that the WGF \eqref{def: WGF} can be associated with a particle-level dynamic, as given in the following proposition. 
% We would also like to mention the particle level dynamics for the WGF \eqref{def: WGF}, which is the foundation for our error estimation:
\begin{proposition}
\cite{liu2022neural} Assume $\boldsymbol{X} \in \mathcal{M}$ is a random process and solves the following equation,
\begin{align}
    \label{def: particle dyn wgf}
    \dot{ \boldsymbol{X}}=-\nabla_X\frac{\delta}{\delta\rho}\mathcal{F}(\rho(\boldsymbol{X}), \boldsymbol{X})
\end{align}    
where $\rho$ is the density of $\boldsymbol{X}$. Then $\rho$ solves the WGF \eqref{def: WGF}.
\end{proposition}
The particle-level dynamic \eqref{def: particle dyn wgf} provides a physical interpretation of WGF. More importantly, we will use it to establish the error bound of our approximation using parameterized WGF below.

\subsection{Push-forward maps and parameterized Fokker-Planck equation}
In \cite{liu2022neural}, the authors developed a method to parameterize push-forward maps in order to approximately solve the FPE \eqref{def: FPE}. 
% More precisely, let $\lambda$ be a reference distribution, such as the standard Gaussian distribution $\mathcal{N}(0, I)$, and $\{T_{\theta}: \theta \in \Theta\}$ be a family of parametric functions mapping $\mathcal{M}$ to $\mathcal{M}$. Then for any $\theta\in \Theta$, we have a corresponding push-forward distribution $T_{\theta\sharp}\lambda$ defined by
% \begin{align}
%     T_{\theta\sharp}\lambda (E)=\lambda(T_{\theta}^{-1}(E))\ \textrm{for all measurable } E \subset \mathcal{M}.
% \end{align}
% This relation defines a parameterized family of distributions:
% \begin{align}
%     \mathcal{P}_{\Theta} :=\{\rho_{\theta}~:~\rho_{\theta}=T_{\theta\sharp}\lambda, \ \theta\in\Theta\}.
% \end{align}

Fix any reference probability distribution $\lambda$ that is absolutely continuous with respect to the standard Lebesgue measure $\mu$ on $\mathcal{M}$, we use $\varrho=d\lambda /d \mu$, the Radon-Nikodym derivative of $\lambda$ with respect to $\mu$, to denote the reference density determined by $\lambda$. Then a push-forward map $T: \mathbb{R}^d\rightarrow\mathbb{R}^d$ induces a new probability density $T_{\sharp}\varrho$ on $\mathbb{R}^{d}$:
\begin{align*}
    \int_{E} T_{\sharp}\varrho (x) \,d\mu(x) =\int_{T^{-1}(E)} \varrho(z) \,d\mu(z) = \int_{T^{-1}(E)} d\lambda(z) = \lambda (T^{-1}(E))\quad \textrm{ for all measurable } E\subset \mathbb{R}^d,
\end{align*}
where $T^{-1}(E)$ is the pre-image of $E$. Hereafter we use $dx$ instead of $d\mu(x)$ to reduce notation complexity.

Let us take $T$ as parameterized map, namely for any $\theta\in \Theta$, $T_{\theta}: \mathbb{R}^d\rightarrow\mathbb{R}^d$ is a parametric function with parameter $\theta$. 
Here $\Theta$, as a subset of $\mathbb{R}^n$, is called the \emph{parameter space}, where $n$ is the number of parameters of $T_{\theta}$ (i.e., the dimension of $\theta$). Typical examples of $T_{\theta}$ include Fourier expansion, finite element approximation, and (deep) neural networks. 

The map $T_{\cdot\sharp}:\Theta\rightarrow\mathcal{P}$ given by $\theta\mapsto T_{\theta\sharp}\varrho$ naturally defines an immersion from $\Theta$ to the probability manifold $\mathcal{P}$. Collecting all parameterized distributions together, i.e.,
\begin{equation*}
  \mathcal{P}_\Theta = \left\{\rho_\theta = T_{\theta\sharp}\varrho~:~\theta\in\Theta\right\}, %\quad \lambda ~ \textrm{is the reference measure,}
\end{equation*}
we obtain a finite-dimensional submanifold $\mathcal{P}_\Theta$ of $\mathcal{P}$. We can define the tangent space of $\mathcal{P}_\Theta$ at each $\theta=(\theta_1,\dots,\theta_n) \in \Theta \subset \mathbb{R}^{n}$ as
$T_{\rho_\theta}\mathcal{P}_\Theta = \mathrm{span}\{\frac{\partial\rho_\theta}{\partial\theta_1},\cdots, \frac{\partial\rho_\theta}{\partial\theta_n}\}.$ The tangent bundle is then
$\mathcal{T}\mathcal{P}_\Theta = \cup_{\theta\in\Theta}\{\rho_\theta\}\times T_{\rho_\theta}\mathcal{P}_\Theta.$ On the other hand, the cotangent space $T_{\rho_\theta}^*\mathcal{P}_\Theta$ is the dual space of $T_{\rho_\theta}\mathcal{P}_\Theta$, and the cotangent bundle is $\mathcal{T}^*\mathcal{P}_\Theta = \cup_{\theta\in\Theta}\{\rho_\theta\}\times T_{\rho_\theta}^*\mathcal{P}_\Theta.$

% \Yijie{Change all integrals with measure as density}
% \Yijie{Double check if there is redundant notation. }
% \ye{Safe to use $d\lambda$ in integrals from now on. Just use $T_{\theta\sharp}\varrho$ to refer pushed forward density.}

Note that $\mathcal{P}_{\Theta}$ is a finite-dimensional subset of $\mathcal{P}$, hence we can pull back the Wasserstein metric $g^W$ to $\mathcal{P}_{\Theta}$. It is shown in \cite{liu2022neural} that the pullback Wasserstein metric $G(\theta)=T_{\theta\sharp}^*g^W$ on $\mathcal{P}$ is given by
\begin{align}
    \label{def: G matrix}
    G(\theta)=\int \nabla \Psi \circ T_{\theta}(z)\nabla\Psi \circ T_{\theta}(z)^{\top}d\lambda (z).
\end{align}
where $\Psi_{\theta}=(\psi_1, \cdots, \psi_n)^{\top}$, and $\psi_{j}$ is solved from the following elliptic equation,
\begin{align}
    \label{def: psi pde}
    \nabla \cdot (\rho_{\theta}\nabla \psi_j(x))=\nabla\cdot \Big(\rho_{\theta}\frac{\partial T_{\theta}}{\partial \theta_j}\circ T^{-1}_{\theta}(x) \Big).
\end{align}
The parameterized version of FPE on parameter manifold $\Theta$ can be expressed as %\ye{we should use $G^+$ instead of $G^{-1}$ throughout for generalizability and consistency now.}
\begin{align}
\label{eq:PFPE}
    \dot\theta = -G(\theta)^{-1}\nabla_{\theta} F(\theta),
\end{align}
where $F(\theta):=\mathcal{F}(\rho_{\theta})$ and $\Fcal$ is the scaled relative entropy defined in \eqref{def: entropy H}.

If directly computing the parameterized FPE \eqref{eq:PFPE}, the cost is intractable in higher dimensions because it requires solving $n$ elliptic equations \eqref{def: psi pde} to obtain $G(\theta)$. To alleviate this difficulty, authors in \cite{liu2022neural} derived a minimax formulation that can advance the dynamics without solving the elliptic equations. Although the cost becomes manageable, it can still be high if the spatial dimension $d$ is not small.  In this work, we advocate a new, yet naturally derived pullback metric $\widehat{G}$ to replace $G$, and develop a computational framework to solve general WGF. Details shall be provided in the next subsection.

%It is much easier to evaluate $\widehat{G}$ related terms than $G$, while the convergence rate for the two dynamics \eqref{}, \eqref{}  are of the same order, as will be shown in section \ref{}.

\subsection{Parameterized WGF with a new pullback Wasserstein metric} 
Following the study in \cite{wu2023parameterized}, we replace $G(\theta)$ by a new pullback Wasserstein metric defined as 
\begin{align}
        \widehat{G}(\theta)=\int  \partial_{\theta} T_\theta(z)^{\top} \partial_{\theta} T_\theta(z)~d\lambda(z).\label{relaxed metric tensor}
\end{align}

For any energy functional $\mathcal{F}$ with smooth variation $\frac{\delta \mathcal{F}}{\delta\rho}$, %we denote its restriction on the parameter space as $F(\theta):=\mathcal{F}(\rho_{\theta})$. 
the corresponding parameterized WGF is given by 
\begin{align}
\label{def: relaxed pwgf}
        \dot\theta = -\widehat{G}(\theta)^{\dagger}\nabla_{\theta} F(\theta),
\end{align}
where $\widehat{G}(\theta)^{\dagger}$ is the Penrose-Moore pseudo inverse of $\widehat{G}(\theta)$.

The motivation of introducing $\widehat{G}(\theta)$ is three-folded:
\begin{itemize}
    \item The definition is directly inspired by \eqref{def: G matrix} and \eqref{def: psi pde}. In the $1$-dimensional case, the new metric $\widehat{G}$ coincides with the exact matrix $G$, and the proof is shown in \cite{liu2022neural}.
     \item A deeper motivation is influenced by \cite{otto-PME}. The geometry of $(\Theta, \widehat{G}(\theta))$ is an isometric embedding into the flat Riemannian space $(\mathcal{O}, g^{\mathcal{O}})$ introduced in \cite{otto-PME}. Specifically, let $\varrho \in P(\mathcal{M})$ be a fixed density function and $\mathcal{O}$ the set of diffeomorphisms $T: \mathcal{M} \rightarrow \mathcal{M}$, and define the Riemannian metric $g^{\mathcal{O}}$ on $\mathcal{O}$ at $T$ by $g^{\mathcal{O}}(T)(v_1,v_2):= \int v_1 v_2  \varrho \, dx$ for any $v_1,v_2 \in \mathcal{T}_T\mathcal{O}$. Here $\mathcal{T}_T\mathcal{O}$ is the tangent space of $\mathcal{O}$ at $T$, which is set to be the space of all vector fields on $\mathcal{M}$ for every $T$ in \cite{otto-PME}. Then it is shown that the mapping $\Pi: \theta \mapsto \rho_{\theta}$ yields an isometric submersion from $(\mathcal{O},g^{\mathcal{O}})$ to $(P(\mathcal{M}), g)$, where $g$ is the 2-Wasserstein metric. Consider a (deep) neural network structure $T_{\theta}$ with parameter $\theta$ such that $T_{\theta}:\mathcal{M} \to \mathcal{M}$ is a diffeomorphism for every $\theta \in \Theta$ and $T_{\cdot}(\cdot) \in C^1(\Theta \times \mathcal{M};\mathbb{R})$. 
    % \textcolor{blue}{(Do we need $T_{\cdot}(\cdot) \in C^1(\Theta \times \mathcal{M};\mathbb{R})$ to ensure since we implicitly used DCT and Fubini in some of our derivations? If yes it's OK since this is easy to satisfy for the NNs we used).} 
    Consider the map $\mathscr{T}: \Theta \to \mathcal{O}$ defined by $\theta \mapsto \mathscr{T}(\theta):= T_{\theta}$, we define the pullback metric on $\Theta$ as 
    \begin{equation}
        \widehat{G}(\theta) = \mathscr{T}^* g^{\mathcal{O}}(T_{\theta})
    \end{equation}
    where $\mathscr{T}^*$ is the pullback operation induced by $\mathscr{T}$. 
    To obtain $\widehat{G}(\theta)$, consider any curve $\{\theta(t)\}_{-\epsilon \leqslant t \leqslant \epsilon}$ for some $\epsilon > 0$ passing through $\theta(0)$ at $t=0$ and denote $\dot{\theta}(0) = \frac{d}{dt} \theta(t) |_{t=0}$. By the definitions of the pullback operation and the metric $g^{\mathcal{O}}$ above, we have
    \begin{align*}
        \widehat{G}(\theta(0))(\dot \theta(0), \dot \theta(0)) & = g^{\mathcal{O}}(T_{\theta}) \left( \frac{d}{dt} T_{\theta(t)}|_{t=0}, \frac{d}{dt} T_{\theta(t)}|_{t=0} \right) \\
        & = \int \left( \dot{\theta}(0)^{\top}\partial_{\theta} T_{\theta(0)}(z)^{\top} \partial_{\theta} T_{\theta(0)}(z) \dot{\theta}(0) \right) \varrho(z) \, dz\\
        & = \dot \theta(0)^{\top} \left( \int \partial_{\theta} T_{\theta(0)}(z)^{\top} \partial_{\theta} T_{\theta(0)}(z) \, d \lambda(z) \right) \dot \theta(0),
    \end{align*}
    which implies \eqref{relaxed metric tensor} since $\theta(0)$ and $\dot{\theta}(0)$ are arbitrary. The matrix $\widehat{G}(\theta) = ( \widehat{G}(\theta)_{ij} )_{1 \leqslant i,j \leqslant m}$ has components 
    \begin{equation}
        \widehat{G}(\theta)_{ij} = \sum_{k=1}^d \int \partial_{\theta_i} T_{\theta}^{(k)}(z) \cdot \partial_{\theta_j} T_{\theta}^{(k)}(z) \, d \lambda(z)
    \end{equation}
     where $T_{\theta}^{(k)}: \mathcal{M} \rightarrow \mathbb{R}$ is the $k$-th component of $T_{\theta}: \mathcal{M} \rightarrow \mathcal{M}$. This is also discussed in \cite{wu_theory_2023}.

    \item Compared with \eqref{def: G matrix}, the new metric is more computationally efficient. Directly evaluating the metric tensor $G$ requires solving $n$ different elliptic PDEs where $n$ is the number of the parameters in the push-forward map $T_{\theta}$, and $m$ can be large if we choose $T_{\theta}$ to be a deep neural network. To circumvent the computation complexity challenge, a bi-level minimization scheme was proposed in \cite{liu2022neural}. However, it may still be computationally expensive to solve such optimization problems in general. We demonstrate the difference in the computation time in the experiment for computing the solution of the Fokker-Planck equation in Section \ref{sec:FPE}.

    % \liu{We may mention that using $G(\theta)$ instead of $\widehat{G}(\theta)$ still possesses some advantages in the following two aspects. 1. In error estimation (the $\widehat{\delta}_0\geqslant \delta_0$; 2. The trajectories of the Lagrangian coordinates (i.e. trajectories of $\{T_{\theta_t}(z)\}_{t\geqslant 0}$) of WGF using $G(\theta)$ is guaranteed to be a gradient flow in $\mathbb{R}^d$; But that of $\widehat{G}(\theta)$ may not. Can we do experiments on this? For example just compute two gradient flows w.r.t. linear functional $\mathcal F(\rho) = \int V\rho ~dx.$}
\end{itemize}

The gradient $\nabla_{\theta} F(\theta)$ can be evaluated through the following formulation as derived in \cite{wu2023parameterized},
\begin{align}
\label{eq: grad para f}
\nabla_{\theta}F(\theta)=\int \partial_{\theta}T_{\theta}(z)^\top \nabla_X \frac{\delta}{\delta\rho}\mathcal{F}( T_{\theta\sharp}\varrho(\cdot), \cdot)\circ T_{\theta}(z)~d\lambda(z).
\end{align}
%The general parameterized WGF (PWGF) is given in the following theorem.
%
%\begin{theorem}
%\label{theorem: PWGF}
%    Given the energy functional $\mathcal{F}$, the Wasserstein gradient flow of $\mathcal{F}$ on probability submanifold $\mathcal{P}_{\Theta}$ is:
%    \begin{align}
%    \label{eq:PWGF}
%        \dot\theta = -G(\theta)^{+}\nabla_{\theta} F(\theta),
%    \end{align}
%    where $G(\theta)^{+}$ is the Penrose-Moore pseudo inverse of $G(\theta)$.
%\end{theorem}

%We want to highlight that the inverse still not known yet remain as an open question how to define the flow if Gis nogt inver. ?However in our situation once we replace g by ghat 
%
%The computational cost of \eqref{eq:PWGF} is very high due to the complexity in estimating $G(\theta)$.
%To overcome this issue, we replace the metric $G(\theta)$ with the relaxed pullback metric $\widehat{G}(\theta)$ defined by
%\begin{align}
%        \widehat{G}(\theta)=\int  \partial_{\theta} T_\theta(z)^{\top} \partial_{\theta} T_\theta(z)~d\lambda(z).\label{relaxed metric tensor}
%\end{align}

Following \cite{wu2023parameterized}, the formula \eqref{eq: grad para f} suggests that $\nabla_{\theta} F(\theta)$ is in the range of $\widehat{G}(\theta)$ even when $\widehat{G}$ is not invertible. Therefore, the right hand side of \eqref{def: relaxed pwgf} is always well defined. 
%
% In the following, we denote $\widehat{G}^{\dagger}$ as the Moore-Penrose inverse of $\widehat{G}$ and consider the dynamics in parameter space:
%
For convenience, we still call \eqref{def: relaxed pwgf} the parameterized WGF, or again PWGF for short. 
%Theoretical analysis on the convergence of dynamics \eqref{def: relaxed pwgf} shows that $\widehat{G}$ is a reasonable substitution with guaranteed approximation to the true solution. 

It is interesting to consider the particle level dynamics corresponding to \eqref{def: relaxed pwgf}. For a fix $z_0\in \mathcal{M}$, $\boldsymbol{Y}=T_{\theta}(z_0)$ is the push-forward point in $\mathcal{M}$. When 
$\theta$ varies as a function of $t$ according to \eqref{def: relaxed pwgf}, $\boldsymbol{Y}(t)$ forms a curve in $\mathcal{M}$, and it satisfies 
\begin{align}
\label{def: particle pwgf}
    \dot {\boldsymbol{Y}}=\partial_{\theta}T_{\theta}(z_0) \dot\theta=-\partial_{\theta}T_{\theta}(z_0)\widehat{G}(\theta)^{\dagger}\nabla_\theta F(\theta), \quad \boldsymbol{Y}_0=T_{\theta(0)}(z_0).
\end{align}
%where $z_0=T_{\theta}^{-1}(Y)$.

To further investigate the properties of parameterized particle dynamics \eqref{def: particle pwgf}, we introduce the kernel operator $\mathcal{K}$: for any $f\in L^2(\mathcal{M};\mathcal{M}, \lambda)$, $\mathcal{K}_{\theta}[f]\in L^2(\mathcal{M}; \mathcal{M}, \lambda)$ is defined as
\begin{align}
  \label{kernel op}
           \mathcal{K}_\theta[f](\cdot) = \partial_{\theta} T_\theta(\cdot)\widehat{G}(\theta)^{\dagger} \int \partial_{\theta} T_\theta(z)^\top f(z)~d\lambda(z)
           =\int K_\theta(\cdot , z) f(z)~d\lambda(z),
\end{align}
where the kernel matrix $K_\theta(z', z) \in \mathbb{R}^{d\times d}$ is defined by
\begin{equation}
    K_\theta(z', z) := \partial_{\theta} T_\theta(z') \widehat{G}(\theta)^{\dagger}  \partial_{\theta} T_\theta(z)^\top.
\end{equation}
%
% As proved in \cite{}, if the functional $\mathcal{F}$ admits an integral form:
% \begin{align}
%     \label{def: integral form f}
%     \mathcal{F}(\rho)=\int J( \rho(x), x)\rho(x)~dx,
% \end{align}
% for some smooth function $J(\cdot, \cdot): \mathbb{R}^{1+d}\rightarrow \mathbb{R}^d$, then we can write $\nabla_\theta \mathbb{F}(\theta)$ as:
% \begin{align}
%         \nabla_{\theta}\mathbb{F}(\theta)=\int \partial_{\theta}T_{\theta}(z)^\top \nabla_X \frac{\delta}{\delta\rho}\mathcal{F}( T_{\theta\sharp}\lambda(\cdot), \cdot)\circ T_{\theta}(z)~d\lambda(z),
% \end{align}
% and we can conclude $\nabla_{\theta}\mathbb{F}(\theta)\in Range(\widehat{G}(\theta))$ for any $\theta\in \Theta$. 
It is shown that $\mathcal{K}_{\theta}[f]$ is the orthogonal projection of $f$ onto the tangent space spanned by $\partial_{\theta}T_{\theta}$ \cite{wu2023parameterized}. Combining \eqref{eq: grad para  f}, \eqref{def: particle pwgf} and \eqref{kernel op}, we obtain
\begin{align}
    \label{eq: kernel form particle dyn}
    \dot{\boldsymbol{Y}}&=-\partial_{\theta}T_{\theta}(z_0)\widehat{G}(\theta)^{\dagger}\int \partial_{\theta}T_{\theta}(z)^\top \nabla_X \frac{\delta}{\delta\rho}\mathcal{F}( T_{\theta\sharp}\varrho(\cdot), \cdot)\circ T_{\theta}(z)~d\lambda(z)\nonumber \\
    &=-\mathcal{K}_{\theta}[\nabla_X \frac{\delta}{\delta\rho}\mathcal{F}( T_{\theta\sharp}\varrho(\cdot), \cdot)\circ T_{\theta}](z_0)\\
    &=-\mathcal{K}_{\theta}[\nabla_X \frac{\delta}{\delta\rho}\mathcal{F}( T_{\theta\sharp}\varrho(\cdot), \cdot)\circ T_{\theta}](T_{\theta}^{-1}(\boldsymbol{Y})). \nonumber
\end{align}
The particle-level dynamics \eqref{eq: kernel form particle dyn} is determined by the time-evolving push-forward map $T_{\theta}$. We will use it %\eqref{eq: kernel form particle dyn} 
to derive an upper bound on the error of PWGF \eqref{def: relaxed pwgf} in the next section.

\subsection{Error bounds in Wasserstein metric}
We provide an error analysis of the proposed PWGF formulation. Specifically, we establish an upper bound on the difference, measured by Wasserstein distance, between the approximation $\rho_{\theta(t)}(\cdot) = T_{\theta(t)\sharp}\varrho (\cdot)$ where $\theta(t)$ solves PWGF and the true solution $\rho(t,\cdot)$ to the original WGF. We consider the error analysis in two scenarios: the first variation $\frac{\delta \Fcal}{\delta \rho}$ is Lipschitz continuous and the case of Fokker-Planck equation where $\frac{\delta \Fcal}{\delta \rho}$ is not Lipschitz continuous.

\subsubsection{Error analysis with Lipschitz continuity assumption} In this subsection, we assume the following condition on $\mathcal{F}$ holds.
%We first establish error estimates for dynamics \eqref{def: relaxed pwgf} under the assumption that the first variation $\frac{\delta \Fcal}{\delta \rho}$ is Lipschitz continuous, as defined below.
\begin{assumption}
\label{assumption: lip constant of delta F}
There exists a constant $C_{\mathcal{F}}>0$ such that for any two push-forward maps $T$ and $ \widetilde{T}$ there is
\begin{align}
    \int \left|\nabla_X \frac{\delta}{\delta\rho}\mathcal{F}( T_{\sharp}\varrho(\cdot), \cdot)\circ T(z)-\nabla_X \frac{\delta}{\delta\rho}\mathcal{F}( \widetilde{T}_{\sharp} \varrho(\cdot), \cdot)\circ \widetilde{T}(z) \right|^2~d\lambda(z)\leqslant C_{\mathcal{F}}\int|T(z)-\widetilde{T}(z)|^2~d\lambda(z).
\end{align}
\end{assumption}
The accuracy of our approach depends on the representation power of the parameterized push-forward map $T_{\theta}$. We use $\delta_0$ to characterize the representation error as below. 
\begin{definition}
    Define the projection error as
    \begin{align}
    \delta_0 & = \sup_{\theta\in \Theta}\min_{\xi\in\mathcal{T}_{\theta}^*\Theta} \left\{ \int \left|\nabla \frac{\delta}{\delta\rho}\mathcal{F}( T_{\theta\sharp}\varrho(\cdot), \cdot) \circ T_{\theta}(z) - \partial_{\theta}T_{\theta}(z)\,\xi \right|^2~d\lambda(z) \right\} \nonumber \\
    & = \sup_{\theta\in\Theta} \left\{ \int \left|\nabla \frac{\delta}{\delta\rho}\mathcal{F}( T_{\theta\sharp}\varrho(\cdot), \cdot) \circ T_{\theta}(z) - \mathcal{K}_\theta[\nabla \frac{\delta}{\delta\rho}\mathcal{F}( T_{\theta\sharp}\varrho(\cdot), \cdot)\circ T_{\theta} ](z) \right|^2~d\lambda(z) \right\}.\label{pseudo delta 0 def}
\end{align}
\end{definition}
This error is essentially the difference between $\frac{\delta \Fcal (\rho_{\theta})}{\delta \rho}$ and its orthogonal projection onto the tangent space $T_{\rho_{\theta}} \Pcal(\mathcal{M})$. 
% \liu{Could we compute the value of $\delta_0$ along the solution curve $\{\theta_t\}$ of the PWGF? This will make the readers feel confident about the accuracy of our numerical method} 
In the definition, we use the fact that the operator $\mathcal{K}_{\theta}$ is the orthogonal projection. Under Assumption \ref{assumption: lip constant of delta F}, we can show that the density $\rho_{\theta}$ approximates the true solution with guaranteed Wasserstein-$2$ error stated in the next theorem.
\begin{theorem}
    Suppose Assumption \ref{assumption: lip constant of delta F} holds for $\mathcal{F}$, $\theta$ is solved from \eqref{def: relaxed pwgf} with initial value $\theta(0)$, and $\rho$ is solved from \eqref{equ: WGF}. Then the Wasserstein-$2$ distance between the push-forward density $\rho_{\theta(t)}(\cdot)$ obtained by PWGF and the true density $\rho(t,\cdot)$ satisfies
    \begin{align}
    \label{eq:W-bound}
        W_2^2(\rho_{\theta(t)}(\cdot), \rho(t,\cdot))\leqslant e^{(1+2C_{\mathcal{F}})t}\epsilon_{\rho(0)}+\frac{2\delta_0}{1+2C_{\mathcal{F}}}\left(e^{(1+2C_{\mathcal{F}})t}-1\right).
    \end{align}
    where $\epsilon_{\rho(0)}=W_2^2(\rho_{\theta(0)}, \rho_0)$ is the initial approximation error.
\end{theorem}
\begin{proof}
Assume $\boldsymbol{X}$ is solved from \eqref{def: particle dyn wgf} and $\boldsymbol{Y}$ is solved from \eqref{def: particle pwgf}. Suppose the Monge map from $\rho_{\theta(0)}$ to $\rho_0$ is given by $\omega$ and we assume the random variables $\boldsymbol{X}, \boldsymbol{Y}$ are coupled via $\boldsymbol{X}(0) = \omega(\boldsymbol{Y}(0))$. %are coupled through
We continue from the calculation in \eqref{eq: kernel form particle dyn},%first decompose the dynamics for $\boldsymbol{Y}^{\Theta}$:
\begin{align}
    \dot{\boldsymbol{Y}}% &=\partial_{\theta}T_{\theta}(z_0)\dot{\theta}=-\partial_{\theta}T_{\theta}(z_0)\widehat{G}^{\dagger}\nabla_{\theta}\mathbb{F}(\theta)\\
    % &=-\partial_{\theta}T_{\theta}(z_0)\widehat{G}^{\dagger}\int \-\partial_{\theta}T_{\theta}(z)^\top \frac{\delta}{\delta \rho}\mathcal{F}\circ T_{\theta}(z)~d\lambda(z)\\
    & = -\mathcal{K}_{\theta}\left[\frac{\delta}{\delta \rho}\mathcal{F}( T_{\theta\sharp}\varrho(\cdot), \cdot)\circ T_{\theta}\right](z) \nonumber \\
    & = -\frac{\delta}{\delta \rho}\mathcal{F}( T_{\theta\sharp}\varrho(\cdot), \cdot)\circ T_{\theta}(z) + \left(\frac{\delta}{\delta \rho}\mathcal{F}( T_{\theta\sharp}\varrho(\cdot), \cdot)\circ T_{\theta}(z)-\mathcal{K}_{\theta} \left[\frac{\delta}{\delta \rho}\mathcal{F}( T_{\theta\sharp}\varrho(\cdot), \cdot)\circ T_{\theta} \right](z)\right). \nonumber
\end{align}
Denote $E(t)=\mathbb{E}\|\boldsymbol{X}-\boldsymbol{Y}\|^2$, we compute,
\begin{align}
\label{ineq: dE est 0}
    \frac{d}{dt}E(t)&=2\mathbb{E}\left[(\boldsymbol{X}-\boldsymbol{Y})\cdot (\dot{\boldsymbol{X}}-\dot{\boldsymbol{Y}})\right]
    \leqslant 2\sqrt{\mathbb{E}\|\boldsymbol{X}-\boldsymbol{Y}\|^2} \sqrt{\mathbb{E}\|\dot{\boldsymbol{X}}-\dot{\boldsymbol{Y}}\|^2}
    \leqslant \mathbb{E}\|\boldsymbol{X}-\boldsymbol{Y}\|^2+ \mathbb{E}\|\dot{\boldsymbol{X}}-\dot{\boldsymbol{Y}}\|^2.
\end{align}
Notice that $\boldsymbol{X}$ is a push-forward of $\boldsymbol{X}(0)$ through the dynamics \eqref{def: particle dyn wgf}, we have
\begin{align}
\label{ineq: dE est 1}
    \mathbb{E}\|\dot{\boldsymbol{X}}-\dot{\boldsymbol{Y}}\|^2&=\int \Big|-\left(\frac{\delta}{\delta \rho}\mathcal{F}( \rho(x), x)-\frac{\delta}{\delta \rho}\mathcal{F}( T_{\theta\sharp}\varrho(\cdot), \cdot)\circ T_{\theta}(z)\right)\nonumber\\
    &\qquad\qquad -\left(\frac{\delta}{\delta \rho}\mathcal{F}( T_{\theta\sharp}\varrho(\cdot), \cdot)\circ T_{\theta}(z)-\mathcal{K}_{\theta}\left[\frac{\delta}{\delta \rho}\mathcal{F}( T_{\theta\sharp}\varrho(\cdot), \cdot)\circ T_{\theta}\right](z)\right)\Big|^2~d\lambda(z)\nonumber\\
    %&\leqslant 2\int \left|\frac{\delta}{\delta \rho}\mathcal{F}\circ T(z)-\frac{\delta}{\delta \rho}\mathcal{F}\circ T_{\theta}(z)\right|^2~d\lambda(z)+2\int \left|\frac{\delta}{\delta \rho}\mathcal{F}( T_{\theta\sharp}\lambda(\cdot), \cdot)\circ T_{\theta}(z)-\mathcal{K}_{\theta}[\frac{\delta}{\delta \rho}\mathcal{F}( T_{\theta\sharp}\lambda(\cdot), \cdot)\circ T_{\theta}](z)\right|^2~d\lambda(z)\nonumber\\
    &\leqslant 2\delta_0 + 2 C_{\mathcal{F}} \, \mathbb{E}\|\boldsymbol{X}-\boldsymbol{Y}\|^2.
\end{align}
Plugging \eqref{ineq: dE est 1} into \eqref{ineq: dE est 0}, we obtain
\begin{align}
    \frac{d}{dt}E(t)\leqslant 2\delta_0+(1+2C_{\mathcal{F}})E(t).
\end{align}
By Gr\"onwall's inequality, we have
\begin{align}
    \label{ineq: E estimation}
    E(t)\leqslant e^{(1+2C_{\mathcal{F}})t}E(0)+\frac{2\delta_0}{1+2C_{\mathcal{F}}}\left(e^{(1+2C_{\mathcal{F}})t}-1\right).
\end{align}
On the other hand, we have 
\begin{align}
\label{eq:W-bound-E}
      W_2^2(T_{\theta\sharp}\varrho, \rho)=W_2^2(\textrm{Law}(\boldsymbol{Y}), \textrm{Law}(\boldsymbol{X})) \leqslant \mathbb{E}\|\boldsymbol{Y}-\boldsymbol{X}\|^2 =  E(t) .
\end{align}
Combining \eqref{ineq: E estimation} and \eqref{eq:W-bound-E} yields \eqref{eq:W-bound}.
\end{proof}

\subsubsection{Asymptotic Analysis for Parametrized Wasserstein Gradient Flow}
For many examples, finding $C_{\mathcal{F}}$ could be a challenging task. However, if there exists a Gibbs solution $\rho_*$ to the WGF satisfying Polyak-Łojasiewicz inequality, then we can show uniform convergence of $\rho_{\theta}$ to the target density $\rho_*$, as stated in the following theorem:

% \liu{We shall reorganize this part as follows.

% Assume that the functional $\mathcal{F}(\rho)=0$ when $\rho$ equals to the Gibbs distribution $\rho_*$. If the functional $F(\rho)$ satisfies certain Polyak-Łojasiewicz inequality as follows
% \begin{equation*}
%   \frac{1}{\lambda}  \int_{\mathbb{R}^d} |\nabla_X \frac{\delta\mathcal{F}}{\delta\rho}(\rho, x)|^2\rho(x)~dx \geqslant  \mathcal{F}(\rho).
% \end{equation*}

% Then we can always show that 
% \begin{equation}
%     \mathcal{F}(\rho_{\theta(t)}) \leqslant \frac{\delta_0}{\Tilde{\lambda}_D D^2}(1-e^{-{\lambda}t})+\mathcal{F}(\rho_{\theta(0)})e^{-{\lambda}t}.
% \end{equation}
% }

% As we have shown in Example \ref{ex:FPE}, Fokker-Planck equation (FPE) is a special case of WGF by choosing $\Fcal$ as in \eqref{def: entropy H}.

% %
% Unfortunately, such a potential $\Fcal$ does not satisfy Assumption \ref{assumption: lip constant of delta F}. However, the existence of Lyapunov function for entropy function \eqref{def: entropy H} enables us to achieve uniform convergence on PWGF, as shown in the following theorem.
%
\begin{theorem}
\label{theorem: kl estimation}
Assume that $\rho_*$ % $\mathcal{F}(\rho_*)=0$ where $\rho_*$ 
is the Gibbs distribution, and the functional $\mathcal{F}(\rho)$ satisfies certain Polyak-Łojasiewicz inequality as follows
\begin{equation}
\label{assumption: PL ineq}
  \frac{1}{\zeta}  \int_{\mathbb{R}^d} \left|\nabla_X \frac{\delta\mathcal{F}}{\delta\rho}(\rho, x) \right|^2\rho(x)~dx \geqslant  \mathcal{F}(\rho) - \mathcal F(\rho_*),
\end{equation}
where $\zeta$ is a positive constant.
Then 
\begin{equation}
\label{eq: f est pl condition}
    \mathcal{F}(\rho_{\theta(t)}) - \mathcal F(\rho_*) \leqslant \frac{\delta_0}{\zeta}(1-e^{-{\zeta}t})+\mathcal{F}(\rho_{\theta(0)})e^{-{\zeta}t}.
\end{equation}
where $\delta_0$ is defined in \eqref{pseudo delta 0 def}.

% Suppose $\theta(t)$ is solved from \eqref{def: relaxed pwgf} with functional
% $\mathcal{F}$ defined in \eqref{def: entropy H} and $\theta(0)=\theta_0$. Let $\rho$ be the solution to \eqref{def: FPE}, then there is
%     \begin{align}
%         \label{ineq: KL estimation}
%         \mathcal{D}_{KL}(\rho_{\theta}(t)\|\rho_*)\leqslant \frac{\delta_0}{\Tilde{\lambda}_D D^2}(1-e^{-D\Tilde{\lambda}_Dt})+\mathcal{D}_{KL}(\rho_{\theta(0)}\|\rho_*)e^{-D\Tilde{\lambda}_Dt}.
%     \end{align}
% where $\Tilde{\lambda}_D=\frac{K}{e^{osc(\phi)}}$.
\end{theorem}
% %The difference between Theorem \ref{theorem: kl estimation} and \cite{} is definition of $\Tilde{\lambda}_D$, 
% % We leave the proof in the appendix.
% %
\begin{proof}
%[Proof of Theorem \ref{theorem: kl estimation}]
Denote $f(z)=\nabla_X \frac{\delta}{\delta\rho}\mathcal{F}( T_{\theta\sharp}\varrho(\cdot), \cdot)\circ T_{\theta}(z)$. Then by the property of orthogonal projection operator $\mathcal{K}_{\theta}$ we have $\int \left(I-\mathcal{K}_{\theta}\right)[f](z)^\top \mathcal{K}_{\theta}[f](z)~d\lambda(z)=0$. Furthermore,
\begin{align}
    \int  f(z)^\top \mathcal{K}_{\theta}[f](z)~d\lambda(z)&=\int  \mathcal{K}_{\theta}[f](z)^\top \mathcal{K}_{\theta}[f](z)~d\lambda(z) + \int  (I-\mathcal{K}_{\theta})[f](z)^\top \mathcal{K}_{\theta}[f](z)~d\lambda(z) \label{eq: fpe proof eq 1}\\
    &=\int  \mathcal{K}_{\theta}[f](z)^\top \mathcal{K}_{\theta}[f](z)~d\lambda(z)\nonumber\\
    &=\int  \lvert f(z)\rvert^2~d\lambda(z)
    -\int  \lvert \left(I-\mathcal{K}_{\theta}\right)[f](z)\rvert^2~d\lambda(z).\nonumber
\end{align}
where the last equality is due to the Pythagoras theorem.
Following the definition of $\delta_0$ in \eqref{pseudo delta 0 def}, we have
\begin{align}
    \label{eq: fpe proof eq 2}
    \int  \lvert \left(I-\mathcal{K}_{\theta}\right)[f](z)\rvert^2~d\lambda(z)\leqslant \delta_0.
\end{align}
By inequality \eqref{assumption: PL ineq}, we have
\begin{align}
\label{eq: fpe proof eq 3}
\int  \lvert f(z)\rvert^2~d\lambda(z)\geqslant \zeta \mathcal{F}(\rho).
\end{align}
We deduce that
\begin{align}
\label{eq: fpe proof eq 4}
    \frac{d}{dt}\mathcal{F}(\rho_{\theta})
    &=\nabla_{\theta}F(\theta)\cdot \dot{\theta} \nonumber \\
    & =-\nabla_{\theta}F(\theta)^\top \widehat{G}(\theta)^{\dagger}\nabla_{\theta}F(\theta) \nonumber \\
    &=-\int  \nabla_X \frac{\delta}{\delta\rho}\mathcal{F}( T_{\theta\sharp}\varrho(\cdot), \cdot)\circ T_{\theta}(z)^\top \partial_{\theta}T_{\theta}(z)\widehat{G}(\theta)^{\dagger }\nabla_{\theta}F(\theta)~d\lambda(z) \nonumber\\
    &=-\int  \nabla_X \frac{\delta}{\delta\rho}\mathcal{F}( T_{\theta\sharp}\lambda(\cdot), \cdot)\circ T_{\theta}(z)^\top \mathcal{K}_{\theta}[\nabla_X \frac{\delta}{\delta\rho}\mathcal{F}( T_{\theta\sharp}\lambda(\cdot), \cdot)\circ T_{\theta}](z)~d\lambda(z) \nonumber  \\
    &=-\int  f(z)^\top \mathcal{K}_{\theta}[f](z)~d\lambda(z),%\nonumber 
    %&=-\frac{1}{D}\int|\mathcal{K}_{\theta}[f](z)|^2~d\lambda(z) \\
    % &\leqslant \frac{\delta_0}{D} - \frac{1}{D}\int \lvert f(z)\rvert^2~d\lambda(z)\\
    % &=\frac{\delta_0}{D}-D\mathcal{I}(\rho_{\theta} \|\rho_*)\\
    % &\leqslant \frac{\delta_0}{D}-\frac{DK}{e^{osc(\phi)}}\mathcal{D}_{KL}(\rho_{\theta}|\rho_*),
    %&=-\frac{1}{D}\int \left|\mathcal{K}_{\theta}[\nabla_X \frac{\delta}{\delta\rho}\mathcal{F}\circ T_{\theta}](z)\right|^2~d\lambda(z).
\end{align}
where the first equality is due to $F(\theta):= \Fcal(\rho_{\theta})$, the second equality is obtained by the formulation of PWGF in \eqref{def: relaxed pwgf}, the third equality uses the substitution of first $\nabla_{\theta} F(\theta)$ using \eqref{eq: grad para f}, the fourth equality is because of the definition of $\Kcal_{\theta}$ in \eqref{kernel op}, and the fifth equality is by the definition of $f$ at the beginning of this proof. 

We also observe that
\begin{align}
\label{eq:bound-fKf}
    -\int  f(z)^\top \mathcal{K}_{\theta}[f](z)~d\lambda(z)
    %&=-\frac{1}{D}\int|\mathcal{K}_{\theta}[f](z)|^2~d\lambda(z) \\
    \leqslant \delta_0 - \int \lvert f(z)\rvert^2~d\lambda(z)
    \leqslant \delta_0-\zeta\mathcal{F}(\rho_{\theta})
    ,%\leqslant \frac{\delta_0}{D}-\frac{DK}{e^{osc(\phi)}}\mathcal{D}_{KL}(\rho_{\theta}|\rho_*),
    %&=-\frac{1}{D}\int \left|\mathcal{K}_{\theta}[\nabla_X \frac{\delta}{\delta\rho}\mathcal{F}\circ T_{\theta}](z)\right|^2~d\lambda(z).
\end{align}
where the first inequality is obtained by combining \eqref{eq: fpe proof eq 1} and \eqref{eq: fpe proof eq 2}, the second equality is due to \eqref{eq: fpe proof eq 3}. Combining \eqref{eq: fpe proof eq 4} and \eqref{eq:bound-fKf} yields 
\begin{equation}
\label{eq:KLineq}
    \frac{d}{dt}F(\theta) \leqslant \delta_0-\zeta F(\theta).
\end{equation}
% \begin{align}
%     \int \left|\nabla_X \frac{\delta}{\delta\rho}\mathcal{F}\circ T_{\theta}(z)\right|^2~d\lambda(z)&=\int \left|\mathcal{K}_{\theta}[\nabla_X \frac{\delta}{\delta\rho}\mathcal{F}\circ T_{\theta}](z)\right|^2~d\lambda(z) + \int \left|(I-\mathcal{K}_{\theta})[\nabla_X \frac{\delta}{\delta\rho}\mathcal{F}\circ T_{\theta}](z)\right|^2~d\lambda(z)\\
%     &\leqslant \delta_0 + \int \left|\mathcal{K}_{\theta}[\nabla_X \frac{\delta}{\delta\rho}\mathcal{F}\circ T_{\theta}](z)\right|^2~d\lambda(z)
% \end{align}
% So we have:
% \begin{align}
%     \frac{d}{dt}\mathcal{D}_{KL}(\rho_{\theta}\|\rho_*)&\leqslant \frac{\delta_0}{D}-\frac{1}{D}\int \left|\nabla_X \frac{\delta}{\delta\rho}\mathcal{F}\circ T_{\theta}(z)\right|^2~d\lambda(z)\\
%     &=\frac{\delta_0}{D}-\frac{1}{D}\mathcal{I}(\rho \|\rho_*)\\
%     &\leqslant \frac{\delta_0}{D}-\frac{}{D}\mathcal{D}_{KL}
% \end{align}
Applying the Gr\"onwall inequality to \eqref{eq:KLineq} yields the claimed estimate \eqref{eq: f est pl condition}.
% \ye{I suggest to break the deductions above and explain the reasons for the important steps (e.g., use orthogonality of $K$ and $I-K$, Pythagoras theorem, recall definition of $\delta_0$, applied the bounds in Theorem 3.2, etc.}
\end{proof}

We would like to mention that the bound obtained in Theorem \ref{theorem: kl estimation} is similar to the results reported in \cite{liu2022neural} except that they use different pullback metrics, namely we use $\widehat{G}(\theta)$ while $G(\theta)$ is considered in \cite{liu2022neural}.

\section{Numerical method}
\label{sec: numerical method}
The PWGF \eqref{def: relaxed pwgf} can be solved based on standard numerical integrators. We showcase how to implement this using the forward Euler scheme, whereas the idea can be easily generated to other methods such as Runge-Kutta 4th-order method, predictor-corrector method, and those with variable step sizes.
In forward Euler scheme, we can disretize the time $t$ as $\{hl: l=0,1,\dots\}$. Let $\theta^{l}$ be the approximation of $\theta(hl)$ of PWGF \eqref{def: relaxed pwgf}, then we can compute $\theta^{l+1}$ given $\theta^l$ by solving
\begin{align}
    \frac{\theta^{l+1}-\theta^l}{h}&=-\widehat{G}(\theta^{l})^{\dagger}\nabla_{\theta}F(\theta^{l}),\label{eq: fwdEuler}
\end{align}
for $l=0,1,\dots$.
The key to solving \eqref{eq: fwdEuler} is finding the minimum norm solution to the linear system $\widehat{G}(\theta^{l}) \eta = \nabla_{\theta} F(\theta^{l})$ for $\eta$. Minimum norm solutions of a linear system can be done by numerous existing methods, such as MINRES \cite{saad2003iterative}. In these methods, it is instrumental to implement the matrix-vector product with the matrix $\widehat{G}(\theta^{l})$. 

In our experiments, we use deep neural networks as the push-forward map $T_{\theta}$. There are several choices of the neural network architecture. We can use the invertible neural networks (e.g., normalizing flow \cite{rezende2015variational}, Real NVP \cite{dinh2016density} and neural ODE \cite{chen2018neural}) or non-invertible neural networks (e.g., the multi-layer perceptron or ResNet \cite{he2016deep}), both has its own advantages. Normalizing flow and continuous normalizing flow simplify the computation of log determinant of Jacobian matrix of the map, so we can easily compute the density function. For the experiment of the Fokker-Planck equation, we use a $60$-layer normalizing flow as the push-forward map $T_{\theta}$. For the experiments of the Porous-medium equation and the aggregation model, we pick the residual neural network as push-forward map:
\[
T_{\theta} = Id + R_{\theta}
\]
where $Id$ is the identity map and $R_{\theta}: \mathbb{R}^d \rightarrow \mathbb{R}^d$ is a standard multilayer perceptron. 
%with two hidden layers, and each hidden layer contains $100$ neurons and $50$ neurons respectively unless mentioned otherwise. 
We take the hyperbolic tangent function as the activation function since we require the second-order derivative in the computation of numerical integrals. The bias for the output layer in $f_{\theta}$ is set to be None. To solve the linear system $\widehat{G}(\theta) \dot \theta = - \nabla_{\theta} F(\theta)$, we utilized MINRES with tolerance $3 \cdot 10^{-4}$ in our experiments.

%As mentioned in \cite{}, it is computational expensive to evaluate the full matrix $\widehat{G}$, however, we can treat all $\widehat{G}$-related terms as solution to linear systems which can be effective solved through iterative method.

To initialize $\theta^{0}:=\theta(0)$, it depends on the given information of the initial distribution. If the initial probability density function is given and easy to generate samples from, such as Gaussian distribution or Gaussian mixture distribution, then one can initialize the $\theta(0)$ such that $T_{\theta(0)}$ is an identity map. If the initial samples are given but the initial probability density function is unknown, one can still initialize the $\theta(0)$ to make $T_{\theta(0)}$ an identity map. If the given initial density function is hard to sample from, one may use Gaussian distribution as a reference distribution and utilize methods such as MCMC \cite{berg2004markov} and Wasserstein GAN \cite{arjovsky2017wasserstein} to enforce the initial distribution.

If the probability density function is desired, one can consult with density estimation methods, such as kernel density estimation, while it's not the focus of this paper. In our experiments, we're given the initial density functions that are easy to sample, hence we initialize the push-forward map as an identity map, i.e., $\rho_{\theta(0)} \, dx = d (T_{\theta(0) \sharp} \lambda)(\cdot) = d \lambda(\cdot) = \rho(0, \cdot) \, dx$. %In our experiments, we initialized $\theta(0)$ to be close to $0$ so that $T_{\theta(0)}$ is an identity map. 

% \Yijie{Add if the initial density is hard to sample. Refer to MCMC, W-GAN, minimize KL-divergence (Sinkhorn). }

% One is to solve $\theta(0)$ by minimizing the KL divergence between $\rho_{\theta}$ and $\rho$, see \cite{rezende2015variational}, and the other one is to set $\theta(0)$ such that $T_{\theta(0)}$ is an identity map. 

% In our experiments, we utilized the second method. Using metrics other than KL divergence can be done similarly. As an alternative, we can choose the reference distribution $\lambda$ based on the $\rho(0,\cdot)$, i.e., $d\lambda(\cdot) = \rho(0,\cdot) dx$, and initialize $\theta(0)$ such that $T_{\theta(0)}$ is the identity map. In this case, it is ensured that $\rho_{\theta(0)}(\cdot) dx = d(T_{\theta(0) \sharp} \lambda)(\cdot) = d\lambda(\cdot) = \rho(0,\cdot)$, and hence the initialization is exact.
%
The numerical scheme is summarized in Algorithm \ref{alg:GFsolver}.
% \begin{algorithm}[H]
% \caption{Parameterized Hamiltonian flow solver}
% %\label{alg:HFsolver}
% \begin{algorithmic}
% \STATE{Initialize $\theta$, pretrain $\theta $ to fit the initial condition for $\rho$} \STATE{Initialize $p_0=G\dot{\theta}_0=\int_M \partial_{\theta}T^{\top}(z)v(T(z))dp(z)=\int_M \partial_{\theta}T^{\top}(z)\nabla \Phi(T(z))dp(z)=E_p[\nabla_{\theta}\Phi(T(z))]$}
% \FOR{$k=1, \cdots, N$}
% \STATE{Solve $\eta$ from the linear system $\hat{G}(\theta)\eta=p_k$}
% \STATE{Sample $\{X_1, \cdots, X_{K_{\theta}}\}$}
% \FOR{$j=1, \cdots, M$}
% \STATE{Update $\theta= \theta_k+h\eta$}
% \STATE{Apply one SGD step with learning rate $\lambda$ to loss function of variable $\eta$}
% \STATE{$\frac{1}{2}\eta \hat{G}(\theta)\eta - p_k^T\eta $}
% \ENDFOR
% \STATE{Set $\theta_{k+1}=\theta, \eta_{k+1}=\eta$}
% \STATE{Sample $\{X_1, \cdots, X_{K_p}\}$, evaluate $\nabla_{\theta}\mathcal{F}(\alpha)$}
% \STATE{Set $p_{k+1}=p_k+h[\frac{1}{2}\eta_{k+1}^T\nabla_{\theta}\hat{G}(\theta_{k+1})\eta_{k+1}-\nabla_{\theta}\mathcal{F}(\alpha)]$}
% \ENDFOR
% \end{algorithmic}
% \end{algorithm}

\begin{algorithm}[H]
\caption{Parameterized Wasserstein gradient flow solver}
\label{alg:GFsolver}
\begin{algorithmic}
\STATE{Initialize the neural network $T_{\theta}$}
% , and solve $\theta^0=\underset{\theta}{\textrm{argmin}}\{\mathcal{D}_{\mathrm{KL}}(\rho_0 \| \rho_{\theta})$\} 
\FOR{$l=0, \cdots, K-1$}
\STATE{Sample $\{X_1, \cdots, X_{N}\}$, evaluate $\nabla_{\theta}F(\theta^l)$}
\STATE{Solve $\eta$ as the minimum norm solution to the linear system $\widehat{G}(\theta^l)\eta = -\nabla_{\theta}F(\theta^l)$}

\STATE{Set $\theta^{l+1}=\theta^l + h \eta$}
% \Yijie{plus or minus}

\ENDFOR
\STATE{return $\theta^{K}, T_{\theta^K}$}
\end{algorithmic}
\end{algorithm}

\section{Numerical examples}
\label{sec:experiments}
In this section, we test the proposed PWGF method, Algorithm \ref{alg:GFsolver}, on three WGFs. Depending on the energy functional $\Fcal$, we choose push-forward maps as ResNet \cite{he2016deep}, or normalizing flow \cite{kobyzev2020normalizing}. Details are shown below.

\subsection{Fokker-Planck equation}
\label{sec:FPE}

We first test Algorithm \ref{alg:GFsolver} on a $30$-dimensional Fokker-Planck equation. We set the coefficient $D=1$ and choose the potential function $V$  to be the Styblinski-Tang function in \eqref{def: FPE}:
\begin{align}
\label{example: fpe V}
    V(x) =\frac{3}{50}\Big(\sum _{i=1}^d x_i^4-16x_i^2 +5x_i \Big).
\end{align}
The initial condition is set to be a Gaussian distribution with mean $\mu_0 = \mathbf{0}$ and variance $\sigma_0 = \mathbf{1}$.
We use an $M$-layer normalizing flow ($M = 40$) as the push-forward map $T_{\theta}$:
\begin{align}
    \label{experiment: nf structure}
    T_{\theta}=f_M\circ f_{M-1}\circ \cdots \circ f_2\circ f_1,
\end{align}
where $f_j\ (1\leqslant j\leqslant M)$ is:
\begin{align*}
    f_j(x)=x+\textrm{tanh}(w_j^\top x+b_j)u_j
\end{align*}
Here $w_j, u_j\in \mathbb{R}^d$ and $b_j\in \mathbb{R}$, and $\theta = \{w_j,b_j,u_j~:~1\leqslant j \leqslant M\}$.

We run Algorithm \ref{alg:GFsolver} with time step size $h=0.005$.  $2,000$ samples are generated from $\rho_{\theta}$ and their projection of these samples onto the $2$-dimensional subspace of $(x_1,x_2)$ are plotted. Notice that $V(x)$ has $4$ centers, and the push-forward samples follow a density that captures this structure as shown in \autoref{ST sampleplot}. The KL divergence \eqref{def: entropy H} is evaluated empirically with all samples by 
\begin{align} \label{eq: KL_eval}
    \nonumber
    \mathcal{H}_{KL}(\rho \| \rho_*) & = \int \frac{1}{D} V(x) \rho(x) + \rho(x) \log \rho(x) \, dx + \log Z_D \\
    \approx & \frac{1}{N} \sum_{i=1}^N \frac{1}{D} V(x^{(i)}) + \log \rho(x^{(i)}) + \log Z_D, \quad \text{ where } \{x^{(i)}\}_{i=1}^N \sim \rho.
\end{align}
We plot the KL divergence curve following the solution of PWGF in Figure \ref{fig: ST kl convergence}, which shows a clear decreasing trend. 

\begin{figure}[t]
    \begin{subfigure}{0.16\textwidth}
        \centering
        \includegraphics[width=0.99\linewidth]{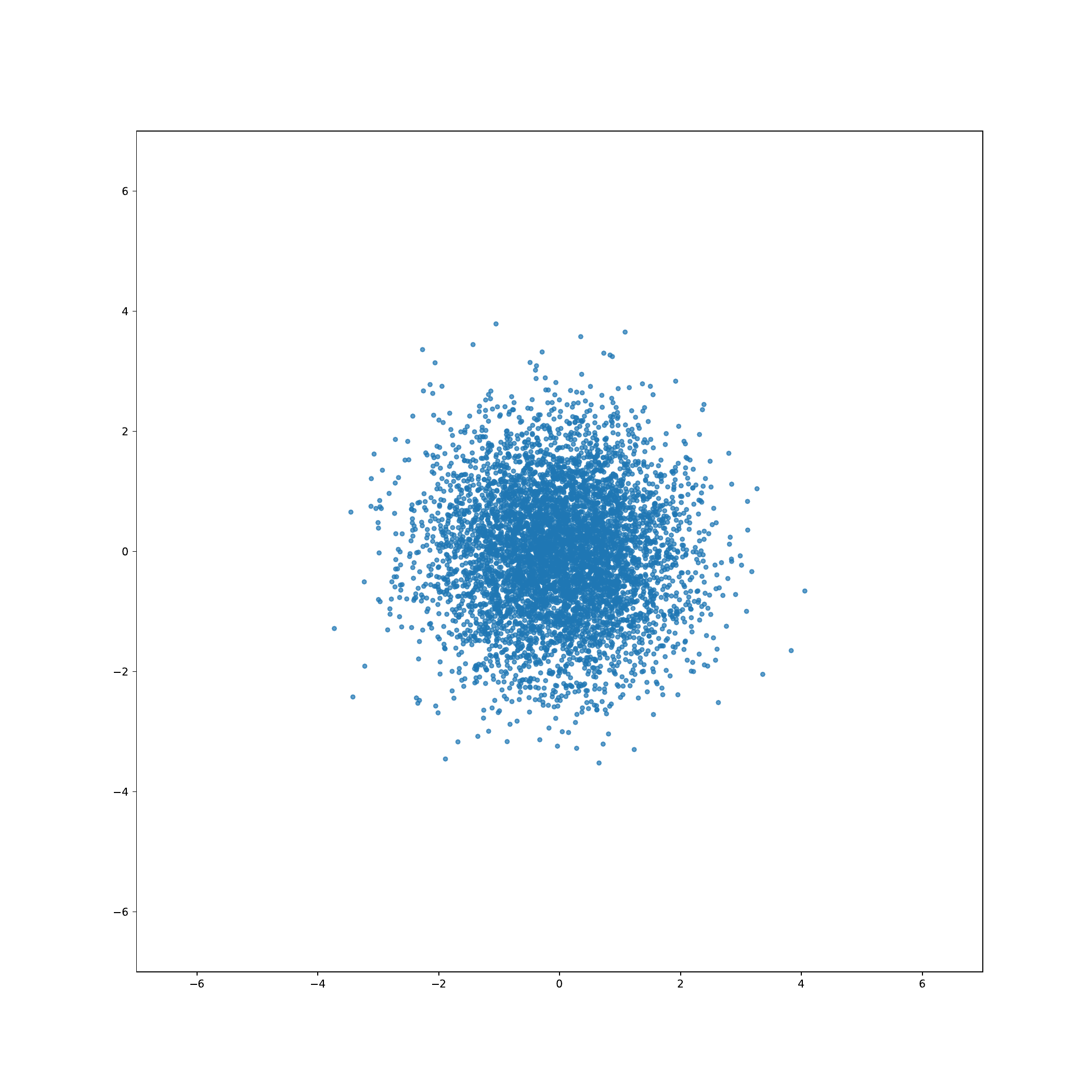}
        \caption{$t=0$}
    \end{subfigure}%
    %~
    \begin{subfigure}{0.16\textwidth}
        \centering
        \includegraphics[width=0.99\linewidth]{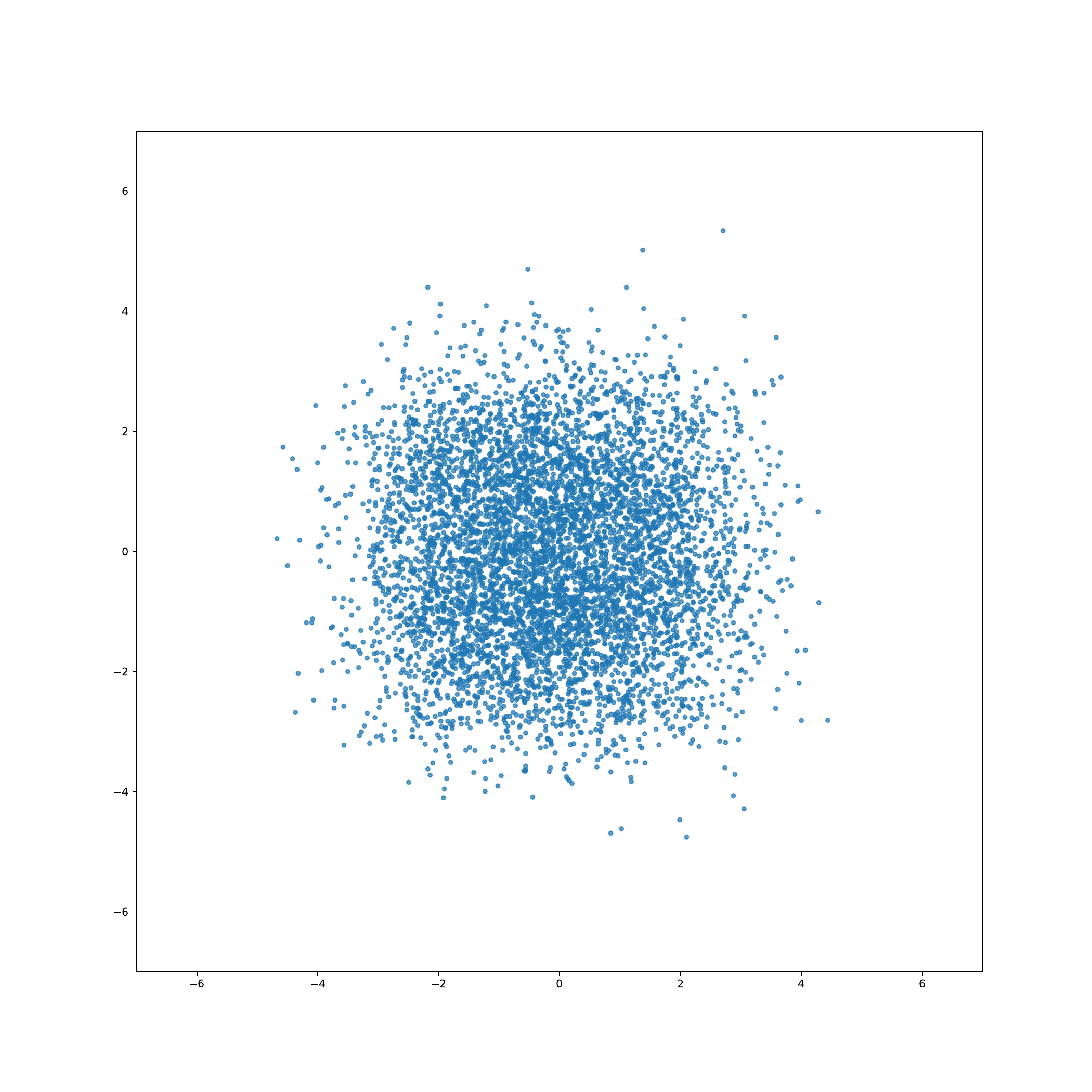}
        \caption{$t=0.3$}
    \end{subfigure}
    %~
    \begin{subfigure}{0.16\textwidth}
        \centering
        \includegraphics[width=0.99\linewidth]{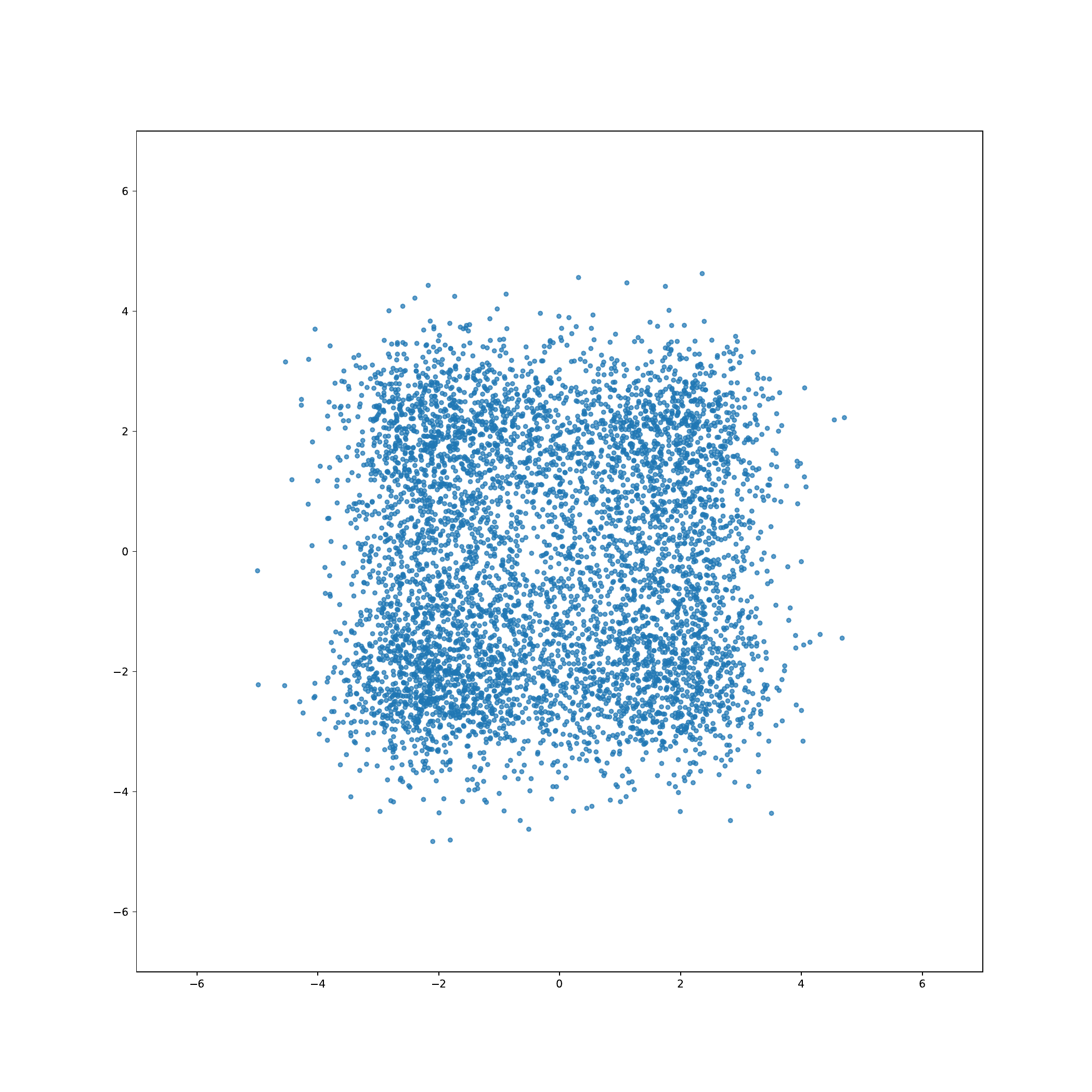}
        \caption{$t=0.6$}
    \end{subfigure}
    %\\
    \begin{subfigure}{0.16\textwidth}
        \centering
        \includegraphics[width=0.99\linewidth]{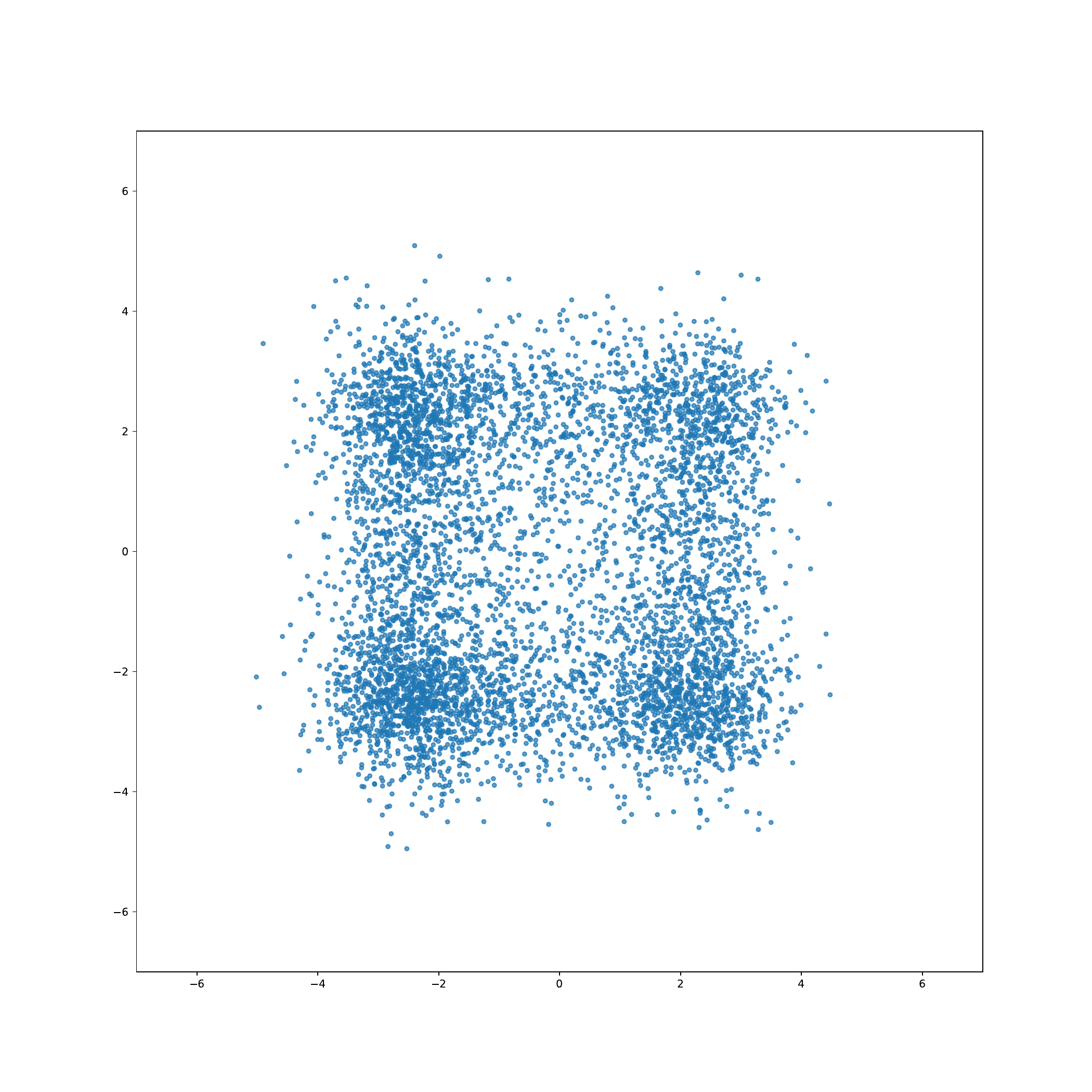}
        \caption{$t=0.9$}
    \end{subfigure}%
    %~
    \begin{subfigure}{0.16\textwidth}
        \centering
        \includegraphics[width=0.99\linewidth]{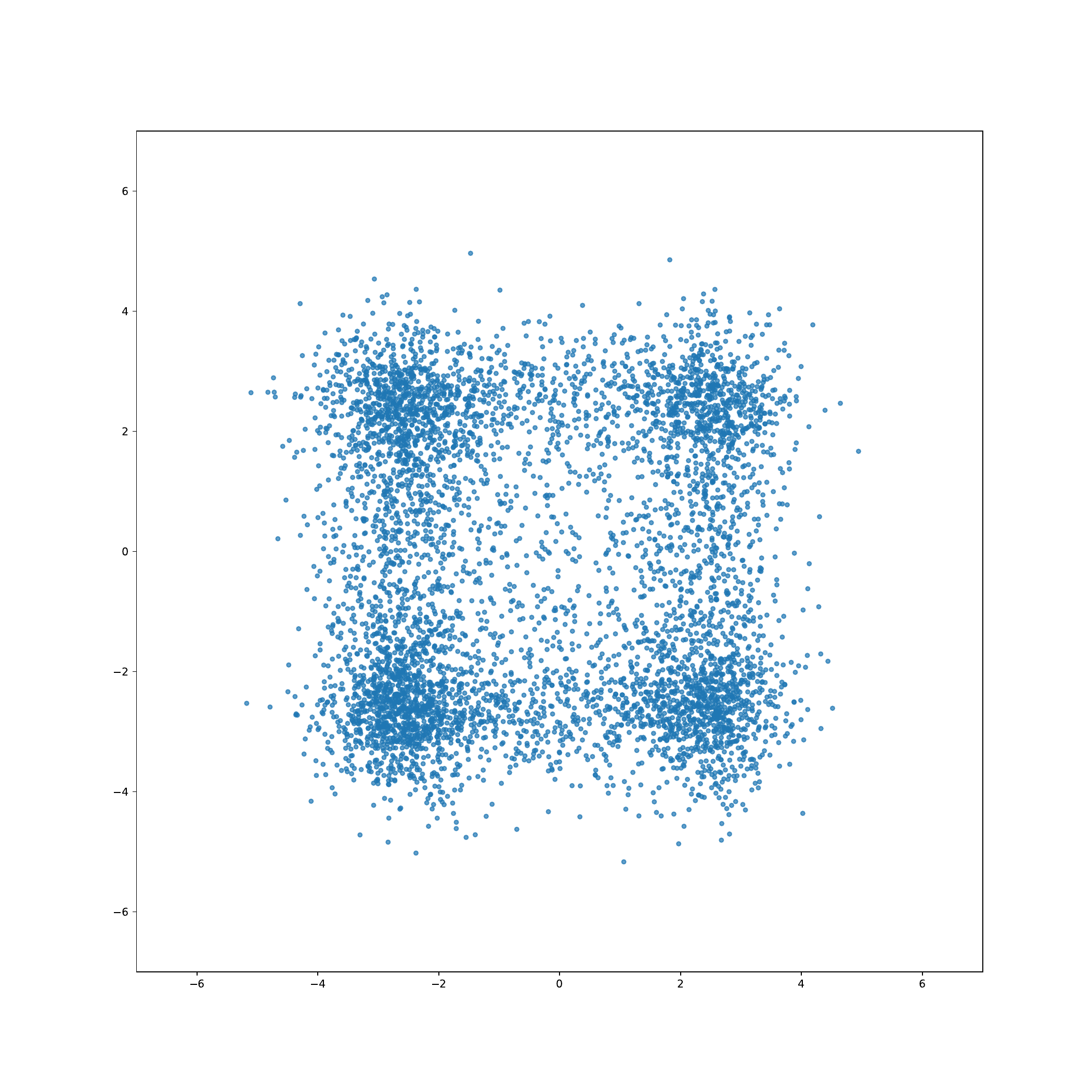}
        \caption{$t=1.2$}
    \end{subfigure}
    %~
    \begin{subfigure}{0.16\textwidth}
        \centering
        \includegraphics[width=0.99\linewidth]{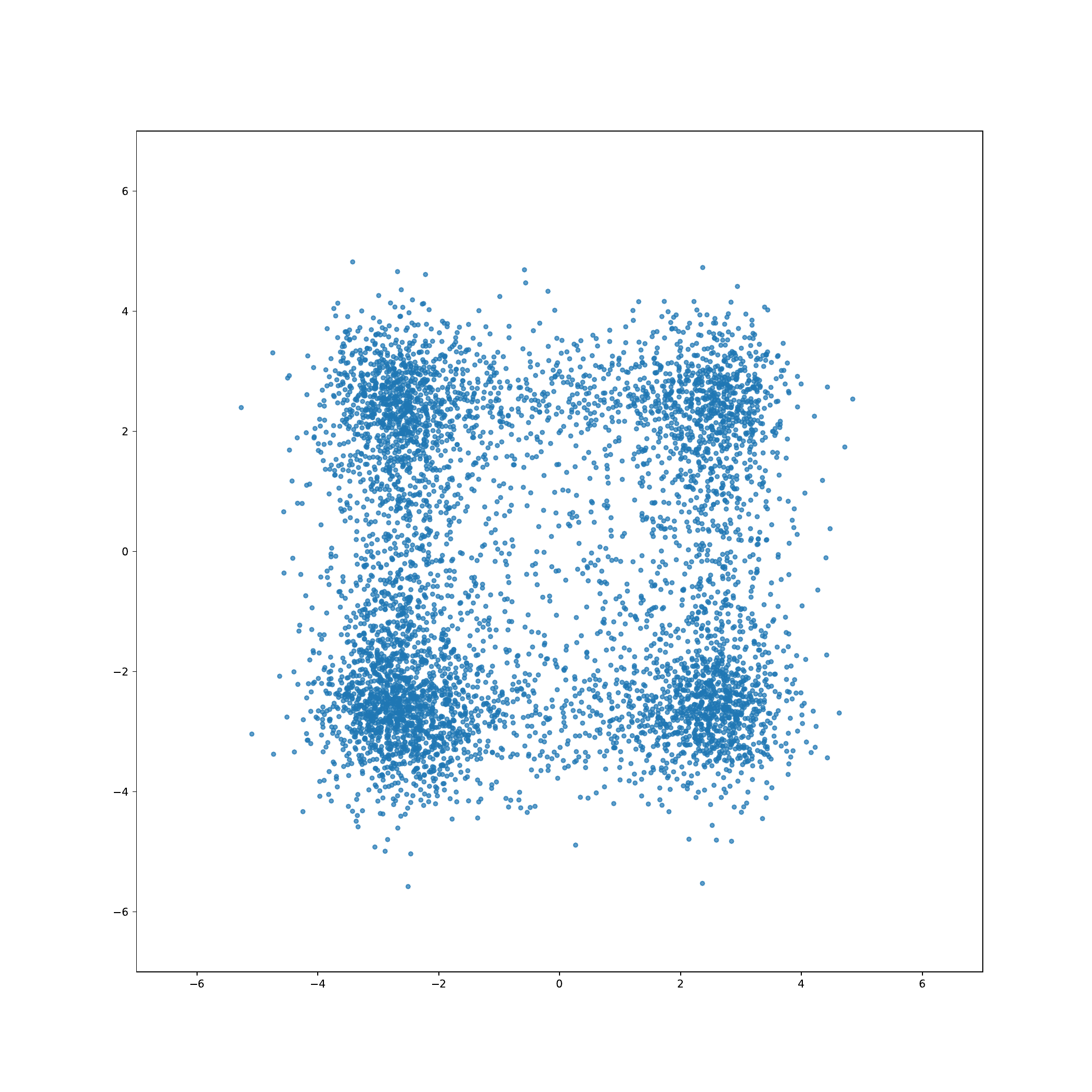}
        \caption{$t=1.5$}
    \end{subfigure}
    \caption{Sample plots of computed $\rho_{\theta}$ at different time $t$ for Fokker-Planck equation with the Styblinski-Tang function as the potential function $V(x)$}
    \label{ST sampleplot}
\end{figure}
\begin{figure}[h]
    \centering
    \includegraphics[width=0.45\linewidth]{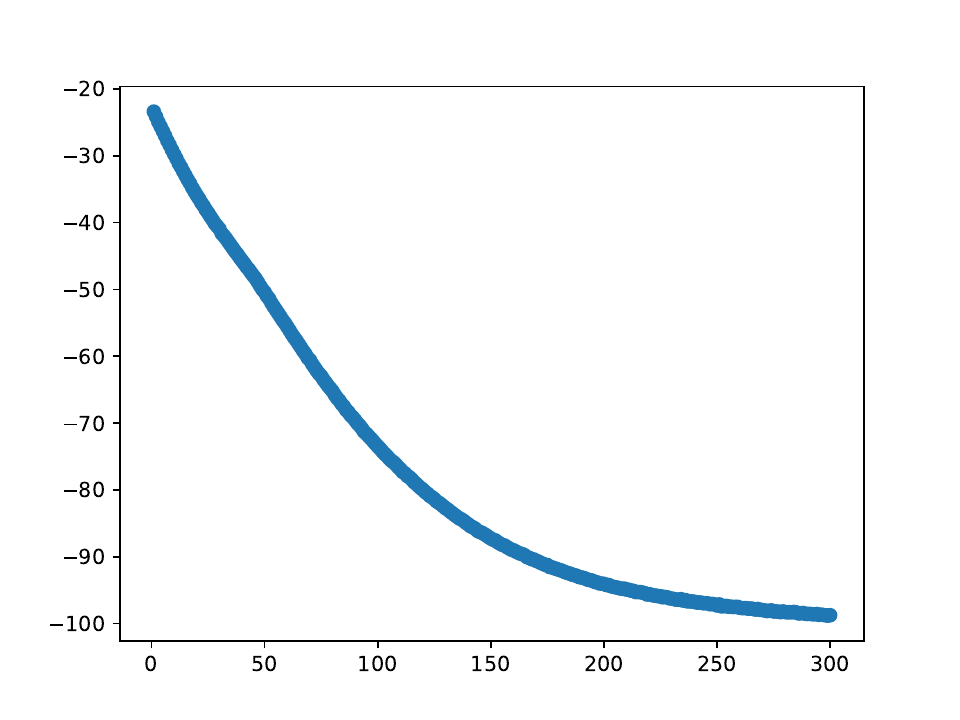}
    \caption{The decay of KL divergence along the solution of PWGF for the Fokker-Planck equation with Styblinski-Tang potential in the $30$-dimensional example.}
    \label{fig: ST kl convergence}
\end{figure}

To evaluate the improvement in its computation efficiency with the new metric $\hat{G(\theta)}$, we compare the running time of PWGF Algorithm \ref{alg:GFsolver} with the PFPE algorithm proposed in \cite{liu2022neural}. We take the $2$-dimensional FPE with $V(x)$ defined in \eqref{example: fpe V}. For both algorithms, we use the same neural network structure ($40$-layer normalizing flow) and the same sample size ($12000$) to evaluate the KL divergence. The $\psi$ neural network in PFPE, as well as the training parameters, are set as its default values as reported in \cite{liu2022neural}. We run both algorithms with the same step size. It takes $14.26$ seconds for PWGF to finish $10$ time steps, while $6093.58$ seconds for PFPE to finish $10$ time steps, indicating a more than $400$ times speed up. Both codes are run on an NVIDIA RTX-3070 GPU with CUDA enabled. 

% \Yijie{Add some descriptions of captions}
\subsection{Porous medium equation}
We apply Algorithm \ref{alg:GFsolver} to solve the porous medium equation \eqref{PME}. In recent years, the porous medium equation has been solved by several other methods, such as physics-informed neural networks (PINN) \cite{wang2023neural} which computes the density function directly in high dimensions and neural network-based implicit particle methods with JKO scheme \cite{lee2023deep} which computes samples of the solution. For all three examples of porous-medium equation, we use a neural network with residual structure as the push-forward map
\begin{align}
    T_{\theta}=Id+R_{\theta}
\end{align}
where $R_{\theta}$ is a multi-layer perceptron with $3$ hidden layers and $100$ neurons in each hidden layer.

\textit{An example with exact weak solution:} 
Suppose the initial condition $\rho(x, 0) = \delta(x)$ is the Dirac delta function.
% \Hao{I or Yijie plan to do this experiment}
% \ye{What is your plan on this experiment? If using chen2018neural to parameterize $T_{\theta}$, this problem can be done easily: as $\rho_{\theta}(x_i(t))$ can be computed alongside with $x_i(t)= T_{\theta(t)}z_i$ for $z_i \sim \lambda$ using chen2018neural, then just compute 
% \[
% F(\theta) = \Fcal(\rho_{\theta}) = \frac{1}{l-1}\int \rho_{\theta}(x)^{l}\,dx \approx \frac{1}{N(l-1)} \sum_{i=1}^{N} \rho_{\theta}(x_i)^{l-1}
% \]
% and apply auto-diff to get $\nabla_{\theta} F(\theta)$. Similarly $T_{\theta}(z_i)$ can be computed and hence $\widehat{G}(\theta)$ using the trick in PWHF.
% }

\begin{align}\label{eq:pme}
    \frac{\partial \rho(x,t)}{\partial t} & = \Delta \rho^m, \quad x \in \mathbb{R}^d \\
    \rho(x,0) & = \delta(x)
\end{align}
\noindent The exact weak solution to \eqref{eq:pme} is 
\begin{equation}\label{eq: zkb solution}
    \rho(x,t) = t^{-\alpha} F\left(\frac{x}{t^{\beta}} \right),\quad \mbox{where} \quad F(\xi) = (C - k \xi^2)_+^{\frac{1}{m-1}},
\end{equation}
and $(s)_+ = \max\{s, 0\}$. The constants are 
\begin{align*}
    \alpha = \frac{d}{d(m-1) + 2}, \quad \beta = \frac{\alpha}{d}, \quad k = \frac{(m-1)\alpha}{2md}
\end{align*}
The constant $C > 0$ in \eqref{eq: zkb solution} can be uniquely determined by normalizing the mass as 
\[
\int \rho \, dx = 1,
\]
which gives 
\[
C = \left[ \frac{\left( \frac{k}{\pi} \right)^{\frac{d}{2}} \Gamma(\frac{d}{2})}{d \cdot B(\frac{d}{2}, \frac{m}{m-1})}\right]^{\frac{1}{\gamma}},
\]
where $B(\cdot, \cdot)$ and $\Gamma(\cdot)$ are beta and gamma functions respectively and $\gamma = \frac{d}{2(m-1)\alpha}$ \cite{vazquez2007porous}. 

% \noindent We conducted the experiments with $d=2$, $d=5$ and $d = 15$ respectively. We use neural network with residual structure as the push-forward map in this example:
% \begin{align}
%     T_{\theta}=Id+R_{\theta}
% \end{align}
% where $R_{\theta}$ is a multi-layer perceptron with $3$ hidden layers and $100$ neurons in each hidden layer.

For the experiments of $d=2$, the parameter $m$ is set to $2.4$, for the experiment of $d = 5$, the parameter $m$ is set to $3$, and for the experiment of $d=15$, the parameter $m$ is set to $2$.  Since the initial condition is a Dirac delta mass which is impossible to be represented by deep neural networks, we set the initial time to a positive value. More precisely, we compute the dynamics of the samples generated from the probability density $\rho$ from $t=0.1$ to $t=0.6$ for $d=2$ and from $t = 1$ to $t = 1.5$ for both $d=5$ and $d=15$. In all experiments, the step size $h = 10^{-4}$, and the number of time steps is $T = 5000$. The sample size to evaluate $\widehat{G}(\theta)$ and $F(\theta)$ are both $30,000$. The graph of the samples with $d = 2$, $d = 5$ and $d = 15$ are shown in \autoref{fig: PME_d=2}, \autoref{fig: PME_d=5} and \autoref{fig: PME_d=15} respectively. In the plots, the blue points are the computed samples and the colored curves are the level sets of the exact solutions. We select darker colors representing smaller values when plotting the level curves to clearly show the boundary of the finite support of the exact solution. In the example with $d=2$, though the computed samples run a bit more quickly than the exact solution, the numerical solution still behaves the same as the exact solution in the sense that they are compactly supported. In the example with $d=15$, all samples are still within the compact support, despite that sampling from a density function in high dimensions remains as a challenging problem. All experiments run on an NVIDIA A$100$ GPU with $40$GB GPU memory.
% \zhou{The plots should be explained in the main text too.}

% d = 2
\begin{figure*}[t!]
    \begin{subfigure}{0.33\textwidth}
        \centering
        \includegraphics[width=1\linewidth]{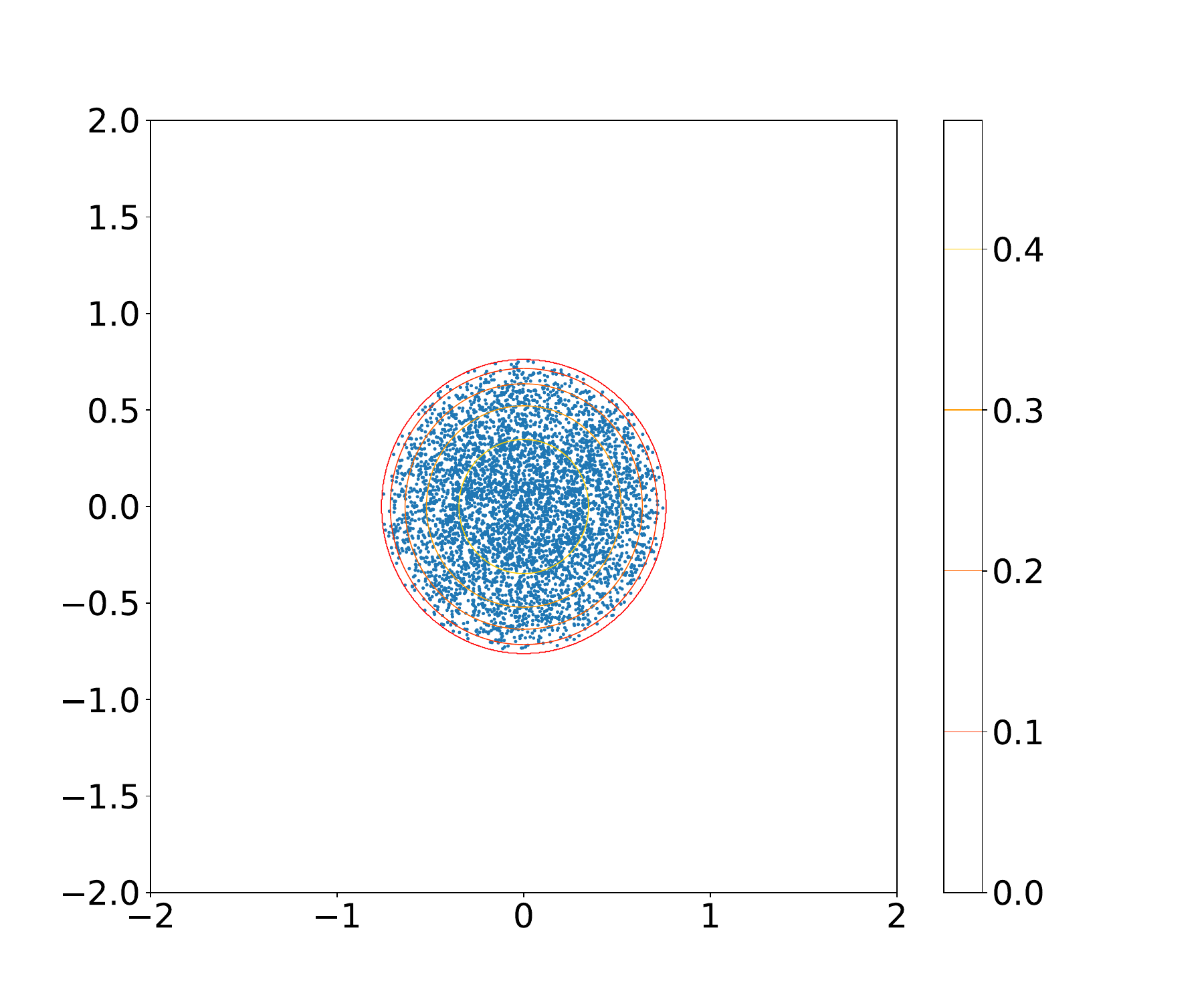}
        \caption{$t=0.1$}
    \end{subfigure}%
    %~
    \begin{subfigure}{0.33\textwidth}
        \centering
        \includegraphics[width=1\linewidth]{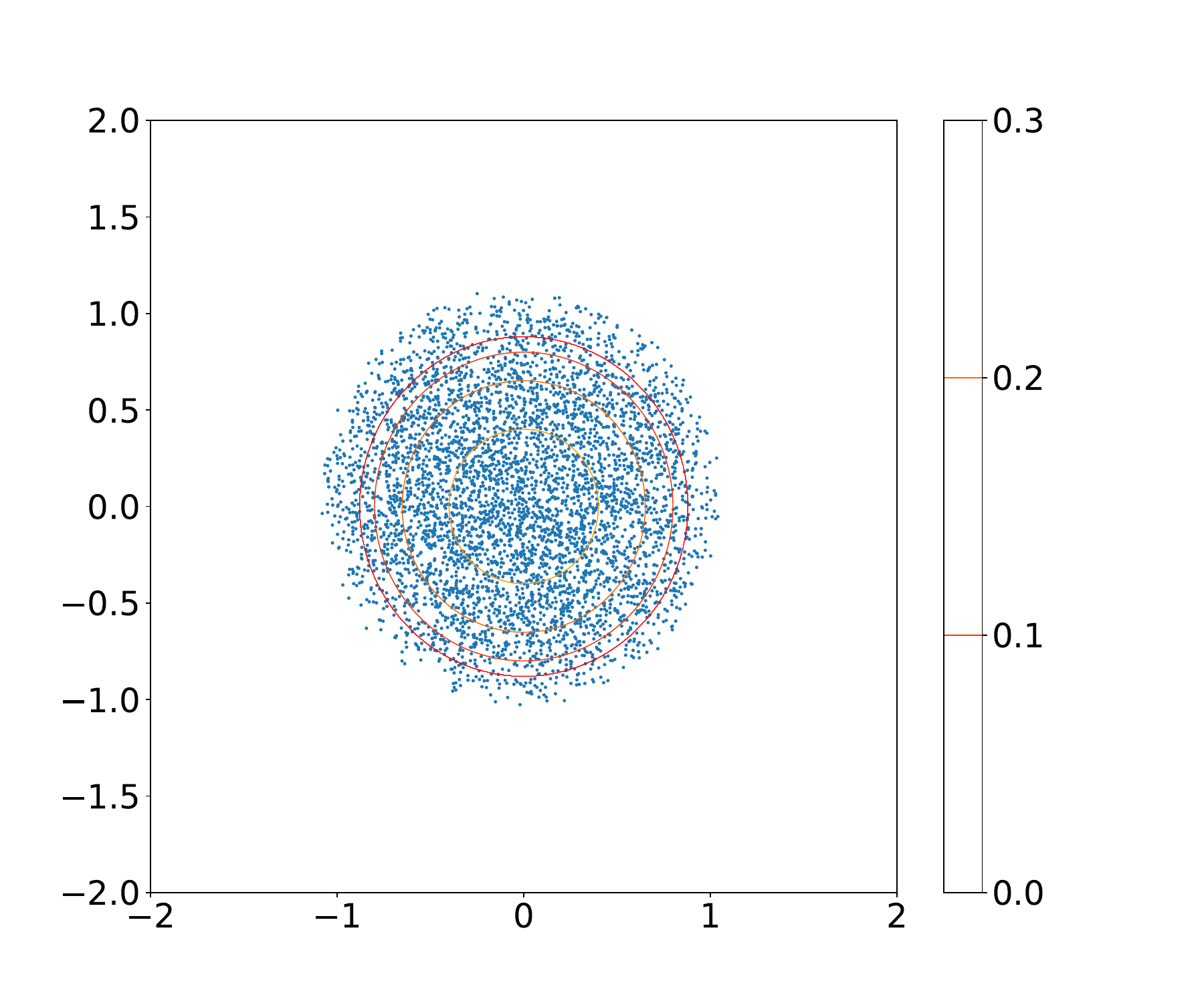}
        \caption{$t=0.2$}
    \end{subfigure}
    %~
    \begin{subfigure}{0.33\textwidth}
        \centering
        \includegraphics[width=1\linewidth]{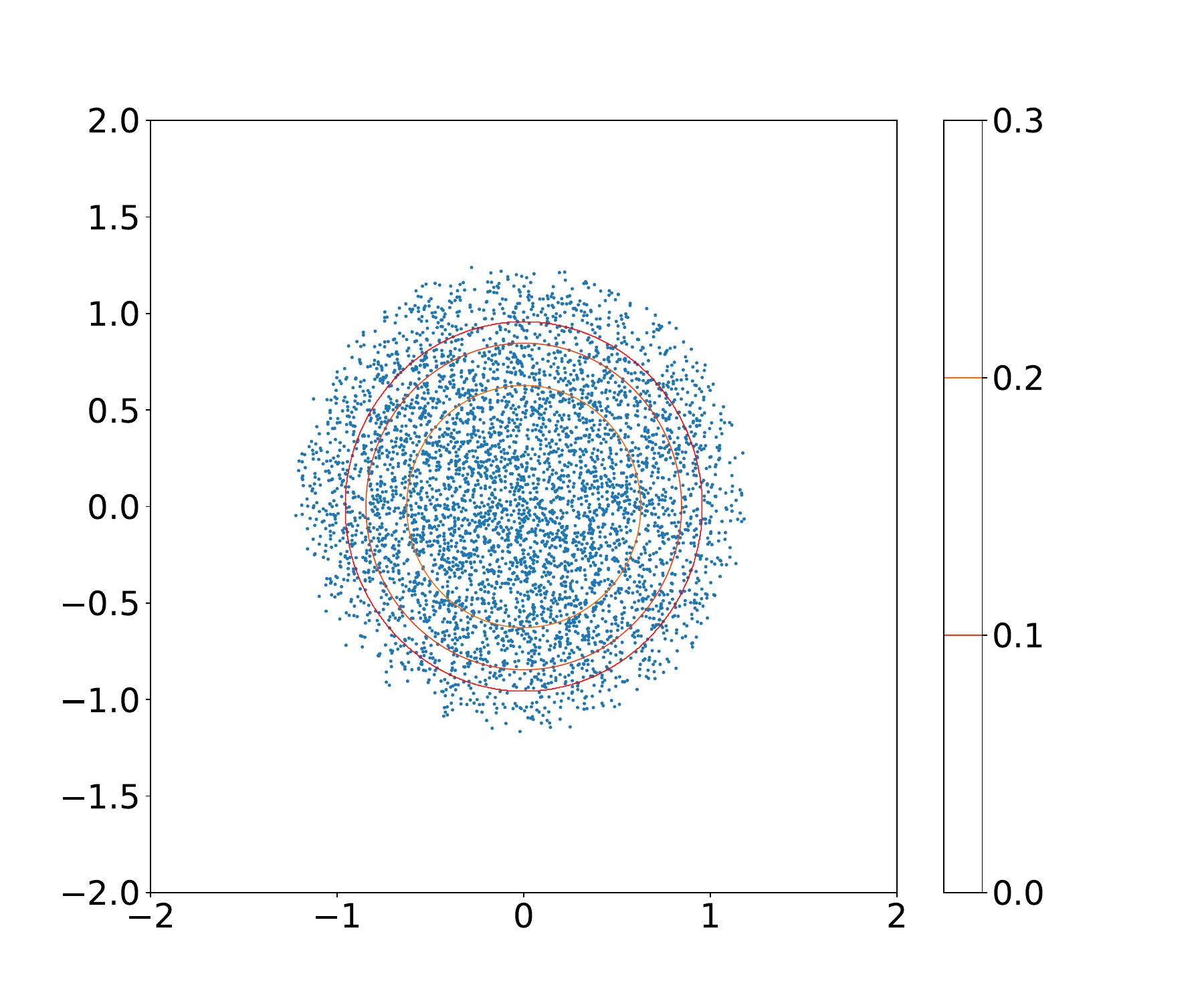}
        \caption{$t=0.3$}
    \end{subfigure}
    %~
    \begin{subfigure}{0.33\textwidth}
        \centering
        \includegraphics[width=1\linewidth]{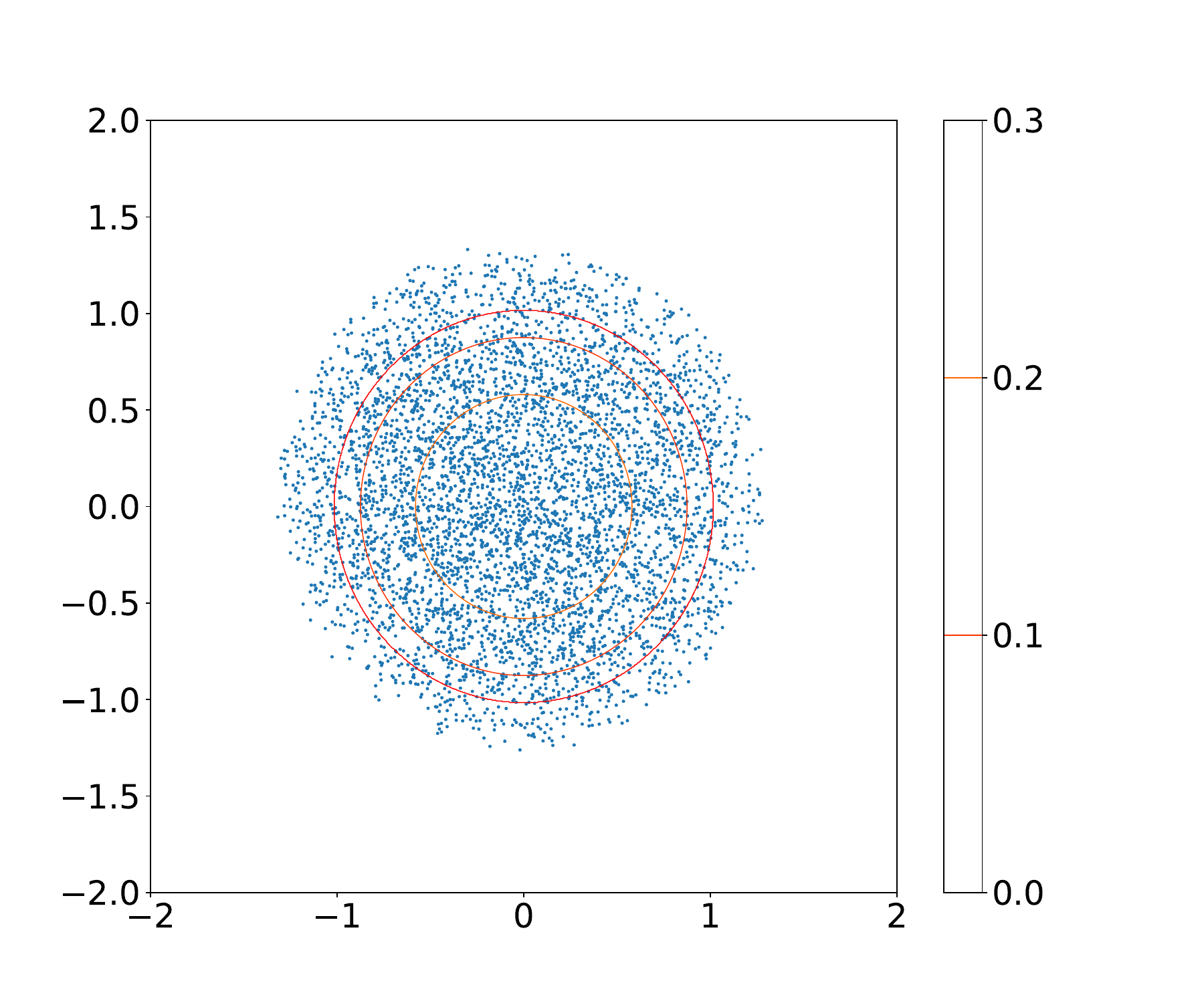}
        \caption{$t=0.4$}
    \end{subfigure}
    %~
    \begin{subfigure}{0.33\textwidth}
        \centering
        \includegraphics[width=1\linewidth]{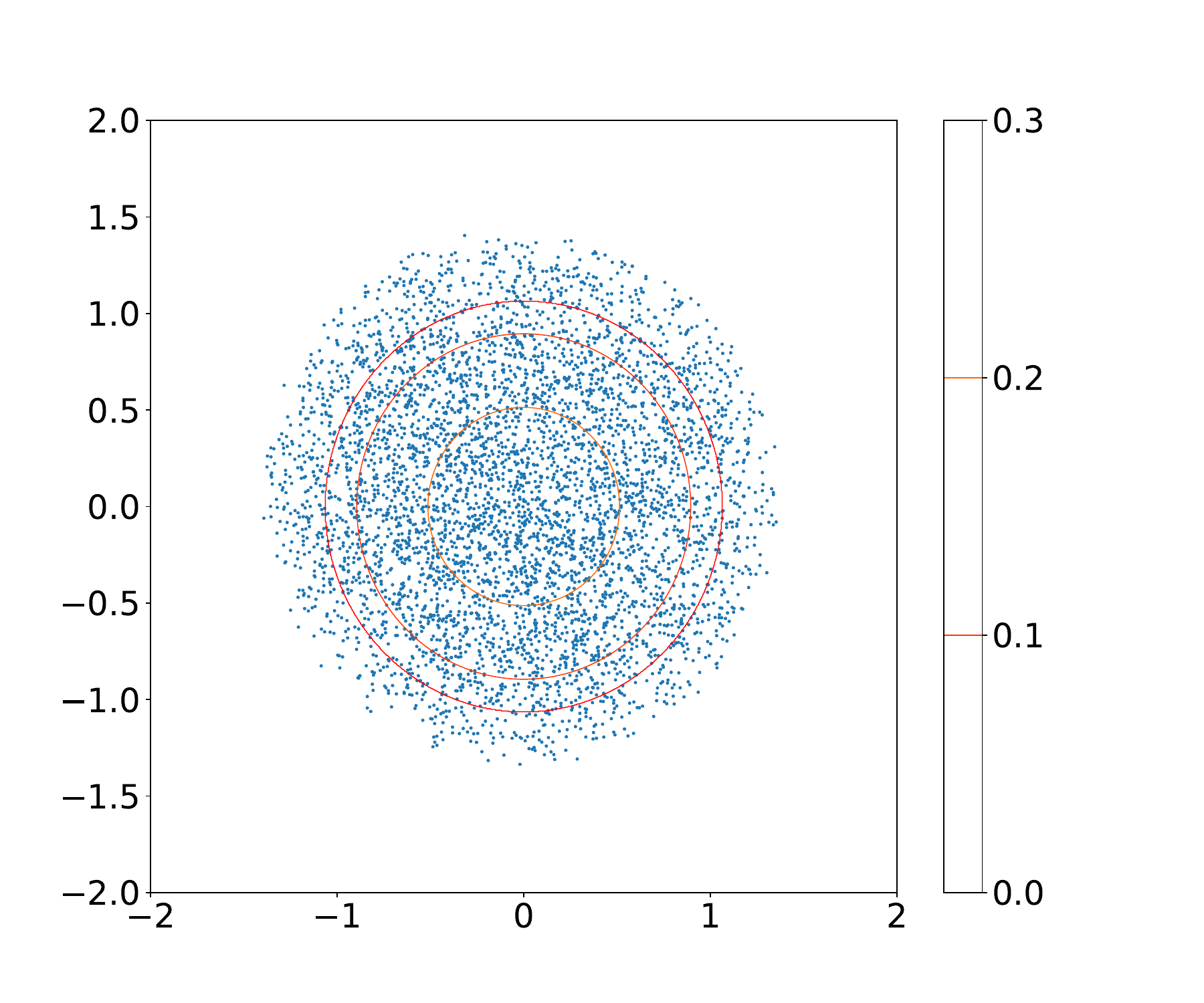}
        \caption{$t=0.5$}
    \end{subfigure}
    %~
    \begin{subfigure}{0.33\textwidth}
        \centering
        \includegraphics[width=1\linewidth]{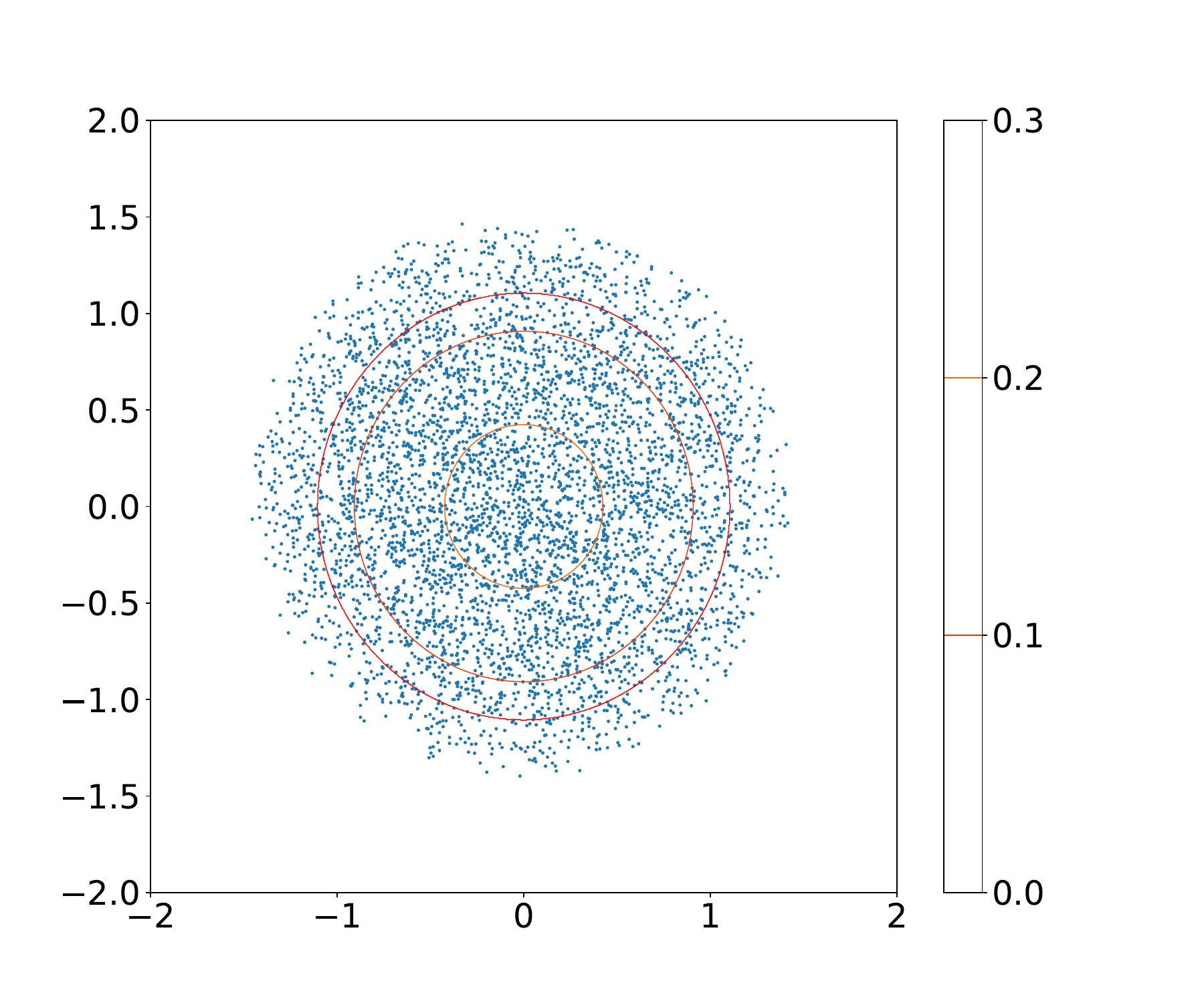}
        \caption{$t=0.6$}
    \end{subfigure}
    \caption{Sample plots of computed $\rho_{\theta}$ at different time $t$ for Porous Medium equation with Dirac Delta function as the initial condition for $d=2$ and $l=2.4$. The figures are plotted with $5000$ samples. In the level curves, darker colors correspond to smaller values to emphasize the support. 
    % \ye{A few suggestions: (1) Add a reference to the section or example number; (2) What is $m$? (3) Remove the legends and titles as they are too small to read but can be explained in the caption. For the axis labels, if want to keep them, they should be larger (similar to the main text size) to read. (4) Python can save figures in PDF format. PDF figures are small in size and do not get blurred when zooming. (5) What are the meanings of the colors? The yellow dots (I guess they are the particles) do not seem to match the color (particles have larger contour) very well in these plots except the first one?}
    }
    \label{fig: PME_d=2}
\end{figure*}

% d = 5
\begin{figure*}[t!]
    \begin{subfigure}{0.33\textwidth}
        \centering
        \includegraphics[width=1\linewidth]{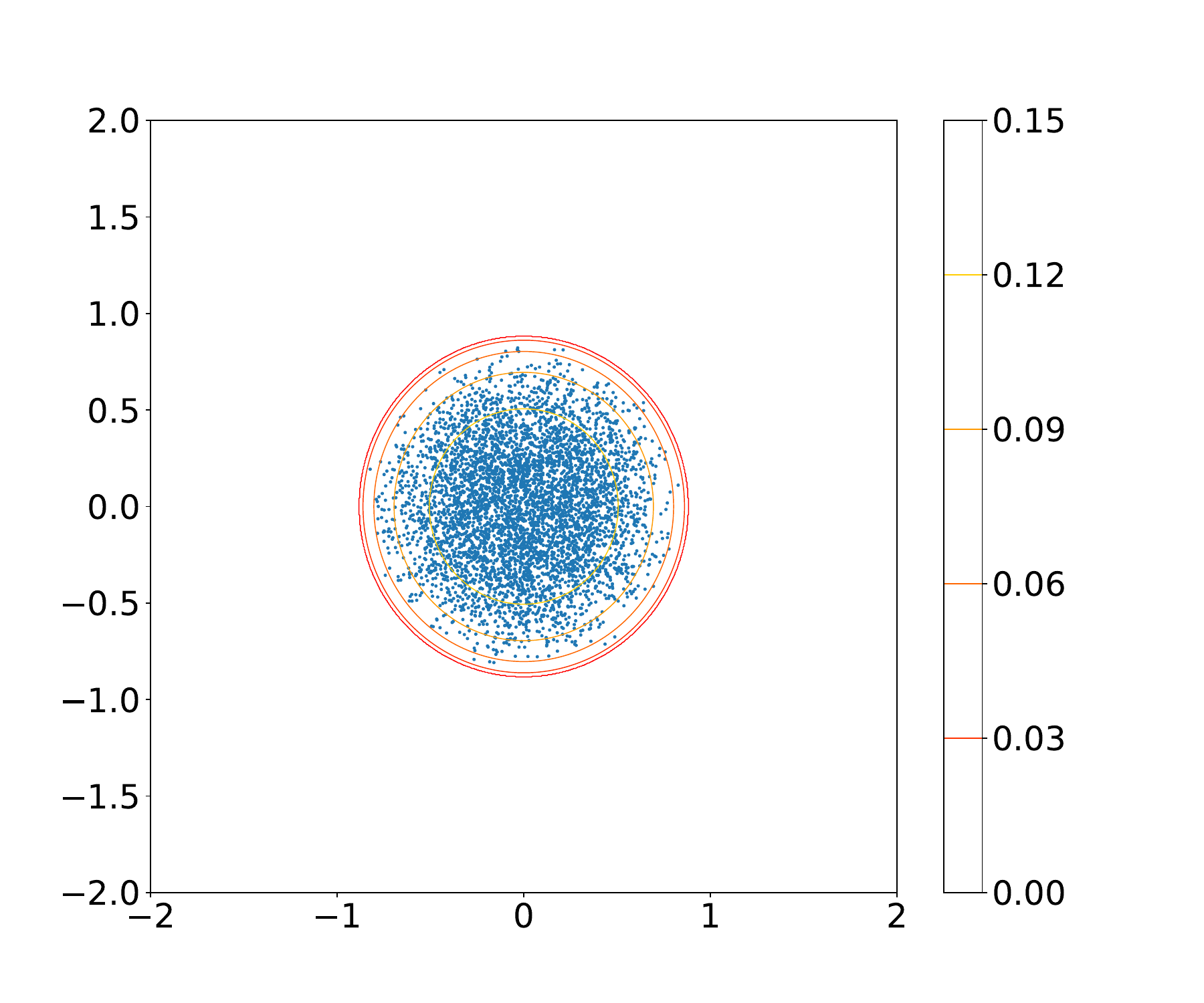}
        \caption{$t=1$}
    \end{subfigure}%
    %~
    \begin{subfigure}{0.33\textwidth}
        \centering
        \includegraphics[width=1\linewidth]{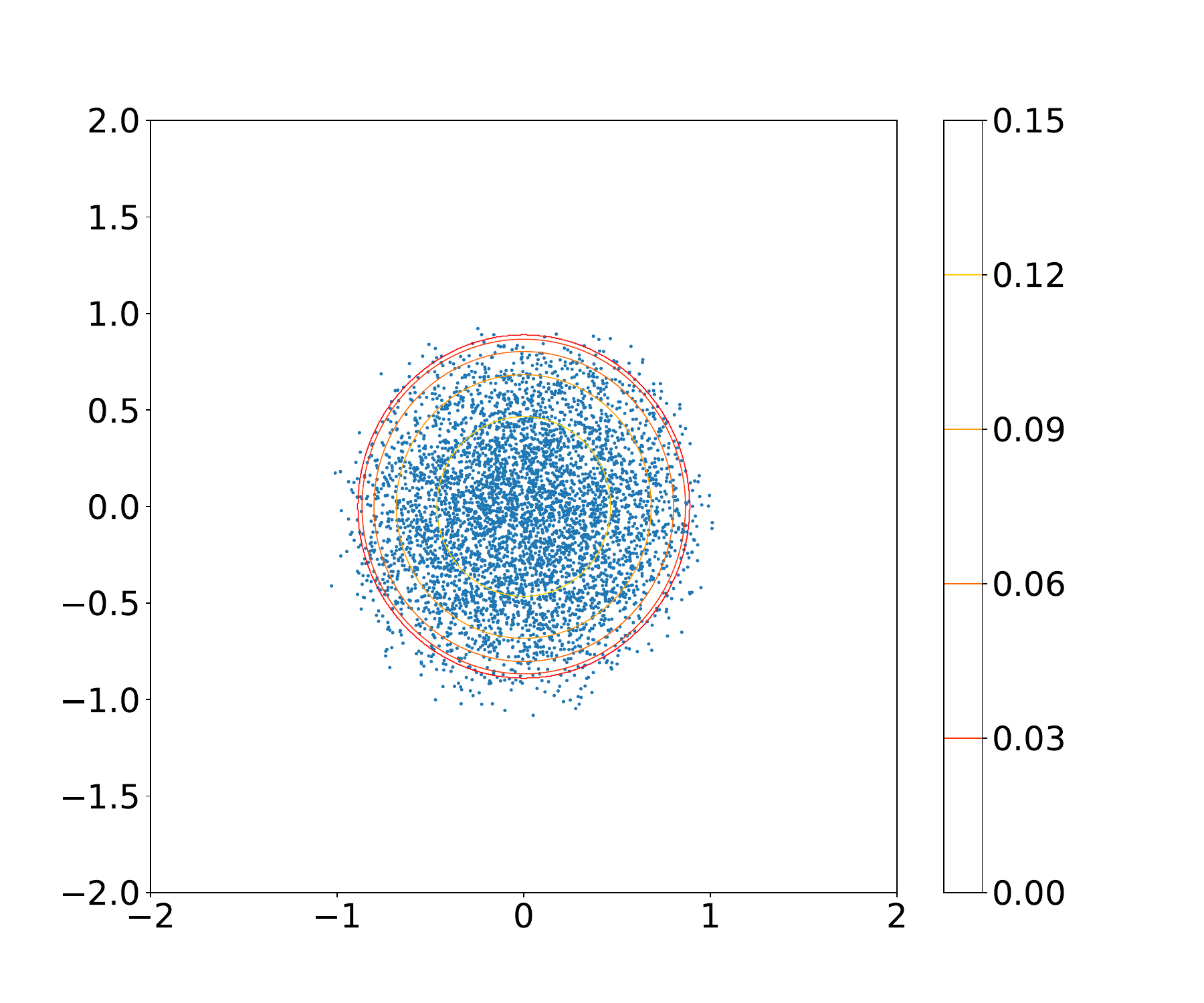}
        \caption{$t=1.1$}
    \end{subfigure}
    %~
    \begin{subfigure}{0.33\textwidth}
        \centering
        \includegraphics[width=1\linewidth]{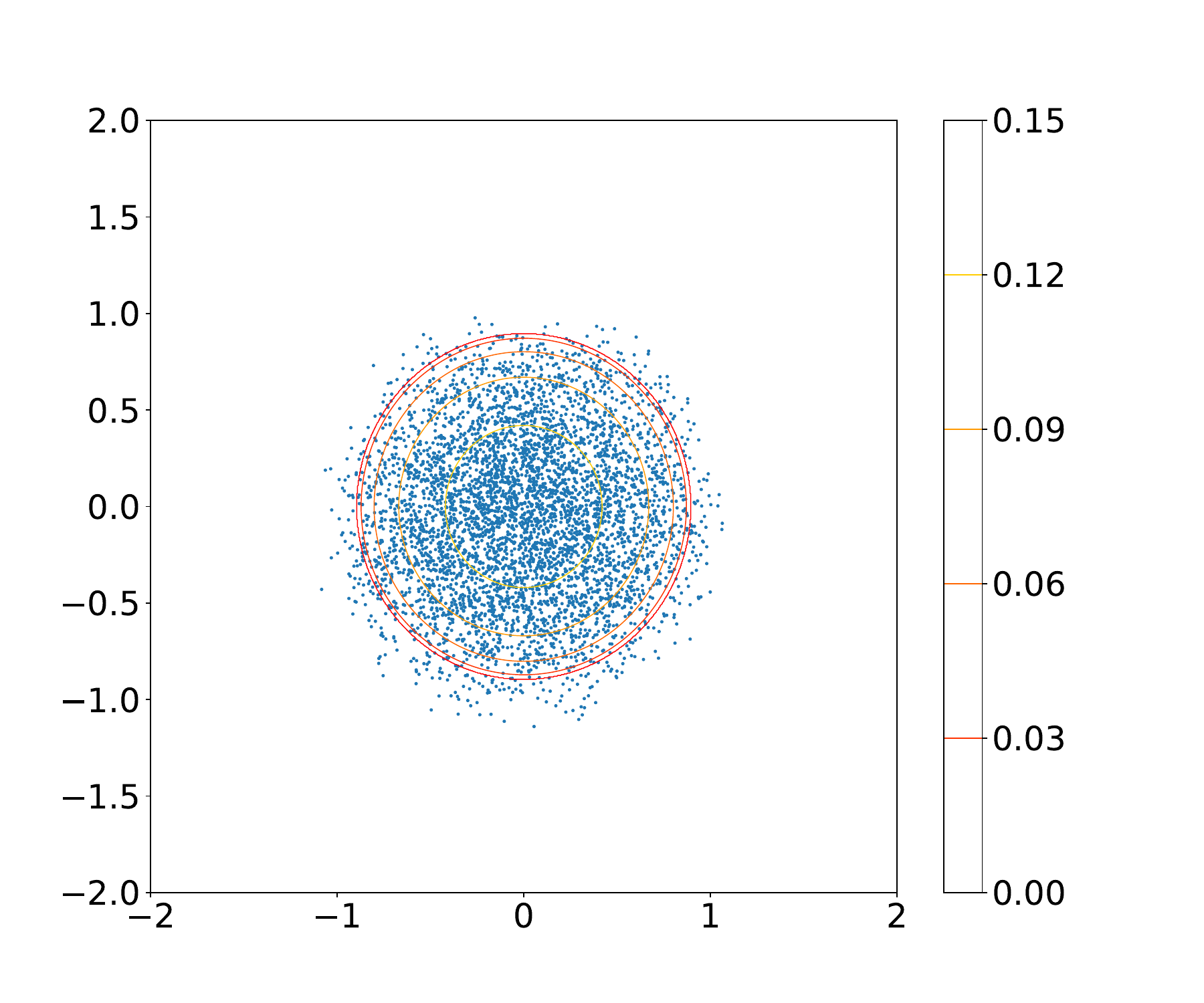}
        \caption{$t=1.2$}
    \end{subfigure}
    %~
    \begin{subfigure}{0.33\textwidth}
        \centering
        \includegraphics[width=1\linewidth]{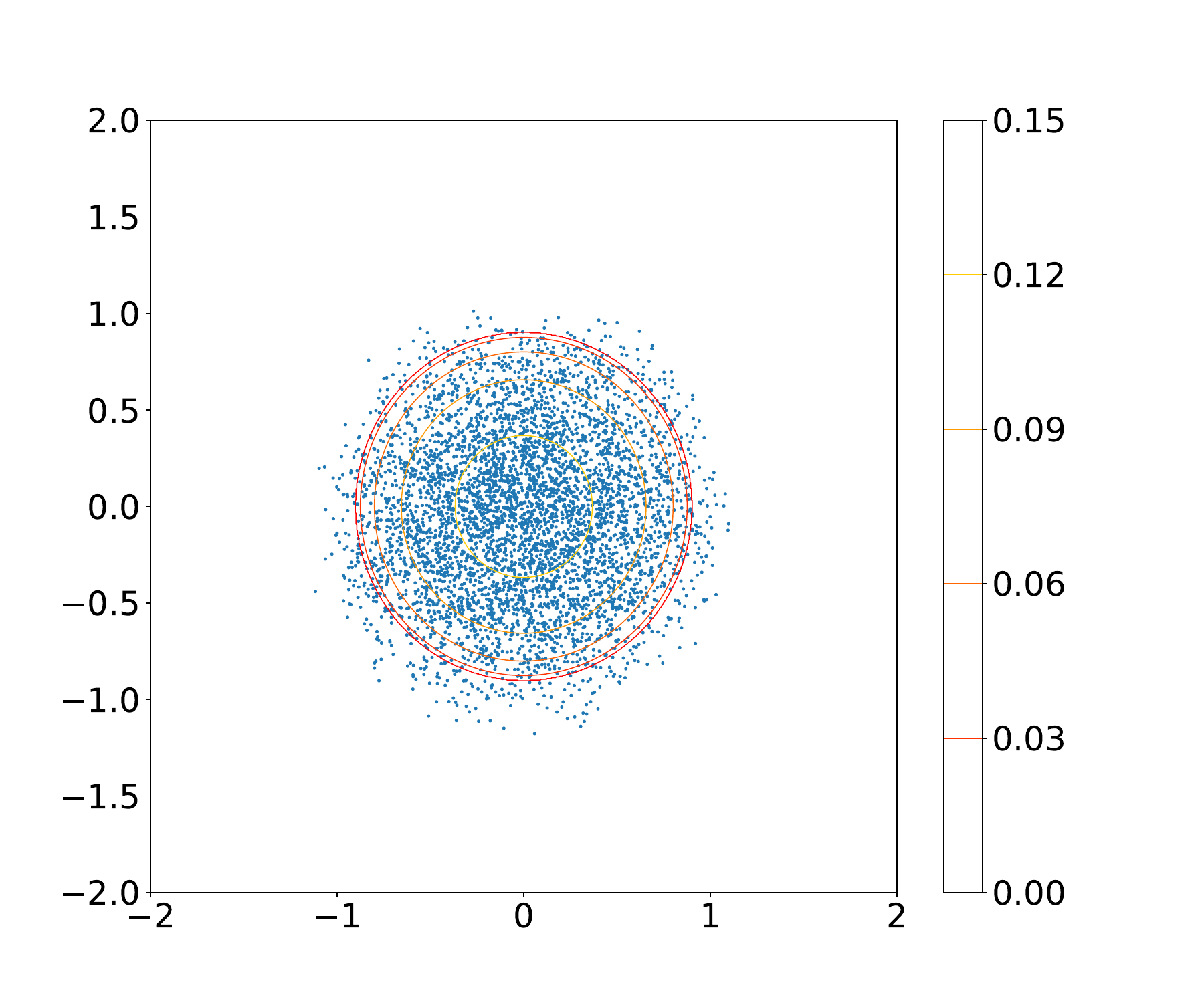}
        \caption{$t=1.3$}
    \end{subfigure}
    %~
    \begin{subfigure}{0.33\textwidth}
        \centering
        \includegraphics[width=1\linewidth]{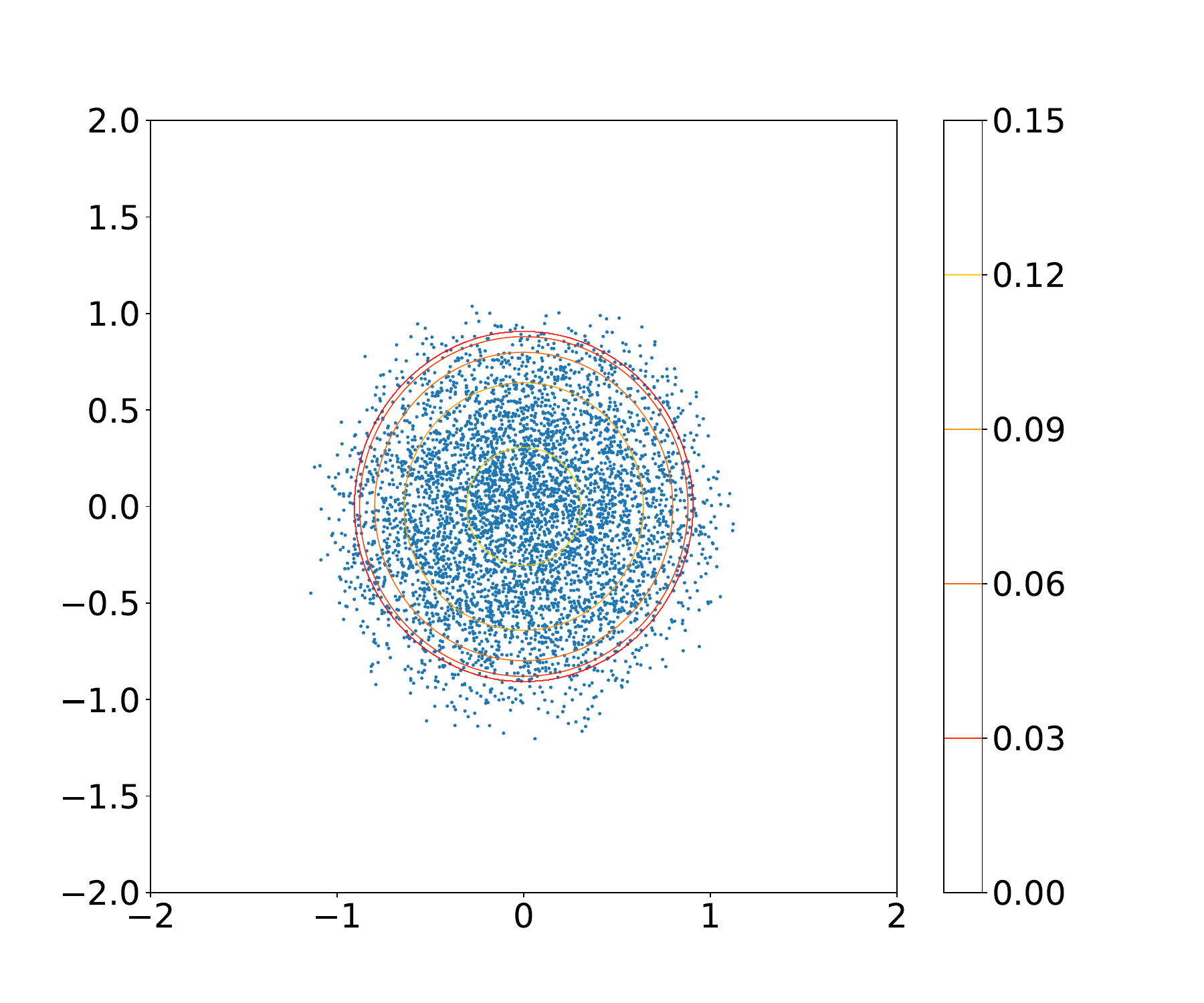}
        \caption{$t=1.4$}
    \end{subfigure}
    %~
    \begin{subfigure}{0.33\textwidth}
        \centering
        \includegraphics[width=1\linewidth]{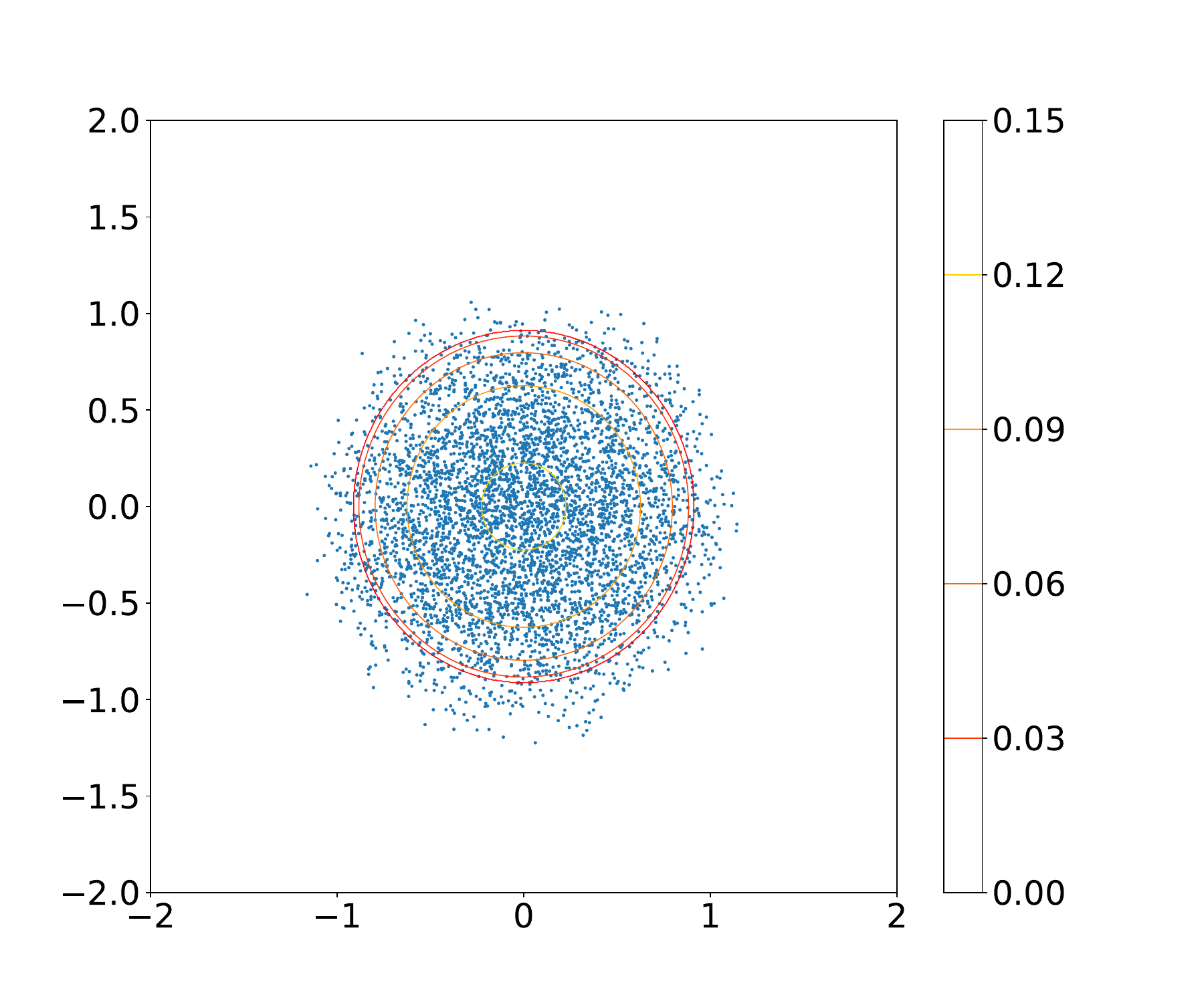}
        \caption{$t=1.5$}
    \end{subfigure}
    \caption{Sample plots of computed $\rho_{\theta}$ at different time $t$ for the porous medium equation with Dirac Delta function as the initial condition for $d=5$ and $l = 3$. The figures are plotted with $5000$ samples. In the level curves, darker colors correspond to smaller values to emphasize the support. }
    \label{fig: PME_d=5}
\end{figure*}

% d = 15
\begin{figure*}[t!]
    \begin{subfigure}{0.33\textwidth}
        \centering
        \includegraphics[width=1\linewidth]{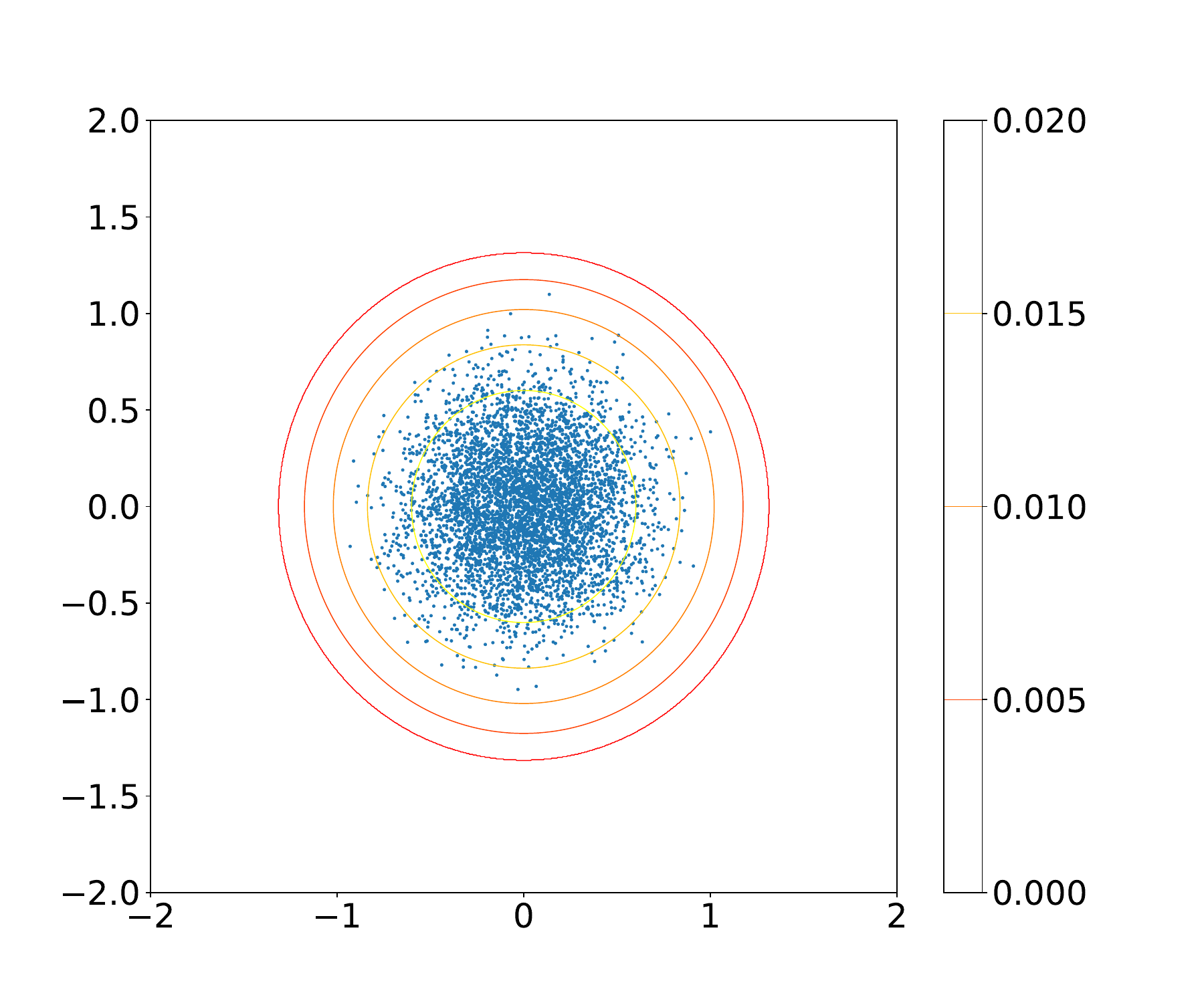}
        \caption{$t=1$}
    \end{subfigure}%
    %~
    \begin{subfigure}{0.33\textwidth}
        \centering
        \includegraphics[width=1\linewidth]{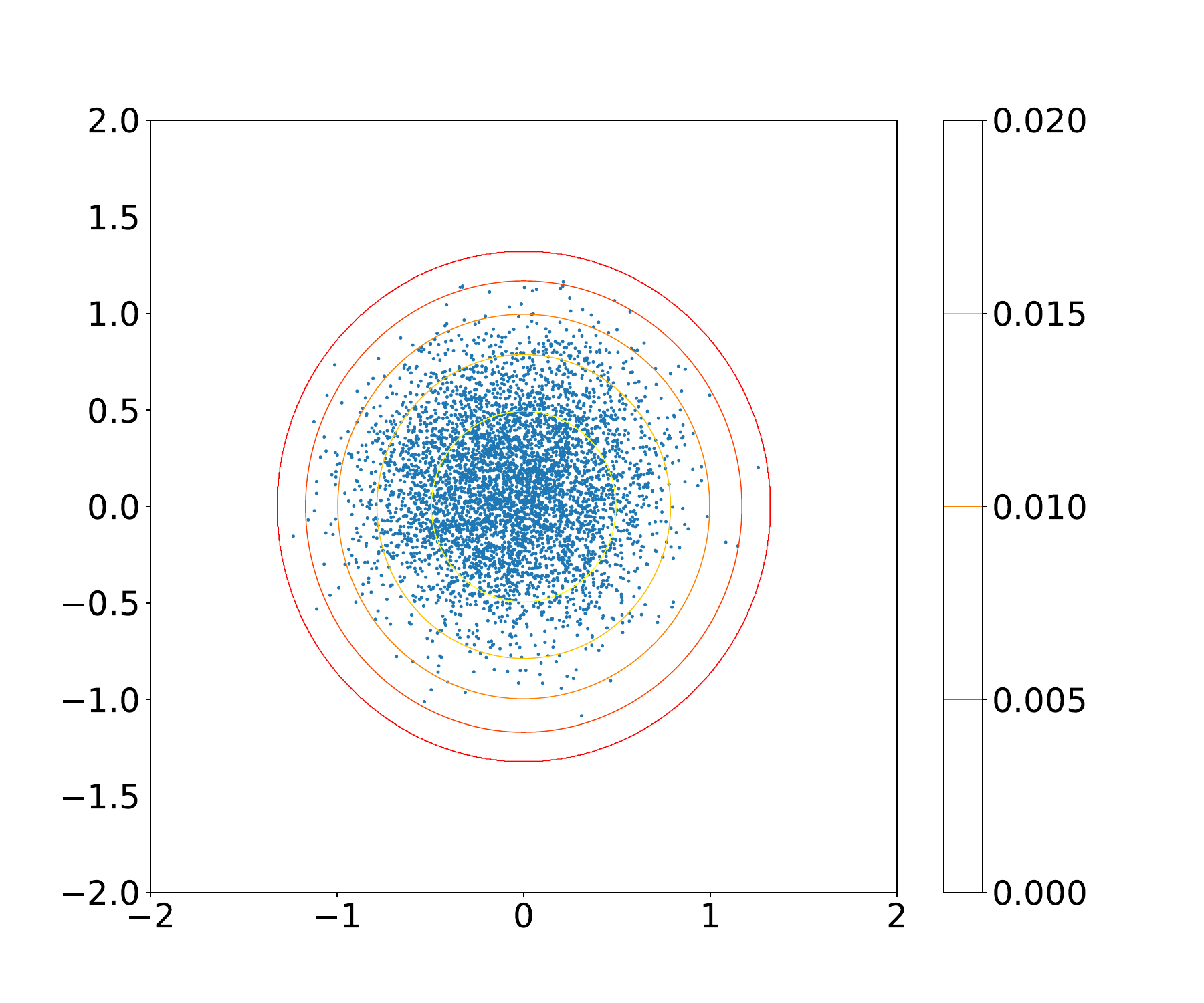}
        \caption{$t=1.1$}
    \end{subfigure}
    %~
    \begin{subfigure}{0.33\textwidth}
        \centering
        \includegraphics[width=1\linewidth]{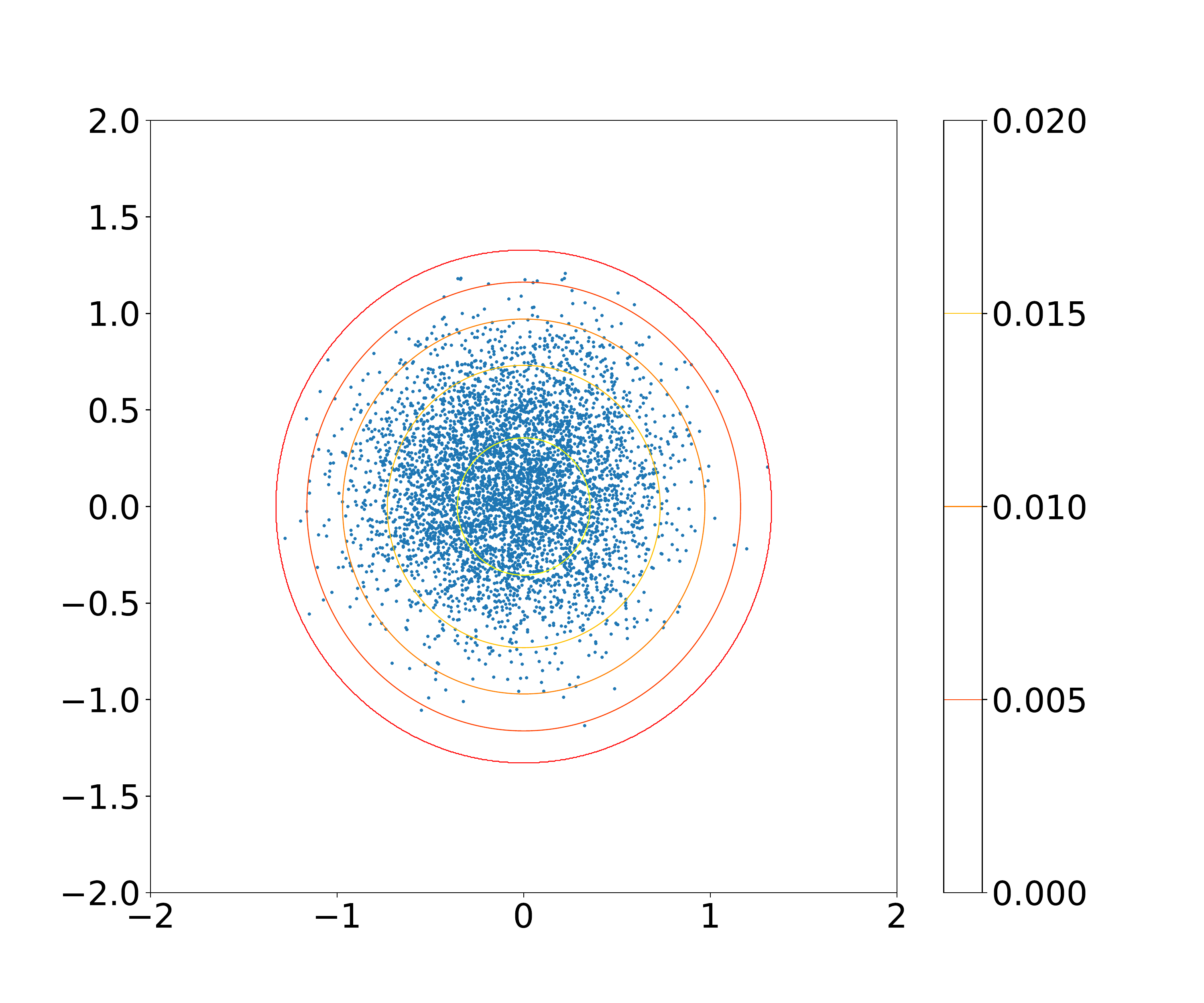}
        \caption{$t=1.2$}
    \end{subfigure}
    %~
    \begin{subfigure}{0.33\textwidth}
        \centering
        \includegraphics[width=1\linewidth]{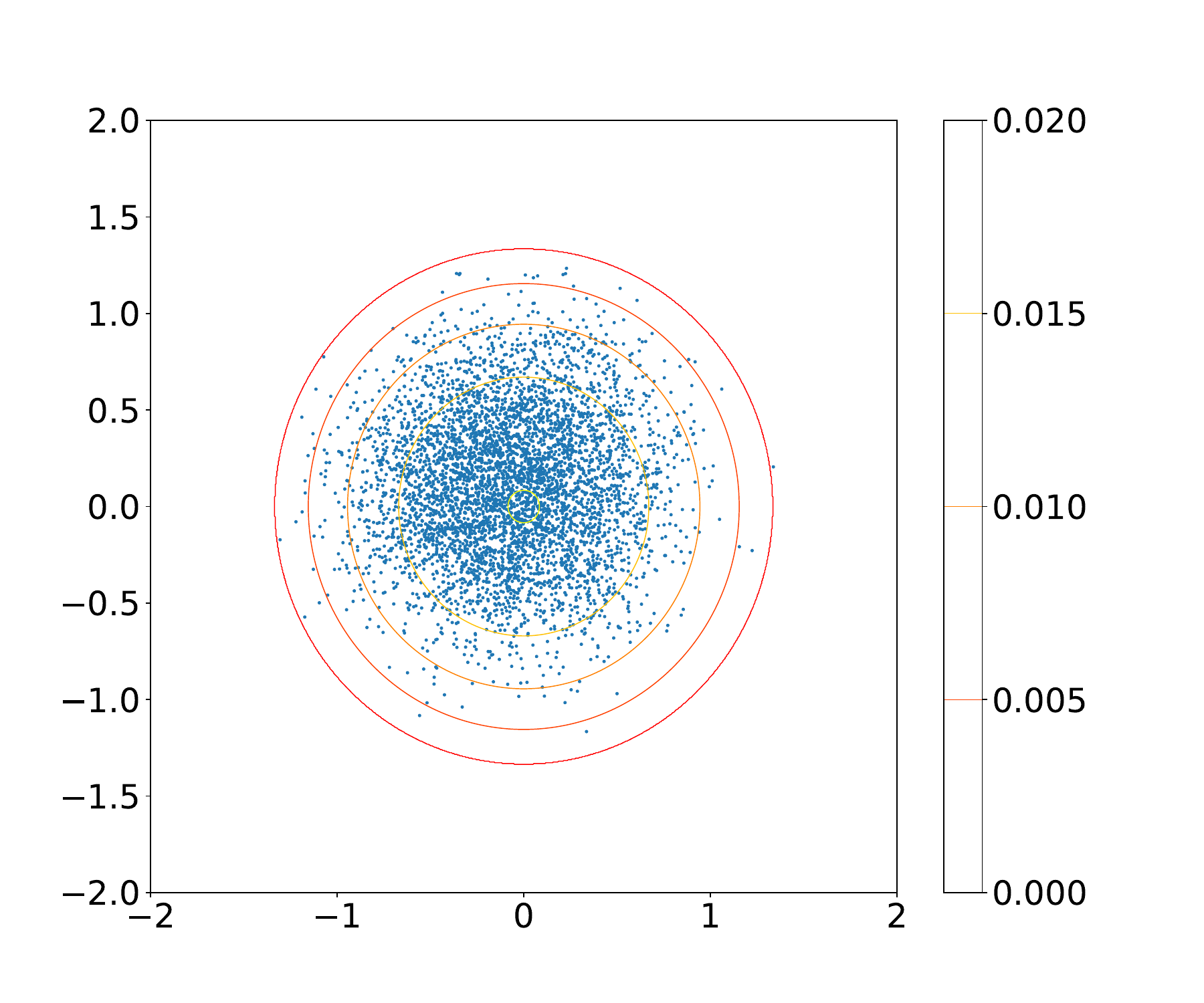}
        \caption{$t=1.3$}
    \end{subfigure}
    %~
    \begin{subfigure}{0.33\textwidth}
        \centering
        \includegraphics[width=1\linewidth]{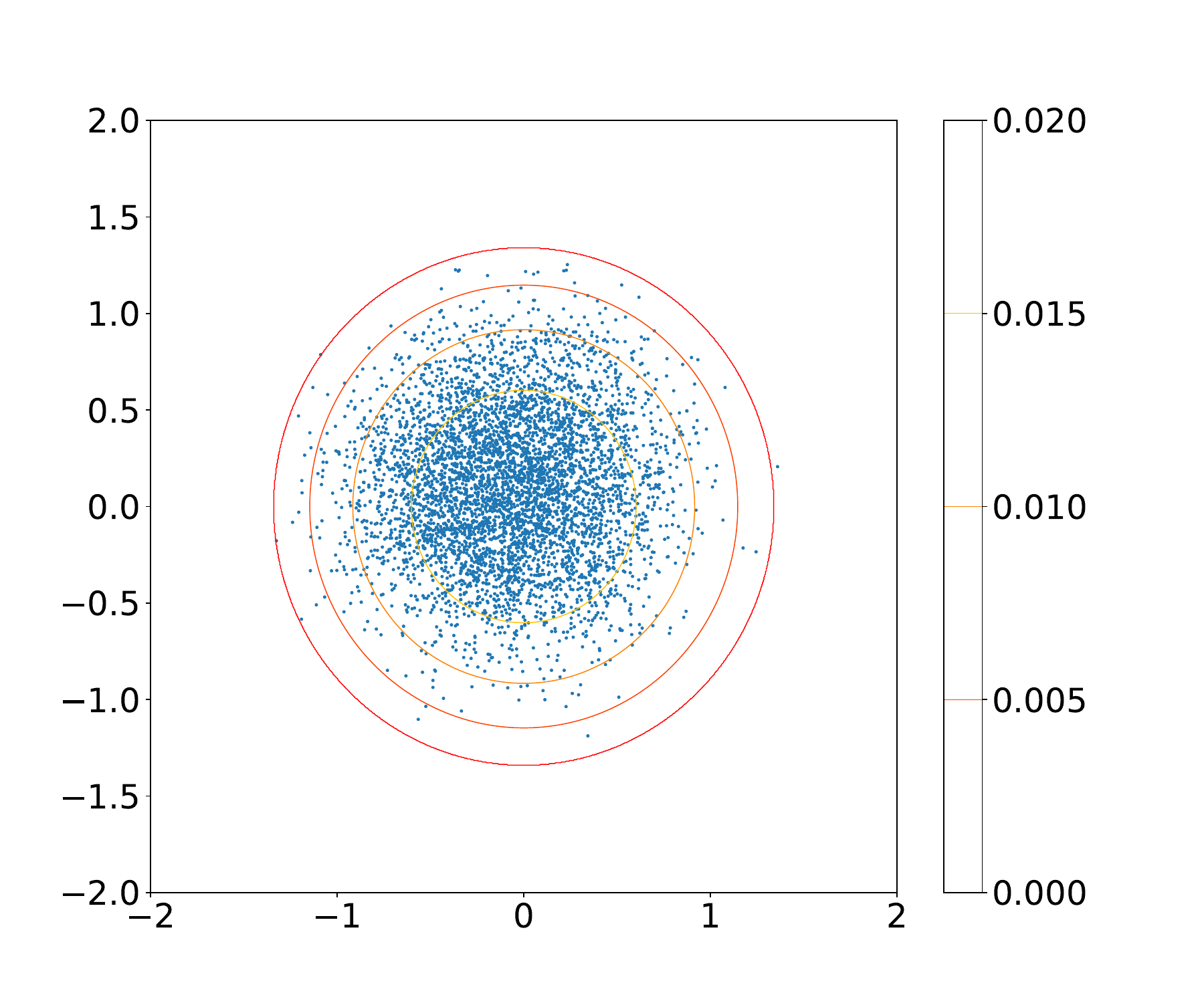}
        \caption{$t=1.4$}
    \end{subfigure}
    %~
    \begin{subfigure}{0.33\textwidth}
        \centering
        \includegraphics[width=1\linewidth]{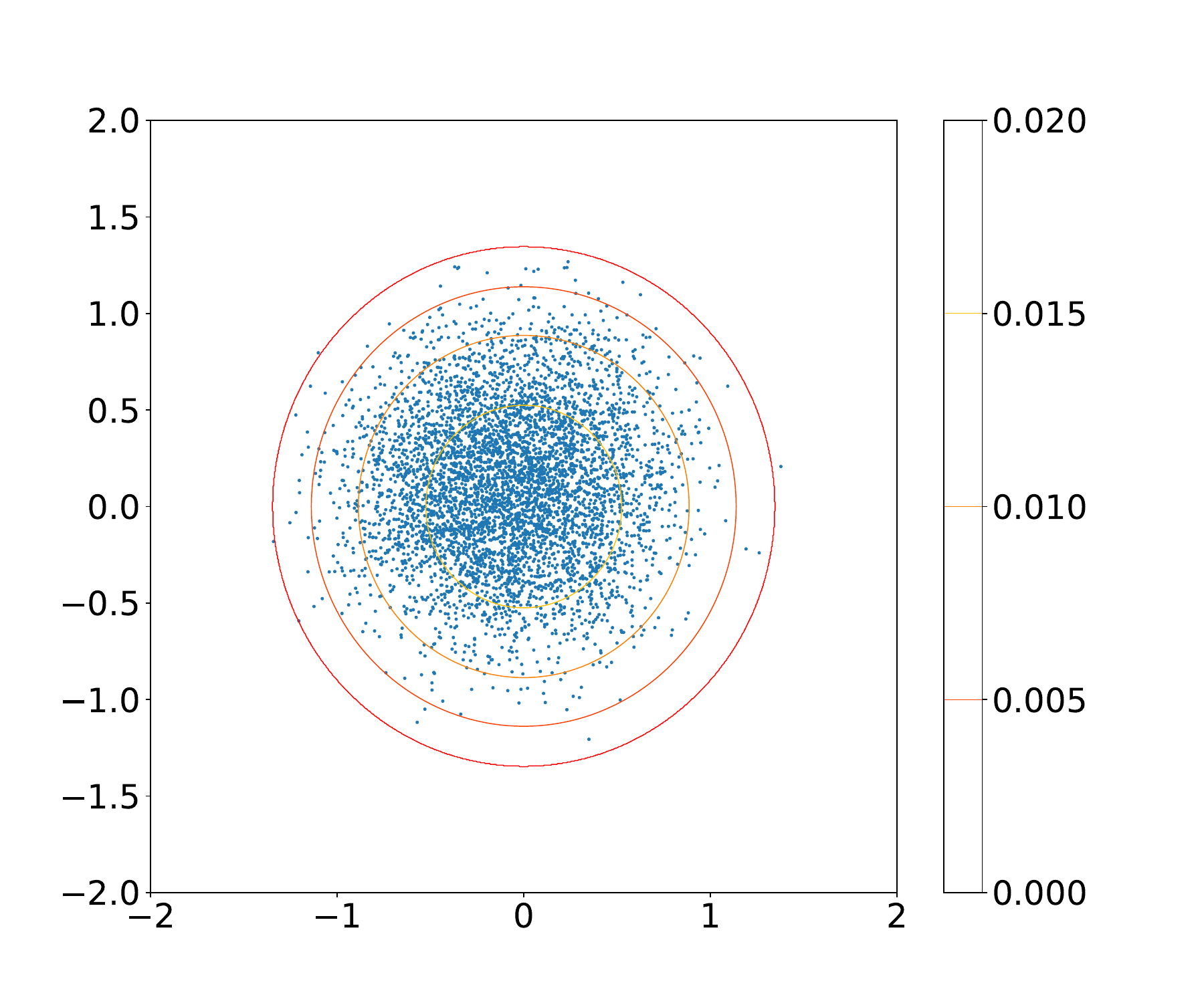}
        \caption{$t=1.5$}
    \end{subfigure}
    \caption{Sample plots of computed $\rho_{\theta}$ at different time $t$ for porous medium equation with Dirac Delta function as the initial condition for $d=15$. The figures are plotted with $5000$ samples. In the level curves, darker colors correspond to smaller values to emphasize the support. }
    \label{fig: PME_d=15}
\end{figure*}

\textit{An Example with Gaussian Mixture Distribution as the Initial Condition:} 
In this example, we apply Algorithm \ref{alg:GFsolver} to the porous medium equation with $d = 10$, $m = 3$ and a Gaussian mixture distribution as the initial condition. The equations are 
\begin{align}
    \frac{\partial \rho}{\partial t} & = \Delta (\rho^m) \\
    \rho(0,\cdot) & = 0.2 \cdot \mathcal{N}(\mu_1, \sigma_1) + 0.8 \cdot \mathcal{N}(\mu_2, \sigma_2)
\end{align}
where $\mu_1 = (0, 0, \cdots, 0)^T \in \mathbb{R}^{10}, \mu_2 (2, 2, \cdots, 2)^T \in \Rbb^{10}, \sigma_1 = \diag(0.1, \cdots, 0.1)^T$ and $\sigma_2 = \diag(0.2, \cdots, 0.2)^T$. 
% We use neural network with residual structure as the push-forward map in this example:
% \begin{align}
%     T_{\theta}=Id+R_{\theta}
% \end{align}
% where $R_{\theta}$ is a multi-layer perceptron with $3$ hidden layers and $100$ neurons in each hidden layer. 
The sample size to evaluate $\widehat{G}(\theta)$ and $F(\theta)$ are both $15,000$. The computed samples are shown in \autoref{fig: GM} plotted with $5000$ samples. The exact solution to this initial value problem remains open but our method can provide a simulated sample dynamics.

% Gaussian mixture
\begin{figure*}[t!]
    \begin{subfigure}{0.33\textwidth}
        \centering
        \includegraphics[width=0.99\linewidth]{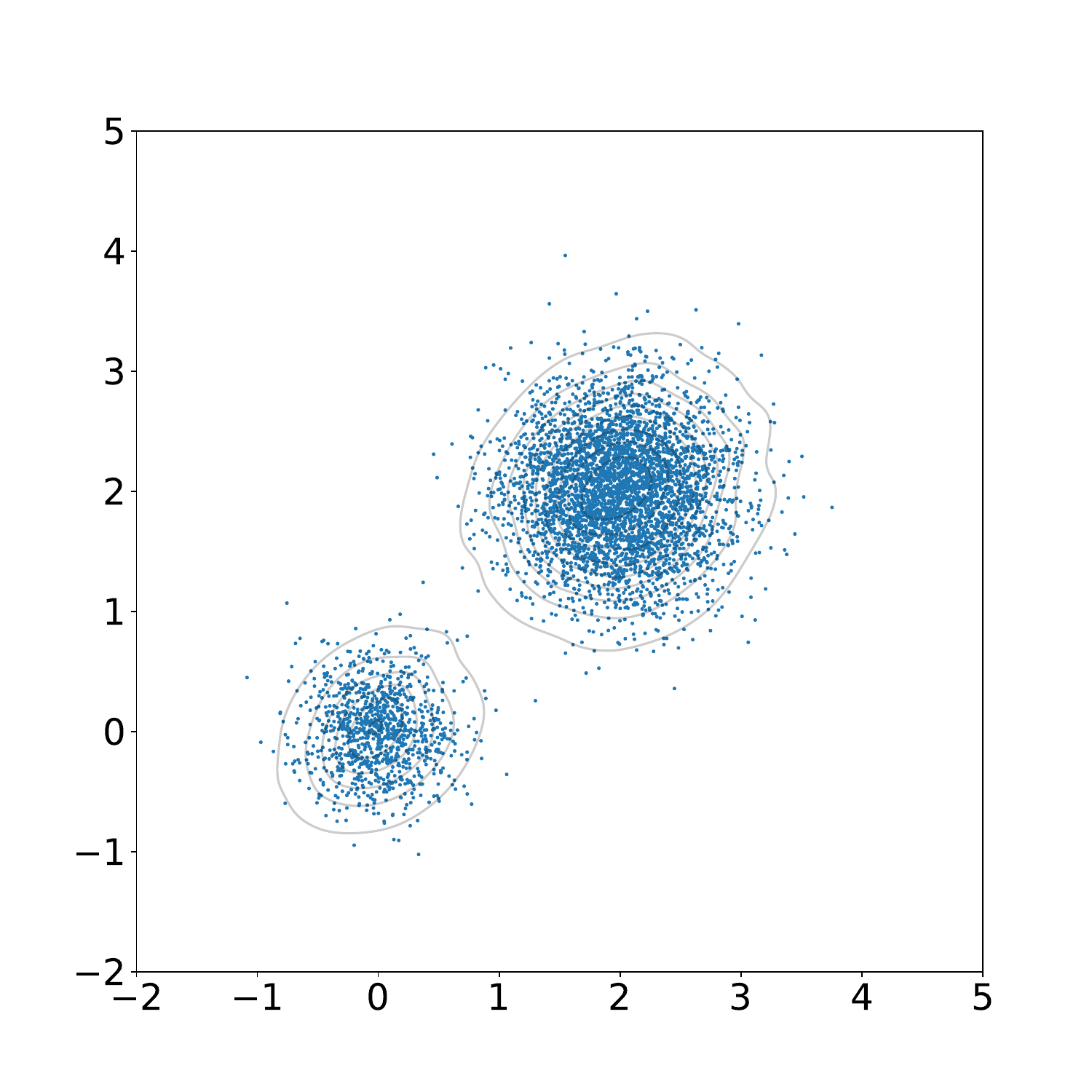}
        \caption{$t=0$}
    \end{subfigure}%
    %~
    \begin{subfigure}{0.33\textwidth}
        \centering
        \includegraphics[width=0.99\linewidth]{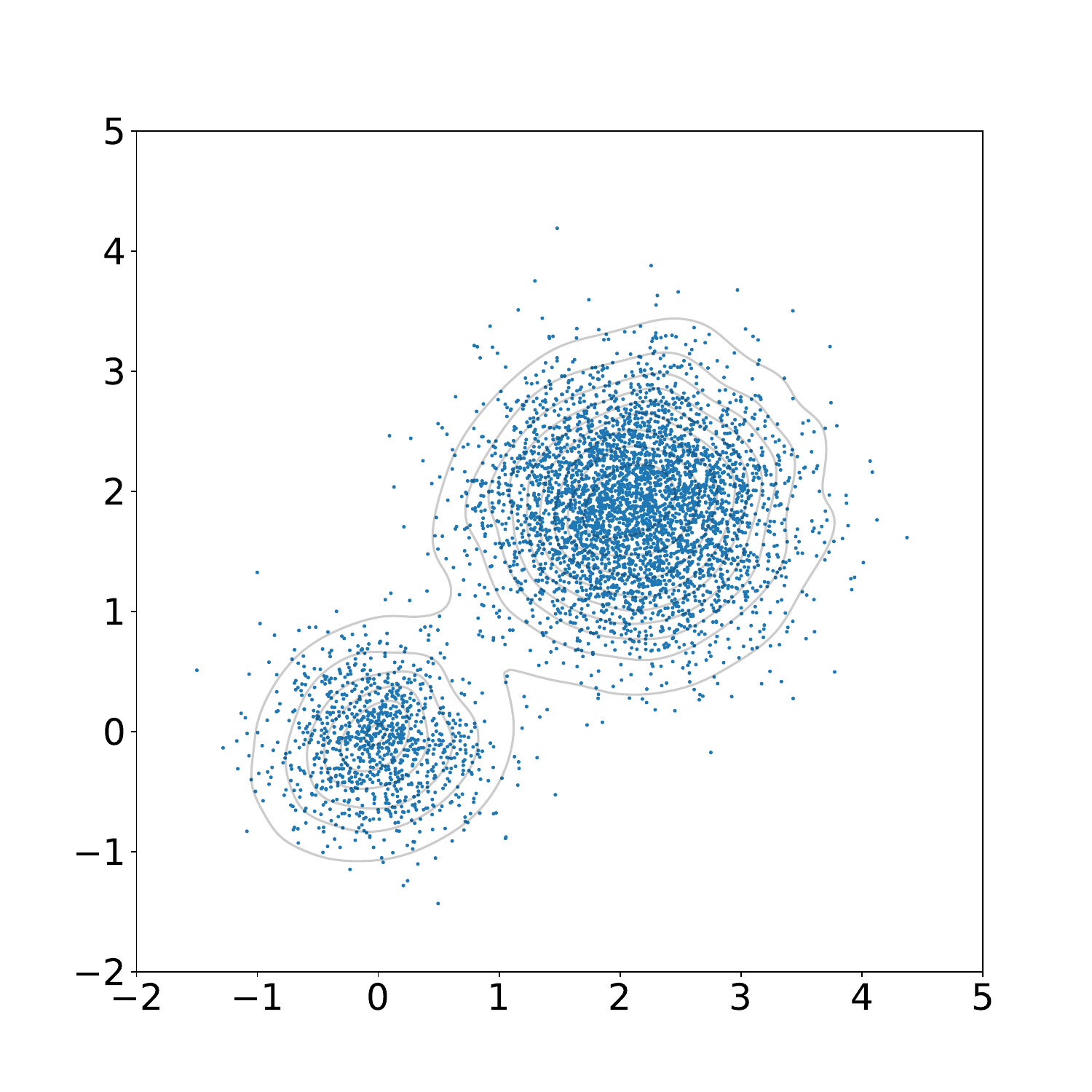}
        \caption{$t=2$}
    \end{subfigure}
    %~
    \begin{subfigure}{0.33\textwidth}
        \centering
        \includegraphics[width=0.99\linewidth]{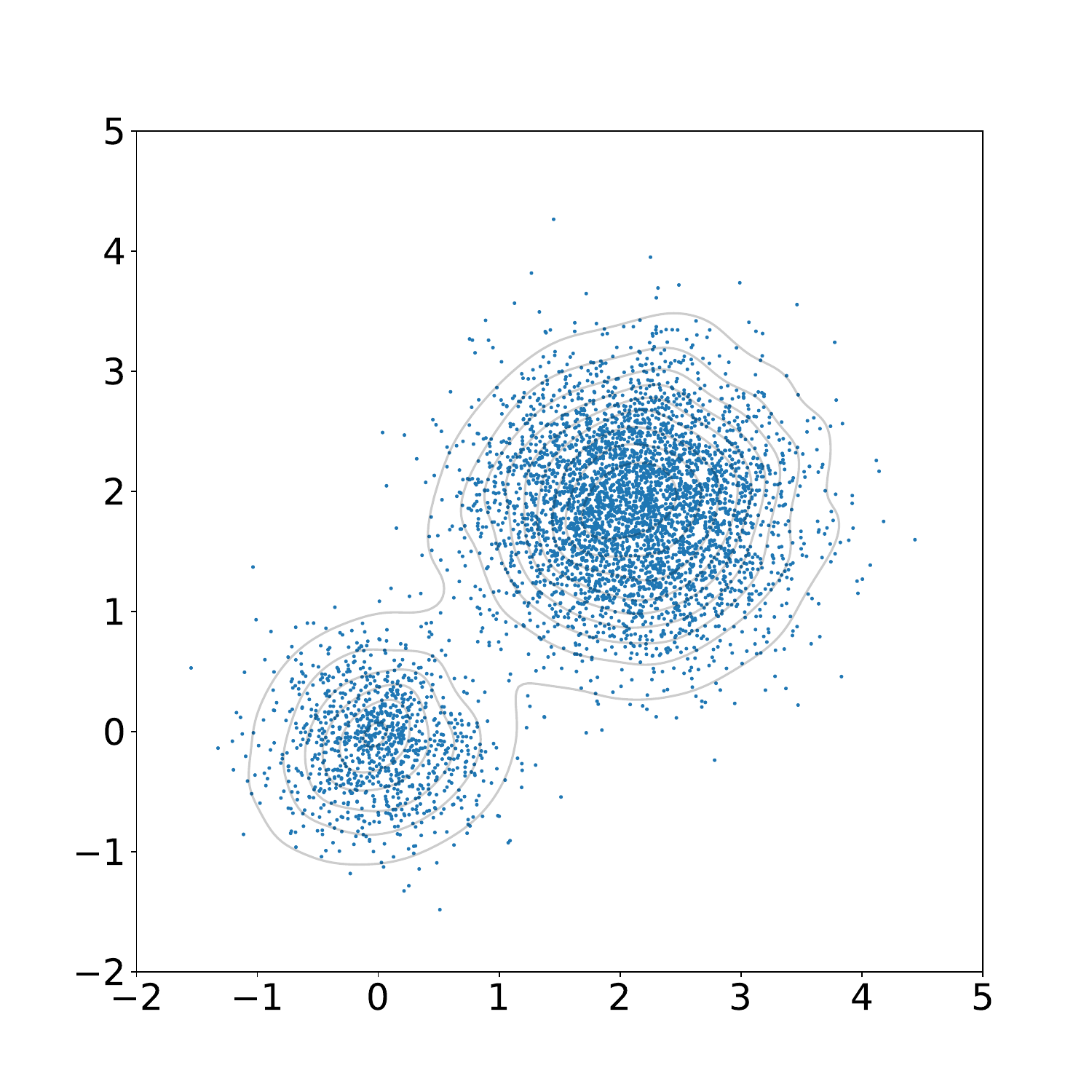}
        \caption{$t=4$}
    \end{subfigure}
    \caption{Sample plots of computed $\rho_{\theta}$ at different time $t$ for porous medium equation with Gaussian mixture distribution as the initial condition. The figures are plotted with $5000$ samples.}
    \label{fig: GM}
\end{figure*}

\textit{An Example with Mixed Gaussian Distribution and Uniform Distribution as the Initial Condition:}
In this example, we apply Algorithm \ref{alg:GFsolver} to the porous medium equation with $d = 5$, $m = 3$ and a mixture of Gaussian and uniform distributions as the initial condition. The equations are 
\begin{align}
    \frac{\partial \rho}{\partial t} & = \Delta (\rho^m) \\
    \rho(0,\cdot) & = 0.2 \cdot \mathcal{N}(\mu, \sigma) + 0.8 \cdot \mathcal{U}_5 ([-1,1]^5)
\end{align}
where $\mu = (2, 2, \cdots, 2)^T \in \Rbb^5$, $\sigma = (0.1, 0.1 , \cdots, 0.1)^T \in \Rbb^5$ and $\mathcal{U}_5 ([-1,1]^5)$ is a uniform distribution in $[-1,1]^5$. Note that this equation doesn't have a closed-form solution and it's not tractable in high dimensions. The sample size to evaluate $\widehat{G}(\theta)$ and $F(\theta)$ are both $20,000$. The computed samples are shown in \autoref{fig: GUM} plotted with $5000$ samples. 

\begin{figure*}[t!]
    \begin{subfigure}{0.33\textwidth}
        \centering
        \includegraphics[width=0.99\linewidth]{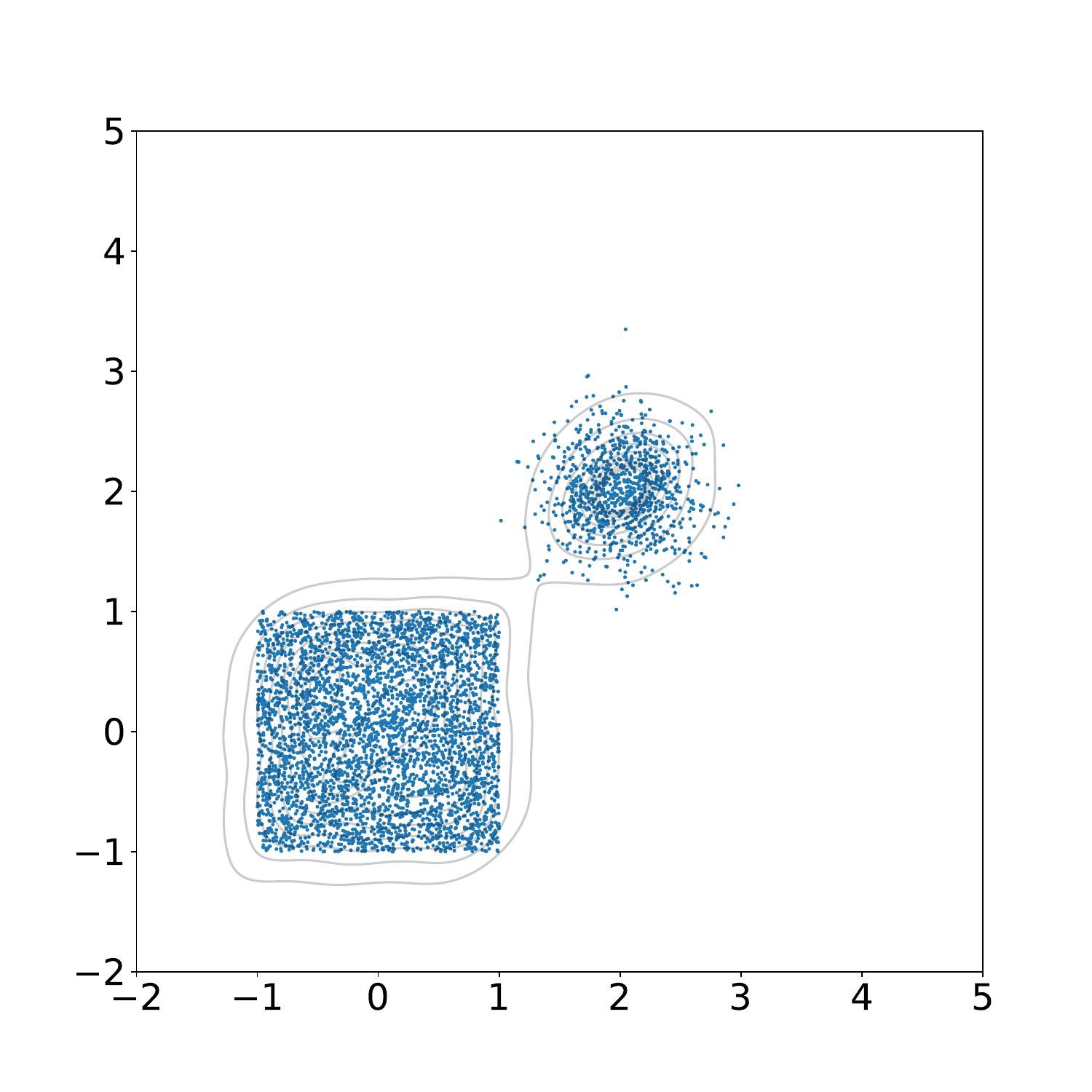}
        \caption{$t=0$}
    \end{subfigure}%
    %~
    \begin{subfigure}{0.33\textwidth}
        \centering
        \includegraphics[width=0.99\linewidth]{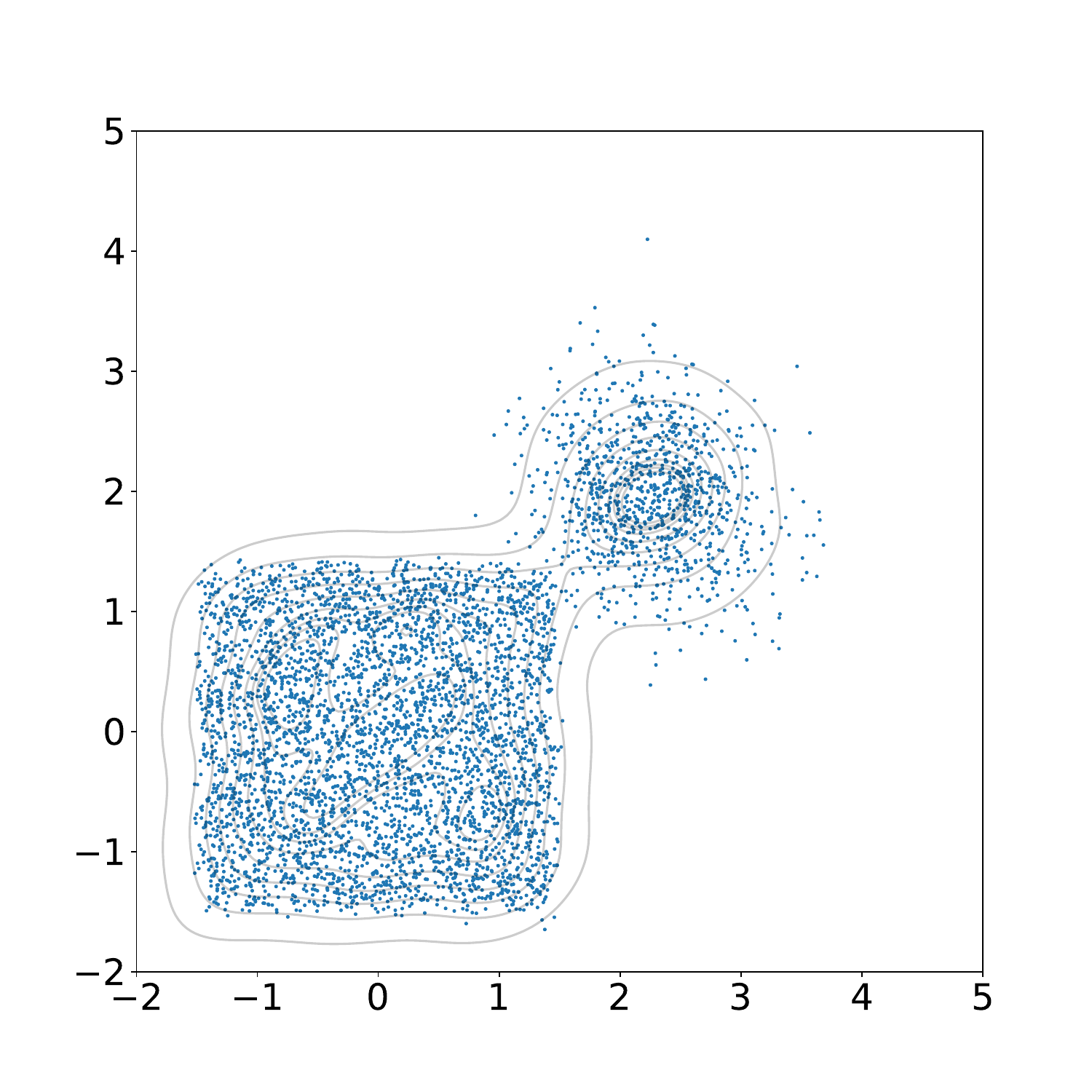}
        \caption{$t=2$}
    \end{subfigure}
    %~
    \begin{subfigure}{0.33\textwidth}
        \centering
        \includegraphics[width=0.99\linewidth]{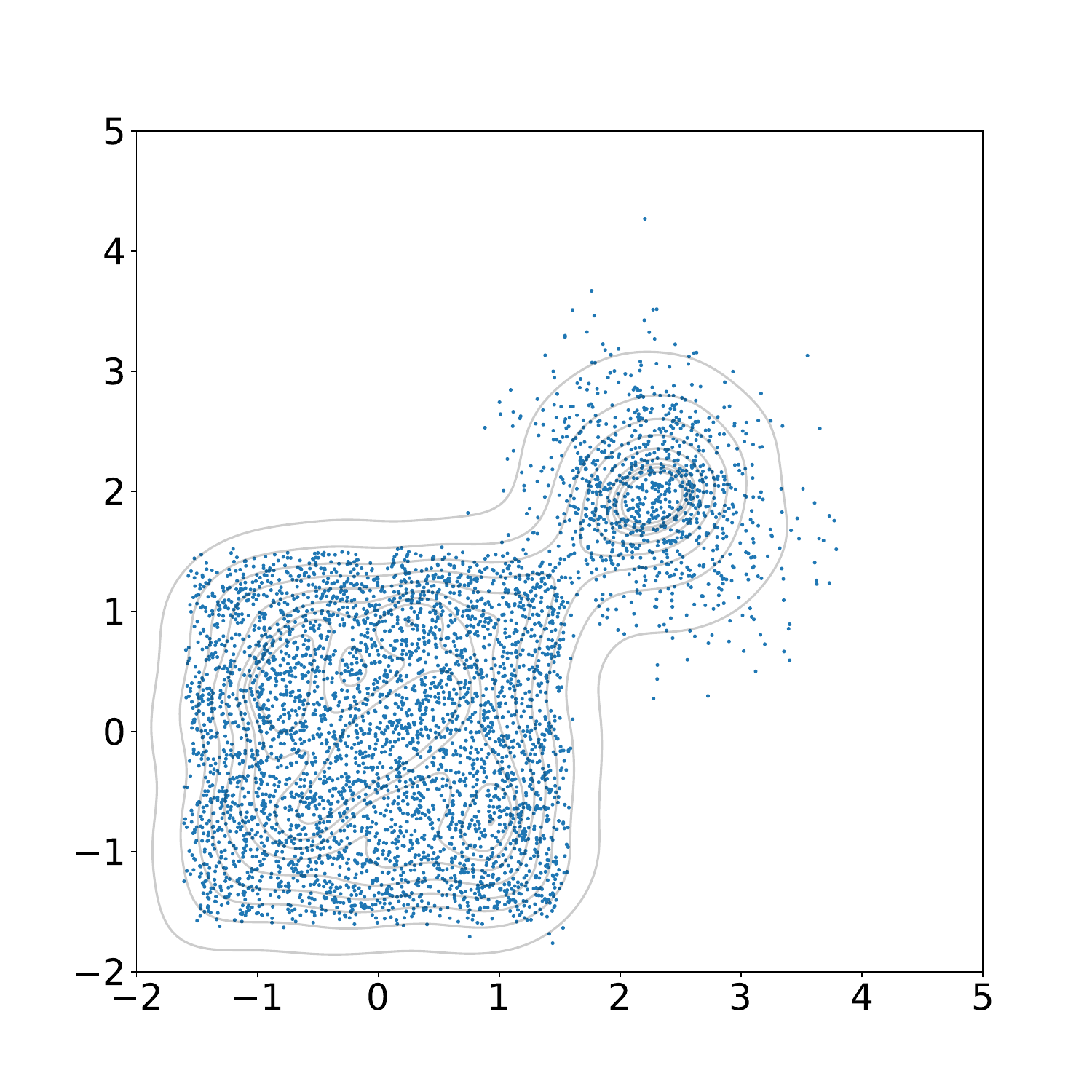}
        \caption{$t=4$}
    \end{subfigure}
    \caption{Sample plots of computed $\rho_{\theta}$ at different time $t$ for porous medium equation with the mixture of Gaussian distribution and uniform distribution as the initial condition. The graphs are plotted with $5000$ samples.}
    \label{fig: GUM}
\end{figure*}

\subsection{Aggregation model as a Wasserstein gradient flow}
\label{sec: aggregation numerical}
In this example, we apply Algorithm \ref{alg:GFsolver} to the aggregation problem with $d=2$ and $a=4$, $b=2$, i.e., the energy functional $\Fcal$ is set as in \eqref{eq:interact-F} with $J(x)$ given by
\begin{align}
    \label{num example: aggregation 2d}
    J(x)=\frac{|x|^4}{4}-\frac{|x|^2}{2}.
\end{align}
We set the initial condition to be Gaussian distribution with mean $\mu_0=(1.25, 1.25)^{\top}$ and variance $\gamma I$ where $\gamma=0.6$.
\begin{align}
    \label{num example: aggregation initial}
    \rho_0(x)=\frac{1}{\sqrt{2\pi}\gamma }e^{-\frac{|x-\mu_0|^2}{2\gamma^2}}
\end{align}
We use a neural network with residual structure as the push-forward map in this example:
\begin{align}
    T_{\theta}=Id+R_{\theta}
\end{align}
where $R_{\theta}$ is multi-layer perceptron with $2$ hidden layers and $50$ neurons in each hidden layer. The sample size to evaluate $\widehat{G}(\theta)$ and $F(\theta)$ are both $10,000$. 

The steady solution $\rho_*$ is a Dirac distribution uniformly concentrated on the ring with radius $0.5$ centered at $\mu_0$.  $\rho$ converges to $\rho_*$ when $t \rightarrow \infty $. This evolution is captured as shown in Figure \ref{aggregation sampleplot}, which shows the distribution $\rho_{\theta(t)}$ of samples from at different times $t$. Note that, despite of the strong singularity of $\rho_*$, the samples generated by Algorithm \ref{alg:GFsolver} can approximate this distribution well.

\begin{figure*}[t!]
    \begin{subfigure}{0.16\textwidth}
        \centering
        \includegraphics[width=0.99\linewidth]{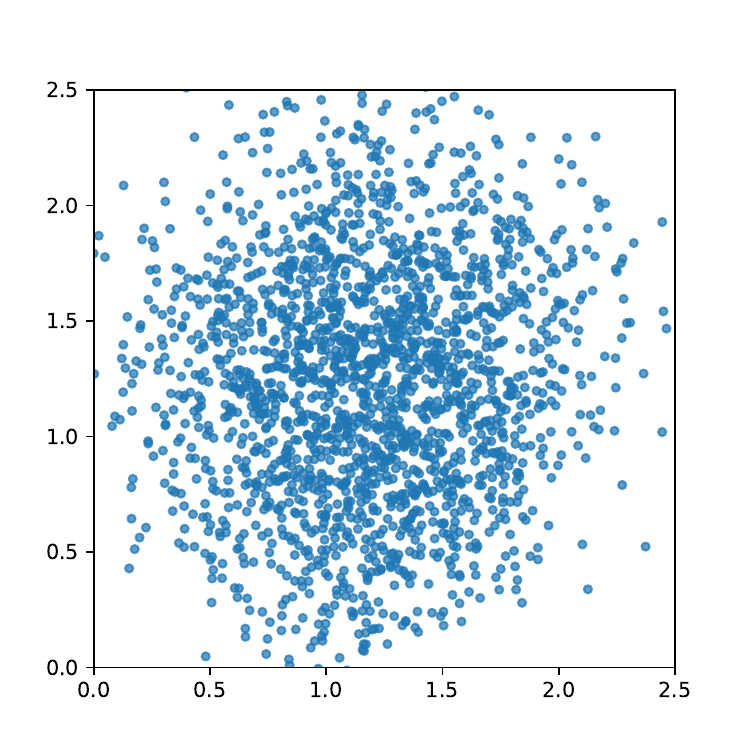}
        \caption{$t=0$}
    \end{subfigure}%
    %~
    \begin{subfigure}{0.16\textwidth}
        \centering
        \includegraphics[width=0.99\linewidth]{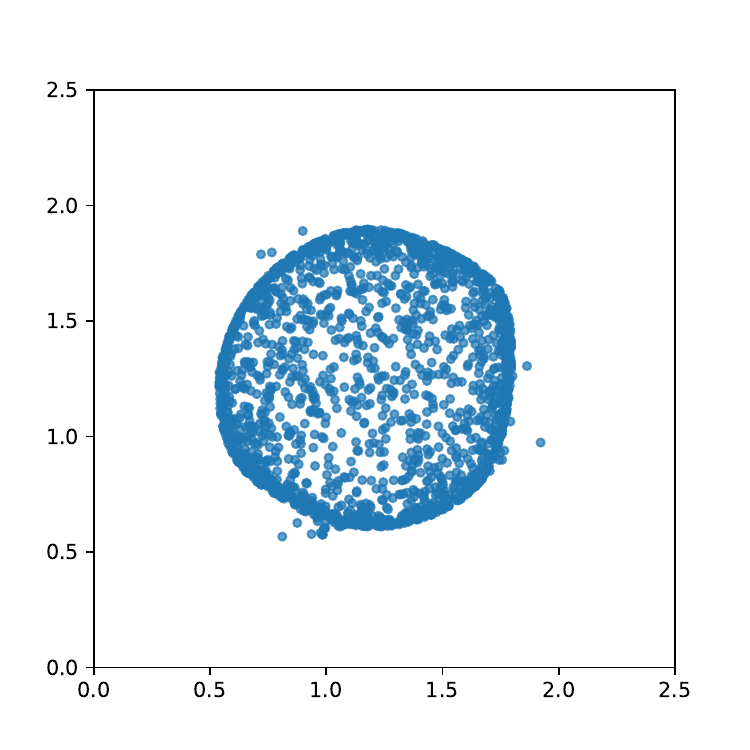}
        \caption{$t=3$}
    \end{subfigure}
    %~
    \begin{subfigure}{0.16\textwidth}
        \centering
        \includegraphics[width=0.99\linewidth]{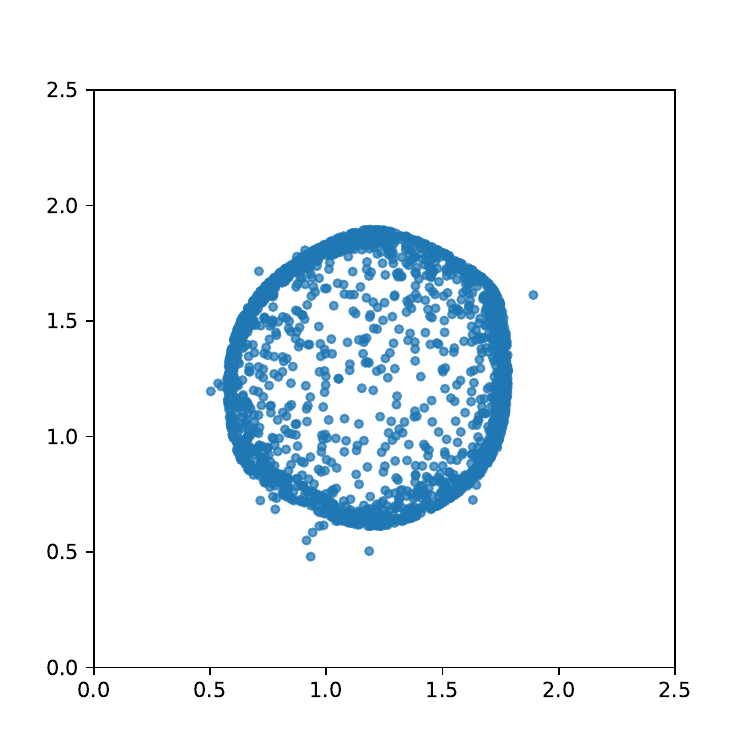}
        \caption{$t=6$}
    \end{subfigure}
    %\\
    \begin{subfigure}{0.16\textwidth}
        \centering
        \includegraphics[width=0.99\linewidth]{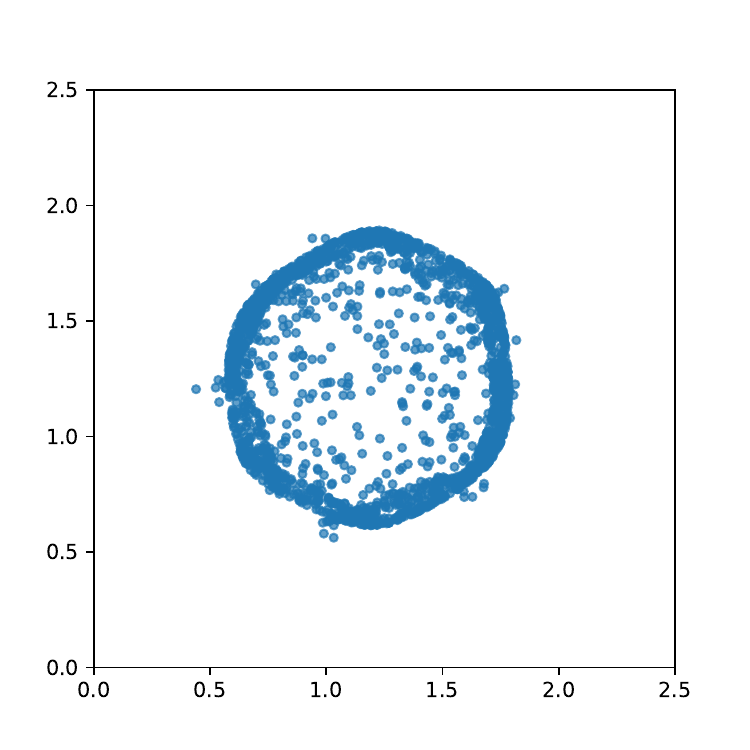}
        \caption{$t=9$}
    \end{subfigure}%
    %~
    \begin{subfigure}{0.16\textwidth}
        \centering
        \includegraphics[width=0.99\linewidth]{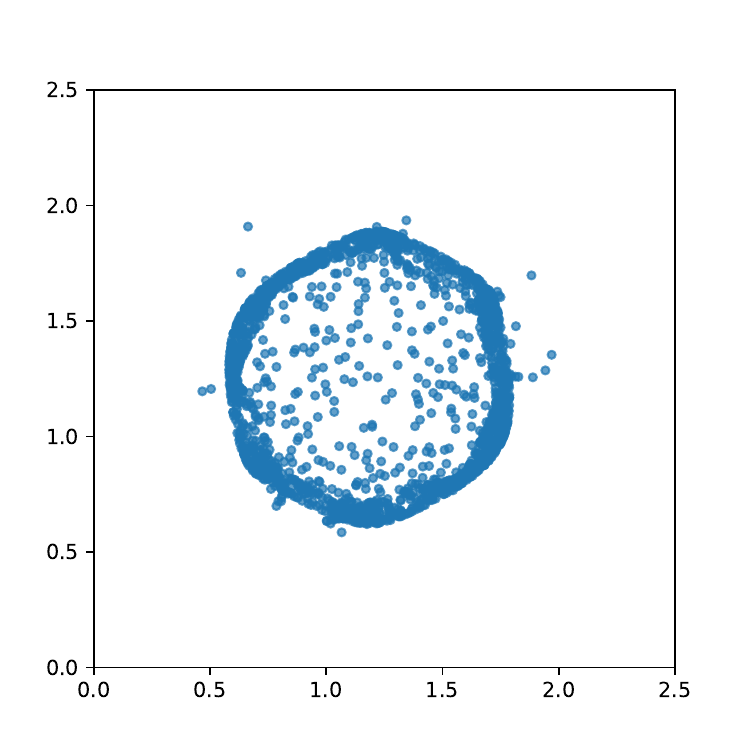}
        \caption{$t=12$}
    \end{subfigure}
    %~
    \begin{subfigure}{0.16\textwidth}
        \centering
        \includegraphics[width=0.99\linewidth]{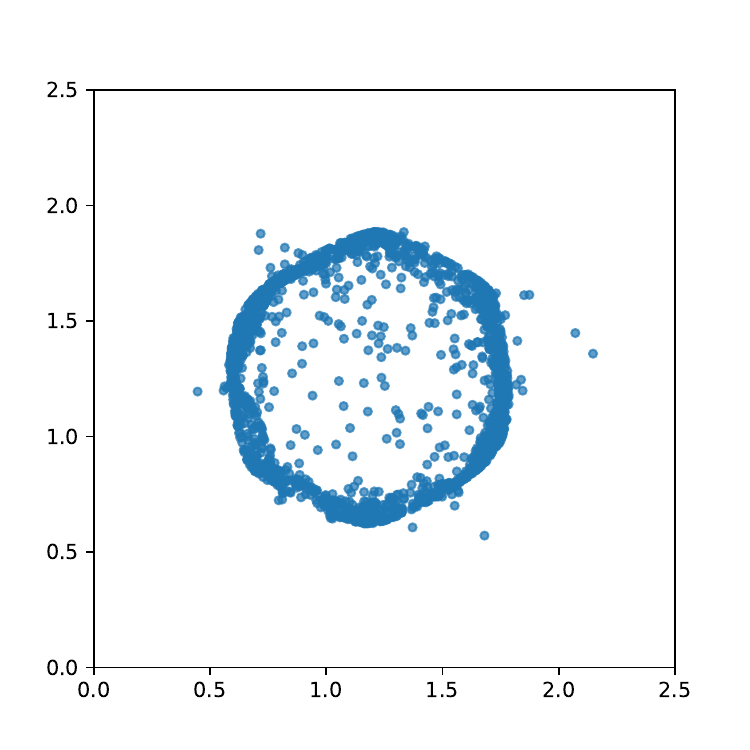}
        \caption{$t=15$}
    \end{subfigure}
    \caption{Sample plots of computed $\rho_{\theta}$ at different time $t$ for aggregation problem}
    \label{aggregation sampleplot}
\end{figure*}

\section{Discussion}
We proposed a new numerical approach to solve Wasserstein gradient flows (WGFs), which is particularly scalable for high-dimensional cases. Our approach used general reduced-order models, like deep neural networks, to parameterize the push-forward maps such that they can push a simple reference density to the density solving the given WGF. Essentially, the gradient flow defined on the infinite-dimensional Wasserstein manifold is reduced to a finite-dimensional dynamical system, the PWGF, for the parameters used in the reduced-order model. %The dynamical system, the PWGF, renders a system of coupled ODEs defined on the finite-dimensional parameter space. 
At the core of our design is the pullback Wasserstein metric on the parameter space. It facilitates the derivation, simplifies the numerical computation, and plays a pivotal role in the error-bound estimates that offer theoretical assurance for PWGF to the original WGF. %The dynamical system, called PWGF, renders a system of coupled ODEs defined on the finite-dimensional parameter space. 

Though the parametrization of the WGF provides a fast algorithm, the accuracy and speed of the algorithm largely rely on the linear system solver for the neural network parameter $\theta$. Since the metric matrix $\widehat{G}(\theta)$ is semi-positive definite and can be degenerate in some scenarios, the minimal residual method is utilized to find the least square solution. %which bypasses the original problem.

Our method is sample-based and provides the value of the density function at sampled points, it has the potential to adapt to tasks such as developing a fast sampler of target distribution by evolving samples from the initial probability density following the gradient flow or image generation following certain dynamics of probability densities. In these scenarios, as the task is more complex, the deep neural network for parametrizing the push-forward map could be large or have sophisticated architecture. One may consider dimension reduction techniques to reduce the dimension of the parameter space. In addition, other improvements such as more efficient linear system solvers or preconditioning techniques may also be considered to boost the accuracy and speed. Meanwhile, better understanding on $\delta_0$ in the error analysis needs to be analyzed regarding various neural network structures. 

% \Yijie{Make the last paragraph more logical. The $\widehat{G}$ matrix may be singular and the ODE system depends on it. The accuracy may be hindered by the minimal residual solver. 

% singular matrix -> minimal residual -> other linear solvers may improve the accuracy and deal with larger systems if the NN is big.} 

% \Yijie{
% \begin{enumerate}
%     \item Summarize the property of the framework. The advantages of the framework and elaborate on it. 
%     \item Future work. Applications. Some improvements. 
%     \begin{itemize}
%         \item Sampling.
%         \item MINRES.
%         \item $\delta_0$ and NN structure. 
%     \end{itemize}
% \end{enumerate}

% }

% \noindent Limitations:
% \begin{enumerate}
%     \item The accuracy depends on the sampling of the initial distribution. 
%     \item The density function is only available on the sampled points. 
%     \item The accuracy relies on the linear solver. If the task is more complex, the neural network is much larger and the linear solver may not solve the problem hard. 
% \end{enumerate}

\section*{Acknowledgments}
This research is partially supported by NSF grants DMS-1925263, DMS-2152960, DMS-2307465, DMS-2307466, and ONR grant N00014-21-1-2891. All the authors made equal contributions. 

\bibliographystyle{siamplain}
\bibliography{references}

\begin{thebibliography}{10}

\bibitem{anderson2022evolution}
{\sc W.~Anderson and M.~Farazmand}, {\em Evolution of nonlinear reduced-order
  solutions for {PDEs} with conserved quantities}, SIAM Journal on Scientific
  Computing, 44 (2022), pp.~A176--A197.

\bibitem{arjovsky2017wasserstein}
{\sc M.~Arjovsky, S.~Chintala, and L.~Bottou}, {\em Wasserstein generative
  adversarial networks}, in International conference on machine learning, PMLR,
  2017, pp.~214--223.

\bibitem{benamou2016discretization}
{\sc J.-D. Benamou, G.~Carlier, Q.~M{\'e}rigot, and E.~Oudet}, {\em
  Discretization of functionals involving the monge--amp{\`e}re operator},
  Numerische mathematik, 134 (2016), pp.~611--636.

\bibitem{berg2004markov}
{\sc B.~A. Berg}, {\em Markov chain Monte Carlo simulations and their
  statistical analysis: with web-based Fortran code}, World Scientific
  Publishing Company, 2004.

\bibitem{blanchet2013parabolic}
{\sc A.~Blanchet and P.~Lauren{\c{c}}ot}, {\em The parabolic-parabolic
  keller-segel system with critical diffusion as a gradient flow in
  $\mathbb{R}$ d, d $\geqslant$ 3}, Communications in Partial Differential
  Equations, 38 (2013), pp.~658--686.

\bibitem{brenner2008mathematical}
{\sc S.~C. Brenner}, {\em The mathematical theory of finite element methods},
  Springer, 2008.

\bibitem{bruna2024neural}
{\sc J.~Bruna, B.~Peherstorfer, and E.~Vanden-Eijnden}, {\em Neural {Galerkin}
  schemes with active learning for high-dimensional evolution equations},
  Journal of Computational Physics, 496 (2024), p.~112588.

\bibitem{carlier2017convergence}
{\sc G.~Carlier, V.~Duval, G.~Peyr{\'e}, and B.~Schmitzer}, {\em Convergence of
  entropic schemes for optimal transport and gradient flows}, SIAM Journal on
  Mathematical Analysis, 49 (2017), pp.~1385--1418.

\bibitem{carrillo2022primal}
{\sc J.~A. Carrillo, K.~Craig, L.~Wang, and C.~Wei}, {\em Primal dual methods
  for wasserstein gradient flows}, Foundations of Computational Mathematics,
  (2022), pp.~1--55.

\bibitem{chen2018neural}
{\sc R.~T. Chen, Y.~Rubanova, J.~Bettencourt, and D.~K. Duvenaud}, {\em Neural
  ordinary differential equations}, Advances in neural information processing
  systems, 31 (2018).

\bibitem{chewi2020svgd}
{\sc S.~Chewi, T.~Le~Gouic, C.~Lu, T.~Maunu, and P.~Rigollet}, {\em Svgd as a
  kernelized wasserstein gradient flow of the chi-squared divergence}, Advances
  in Neural Information Processing Systems, 33 (2020), pp.~2098--2109.

\bibitem{dinh2016density}
{\sc L.~Dinh, J.~Sohl-Dickstein, and S.~Bengio}, {\em Density estimation using
  real nvp}, arXiv preprint arXiv:1605.08803,  (2016).

\bibitem{du2021evolutional}
{\sc Y.~Du and T.~A. Zaki}, {\em Evolutional deep neural network}, Physical
  Review E, 104 (2021), p.~045303.

\bibitem{fan2022variational}
{\sc J.~Fan, Q.~Zhang, A.~Taghvaei, and Y.~Chen}, {\em Variational wasserstein
  gradient flow}, in International Conference on Machine Learning, PMLR, 2022,
  pp.~6185--6215.

\bibitem{fetecau2011swarm}
{\sc R.~C. Fetecau, Y.~Huang, and T.~Kolokolnikov}, {\em Swarm dynamics and
  equilibria for a nonlocal aggregation model}, Nonlinearity, 24 (2011),
  p.~2681.

\bibitem{gaby2024neural}
{\sc N.~Gaby, X.~Ye, and H.~Zhou}, {\em Neural control of parametric solutions
  for high-dimensional evolution pdes}, SIAM Journal on Scientific Computing,
  46 (2024), pp.~C155--C185.

\bibitem{he2016deep}
{\sc K.~He, X.~Zhang, S.~Ren, and J.~Sun}, {\em Deep residual learning for
  image recognition}, in Proceedings of the IEEE conference on computer vision
  and pattern recognition, 2016, pp.~770--778.

\bibitem{Holley1987LogarithmicSI}
{\sc R.~Holley and D.~W. Stroock}, {\em Logarithmic sobolev inequalities and
  stochastic ising models}, Journal of Statistical Physics, 46 (1987),
  pp.~1159--1194, \url{https://doi.org/10.1007/BF01011161}.

\bibitem{hu2022energetic}
{\sc Z.~Hu, C.~Liu, Y.~Wang, and Z.~Xu}, {\em Energetic variational neural
  network discretizations to gradient flows}, arXiv preprint arXiv:2206.07303,
  (2022).

\bibitem{JKO}
{\sc R.~Jordan, D.~Kinderlehrer, and F.~Otto}, {\em The variational formulation
  of the fokker--planck equation}, SIAM journal on mathematical analysis, 29
  (1998), pp.~1--17.

\bibitem{kobyzev2020normalizing}
{\sc I.~Kobyzev, S.~J. Prince, and M.~A. Brubaker}, {\em Normalizing flows: An
  introduction and review of current methods}, IEEE transactions on pattern
  analysis and machine intelligence, 43 (2020), pp.~3964--3979.

\bibitem{ladd2021reactive}
{\sc A.~J. Ladd and P.~Szymczak}, {\em Reactive flows in porous media:
  challenges in theoretical and numerical methods}, Annual review of chemical
  and biomolecular engineering, 12 (2021), pp.~543--571.

\bibitem{Lafferty}
{\sc J.~D. Lafferty}, {\em The {{Density Manifold}} and {{Configuration Space
  Quantization}}}, Transactions of the American Mathematical Society, 305
  (1988), pp.~699--741.

\bibitem{lee2023deep}
{\sc W.~Lee, L.~Wang, and W.~Li}, {\em Deep {JKO}: time-implicit particle
  methods for general nonlinear gradient flows}, arXiv preprint
  arXiv:2311.06700,  (2023).

\bibitem{leveque2007finite}
{\sc R.~J. LeVeque}, {\em Finite difference methods for ordinary and partial
  differential equations: steady-state and time-dependent problems}, SIAM,
  2007.

\bibitem{Li_2019}
{\sc W.~Li, S.~Liu, H.~Zha, and H.~Zhou}, {\em Parametric Fokker-Planck
  Equation}, Springer International Publishing, 2019, p.~715–724,
  \url{https://doi.org/10.1007/978-3-030-26980-7_74},
  \url{http://dx.doi.org/10.1007/978-3-030-26980-7_74}.

\bibitem{li2020fisher}
{\sc W.~Li, J.~Lu, and L.~Wang}, {\em Fisher information regularization schemes
  for wasserstein gradient flows}, Journal of Computational Physics, 416
  (2020), p.~109449.

\bibitem{li2023wasserstein}
{\sc W.~Li and J.~Zhao}, {\em Wasserstein information matrix}, Information
  Geometry,  (2023), pp.~1--53.

\bibitem{liu2022neural}
{\sc S.~Liu, W.~Li, H.~Zha, and H.~Zhou}, {\em Neural parametric fokker--planck
  equation}, SIAM Journal on Numerical Analysis, 60 (2022), pp.~1385--1449.

\bibitem{mokrov2021large}
{\sc P.~Mokrov, A.~Korotin, L.~Li, A.~Genevay, J.~M. Solomon, and E.~Burnaev},
  {\em Large-scale wasserstein gradient flows}, Advances in Neural Information
  Processing Systems, 34 (2021), pp.~15243--15256.

\bibitem{nurbekyan2023efficient}
{\sc L.~Nurbekyan, W.~Lei, and Y.~Yang}, {\em Efficient natural gradient
  descent methods for large-scale pde-based optimization problems}, SIAM
  Journal on Scientific Computing, 45 (2023), pp.~A1621--A1655.

\bibitem{otto-PME}
{\sc F.~Otto}, {\em The {{Geometry}} of {{Dissipative Evolution Equations}}:
  {{The Porous Medium Equation}}}, Communications in Partial Differential
  Equations, 26 (2001), pp.~101--174.

\bibitem{peyre2015entropic}
{\sc G.~Peyr{\'e}}, {\em Entropic approximation of wasserstein gradient flows},
  SIAM Journal on Imaging Sciences, 8 (2015), pp.~2323--2351.

\bibitem{rezende2015variational}
{\sc D.~Rezende and S.~Mohamed}, {\em Variational inference with normalizing
  flows}, in International conference on machine learning, PMLR, 2015,
  pp.~1530--1538.

\bibitem{saad2003iterative}
{\sc Y.~Saad}, {\em Iterative methods for sparse linear systems}, SIAM, 2003.

\bibitem{vazquez2007porous}
{\sc J.~L. V{\'a}zquez}, {\em The porous medium equation: mathematical theory},
  Oxford University Press, 2007.

\bibitem{wang2023neural}
{\sc M.~Wang and J.~Lu}, {\em Neural network-based variational methods for
  solving quadratic porous medium equations in high dimensions}, Communications
  in Mathematics and Statistics, 11 (2023), pp.~21--57.

\bibitem{wu_theory_2023}
{\sc H.~Wu}, {\em Theory and computation of {Wasserstein} geometric flows with
  application to time-dependent {Schrodinger} equation},  (2023).
\newblock Publisher: Georgia Institute of Technology.

\bibitem{wu2023parameterized}
{\sc H.~Wu, S.~Liu, X.~Ye, and H.~Zhou}, {\em Parameterized wasserstein
  hamiltonian flow}, arXiv preprint arXiv:2306.00191,  (2023).

\bibitem{yi2023bridging}
{\sc M.~Yi and S.~Liu}, {\em Bridging the gap between variational inference and
  wasserstein gradient flows}, 2023, \url{https://arxiv.org/abs/2310.20090}.

\bibitem{zuo2024numerical}
{\sc X.~Zuo, J.~Zhao, S.~Liu, S.~Osher, and W.~Li}, {\em Numerical analysis on
  neural network projected schemes for approximating one dimensional
  wasserstein gradient flows}, 2024, \url{https://arxiv.org/abs/2402.16821}.

\end{thebibliography}


\begin{thebibliography}{}

\end{thebibliography}
\end{document}

\maketitle

\section{A detailed example}

Here we include some equations and theorem-like environments to show
how these are labeled in a supplement and can be referenced from the
main text.
Consider the following equation:
\begin{equation}
  \label{eq:suppa}
  a^2 + b^2 = c^2.
\end{equation}
You can also reference equations such as \cref{eq:matrices,eq:bb} 
from the main article in this supplement.

\lipsum[100-101]

\begin{theorem}
  An example theorem.
\end{theorem}

\lipsum[102]
 
\begin{lemma}
  An example lemma.
\end{lemma}

\lipsum[103-105]

Here is an example citation: \cite{KoMa14}.

\section[Proof of Thm]{Proof of \cref{thm:bigthm}}
\label{sec:proof}
\lipsum[106-112]

\section{Additional experimental results}
\Cref{tab:foo} shows additional
supporting evidence. 

\begin{table}[htbp]
{\footnotesize
  \caption{Example table}  \label{tab:foo}
\begin{center}
  \begin{tabular}{|c|c|c|} \hline
   Species & \bf Mean & \bf Std.~Dev. \\ \hline
    1 & 3.4 & 1.2 \\
    2 & 5.4 & 0.6 \\ \hline
  \end{tabular}
\end{center}
}
\end{table}

\bibliographystyle{siamplain}
\bibliography{references}